\documentclass[11pt,leqno]{amsart}
\usepackage[a4paper,margin=27mm,headheight=14pt]{geometry}
\usepackage[T1]{fontenc}
\usepackage{lmodern,amssymb,mathrsfs}
\usepackage{mathtools,amsthm,amscd,booktabs,array,longtable,microtype}
\usepackage{enumitem}
\usepackage[hidelinks,unicode]{hyperref}
\usepackage{bookmark}
\setlist[enumerate]{label=\textup{(\alph*)},leftmargin=*,itemsep=3pt,topsep=5pt}
\newtheorem{theorem}{Theorem}[section]
\newtheorem{proposition}[theorem]{Proposition}
\newtheorem{lemma}[theorem]{Lemma}
\newtheorem{corollary}[theorem]{Corollary}
\theoremstyle{definition}
\newtheorem{definition}[theorem]{Definition}
\newtheorem{example}[theorem]{Example}
\theoremstyle{remark}

\newcommand{\kk}{\mathbb k}
\newcommand{\QQ}{\mathbb Q}
\newcommand{\RR}{\mathbb R}
\newcommand{\CC}{\mathbb C}

\renewcommand{\AA}{\mathbb A}
\newcommand{\PP}{\mathbb P}
\newcommand{\Gm}{\mathbb G_m}
\newcommand{\cA}{\mathcal A}
\newcommand{\cD}{\mathcal D}
\newcommand{\cI}{\mathcal I}
\newcommand{\cL}{\mathcal L}
\newcommand{\cN}{\mathcal N}
\newcommand{\cR}{\mathcal R}
\newcommand{\cW}{\mathcal W}
\newcommand{\cP}{\mathcal P}
\newcommand{\sF}{\mathscr F}
\newcommand{\sC}{\mathscr C}
\newcommand{\sB}{\mathscr B}
\newcommand{\sT}{\mathscr T}
\newcommand{\bA}{\mathbf A}
\newcommand{\id}{\mathrm{id}}
\newcommand{\one}{\mathbf 1}
\newcommand{\eps}{\varepsilon}
\newcommand{\Istate}{\mathtt I}
\newcommand{\Ostate}{\mathtt O}
\DeclareMathOperator{\coker}{coker}
\DeclareMathOperator{\im}{im}
\DeclareMathOperator{\rank}{rank}
\DeclareMathOperator{\supp}{supp}
\DeclareMathOperator{\Spec}{Spec}
\DeclareMathOperator{\codim}{codim}
\DeclareMathOperator{\diag}{diag}
\DeclareMathOperator{\Tor}{Tor}
\DeclareMathOperator{\Fitt}{Fitt}
\DeclareMathOperator{\Hom}{Hom}
\DeclareMathOperator{\Cone}{Cone}
\DeclareMathOperator{\Gr}{Gr}
\DeclareMathOperator{\cl}{cl}
\DeclareMathOperator{\spanlin}{span}

\DeclareMathOperator{\vecop}{vec}
\newcommand{\arxiv}[1]{\href{https://arxiv.org/abs/#1}{arXiv:#1}}
\newcommand{\stacks}[1]{\href{https://stacks.math.columbia.edu/tag/#1}{Tag #1}}
\allowdisplaybreaks[2]
\title[Zero forcing and support varieties]{Zero forcing, support varieties,\ and boundary complexes}
\newcommand{\HanAffiliation}{School of Mathematics and Statistics, Lanzhou University, Lanzhou, Gansu 730000, China}
\newcommand{\LiangAffiliation}{School of Mathematical Sciences, Nankai University, Tianjin, China}
\newcommand{\HanEmail}{dblnhzj@163.com}
\newcommand{\LiangEmail}{cldlllyr@163.com}
\author[Z. Han]{Zongjian Han}
\address[Zongjian Han]{\HanAffiliation}
\email[Zongjian Han]{\HanEmail}
\author[Y. Liang]{Yiran Liang}
\address[Yiran Liang]{\LiangAffiliation}
\email[Yiran Liang]{\LiangEmail}
\thanks{\raggedright\textsc{Zongjian Han.} \HanAffiliation.\enspace
  \textit{E-mail:} \texttt{\HanEmail}}
\thanks{\raggedright\textsc{Yiran Liang.} \LiangAffiliation.\enspace
  \textit{E-mail:} \texttt{\LiangEmail}}
\makeatletter
\let\enddoc@text\@empty
\makeatother
\date{}
\hypersetup{pdftitle={Zero forcing, support varieties, and boundary complexes},pdfsubject={Cellular sheaves, kernel supports, and boundary gluing for arbitrary finite simple graphs},pdfauthor={Zongjian Han and Yiran Liang}}
\begin{document}
\begin{abstract}
Let $G$ be a finite simple graph and let $F$ be a nonempty vertex set.
We study symmetric matrices of pattern $G$ with a null vector of exact
support $F$. Such matrices exist precisely when $F$ is a fort, that is,
no exterior vertex has exactly one neighbor in $F$.
Over $\CC$, the closure $\mathcal D_F$ of their coefficient locus is
irreducible and has codimension
$|\partial F|+r(F)-\rho_F$. Here $\partial F=N(F)\setminus F$,
$r(F)$ counts the components of $G[F]$, and $\rho_F$ is the largest size
of a forest joining these components with distinct exterior-vertex labels.
The difference $r(F)-\rho_F$ is also the generic restricted nullity.
An incidence parametrization gives a saturated prime decomposition of the
reduced zero-forcing failure locus. For a prescribed common kernel,
we identify the coefficient space with a direct sum of equilibrium
stresses, exterior affine dependencies, and free coefficients.
We then construct boundary complexes which retain local values and row
residuals. Their gluing maps are defined over the coefficient ring and
commute with arbitrary base change. Collecting all support signatures
of each seed set gives a composition formula for the minimum seed cost
and an algorithm for zero forcing.
\end{abstract}
\maketitle
\setcounter{tocdepth}{1}
\tableofcontents
\clearpage
\section{Introduction}
Let $G=(V,E)$ be a finite simple graph. A symmetric matrix of pattern $G$
has arbitrary diagonal entries, nonzero entries at the edges, and zero
entries at the other off-diagonal positions. We study the coefficients
for which such a matrix has a null vector of prescribed support. We also
study how the resulting equations combine under a decomposition of $G$.

Zero forcing supplies a local test for excluding those failures. Starting with
black vertices $S$, a black vertex $u$ may force a white vertex $v$ if $v$ is its
only white neighbour. If $Ax=0$ and the coordinates already known to vanish
are black, this step is precisely
\[
 0=(Ax)_u=a_{uv}x_v,\qquad a_{uv}\ne0,
 \qquad\Longrightarrow\qquad x_v=0.
\]
A zero-forcing set therefore determines every null vector for every choice
of the allowed coefficients. Its minimum cardinality is $Z(G)$.
When a set fails to force, the possible supports divide the exceptional
coefficients into algebraic loci. We determine their dimensions and prime
ideals, and express their gluing in terms of boundary values and residuals.

\subsection{The geometry of a supported null vector}
Fix a characteristic-zero field $\kk$ and define
\begin{equation}\label{eq:intro-ring}
 R_G=\kk[d_v\ (v\in V),\ t_e^{\pm1}\ (e\in E)].
\end{equation}
Let $\bA$ be the universal symmetric matrix with diagonal $d_v$ and edge entry
$t_{\{u,v\}}$. Over $\CC$ its coefficient space is
$\Omega_G=\mathbb A^{|V|}\times\mathbb G_m^{|E|}$.
A nonempty set $F\subseteq V$ is a \emph{fort} if
\[
 |N(u)\cap F|\ne1\qquad(u\notin F).
\]
These are exactly the possible supports of nonzero null vectors as the
coefficients vary; the support results of \cite{Spark,Compatible} provide
this starting point. Proposition~\ref{prop:support-witness} gives explicit
integer symmetric witnesses.

For a fort $F$, let $\mathcal D_F$ be the Zariski closure in $\Omega_G$ of the
coefficients admitting a null vector of \emph{exact} support $F$. Write
$\partial F=N(F)\setminus F$, and let $r(F)$ be the number of components of
$G[F]$. An exterior vertex $u\in\partial F$ labels any pair of support
components that it meets. Let $\rho_F$ be the largest number of such labeled
pairs which form a forest and use each exterior label at most once. Our first
main result is
\begin{equation}\label{eq:intro-codim}
 \operatorname{codim}_{\Omega_G}\mathcal D_F
 =|\partial F|+h_F,\qquad h_F=r(F)-\rho_F.
\end{equation}
The variety $\mathcal D_F$ is irreducible, and $h_F$ is the generic dimension
of the nullspace restricted to $F$ (Theorem~\ref{thm:fort-codimension}). The three-leaf example after Definition~\ref{def:labeled-forest}
illustrates the restriction on repeated exterior labels.

There are two parts to the proof. The exact-support incidence variety has an
explicit product parametrization by a torus, an affine space, and complements
of coordinate hyperplanes in zero-sum hyperplanes. This computes its dimension.
The fibers of its coefficient projection are open subsets of projectivized
restricted kernels. To compute their generic dimension, we normalize the null
vector to $\one_F$: the internal rows then become weighted Laplacians, and the
external rows impose labeled zero-sum relations on their componentwise
constants. Linear independence of those rows is exactly the forest condition.
For $J\subseteq\partial F$, let $c_F(J)$ be the number of classes of support
components joined by labels in $J$. The same deficiency also satisfies
\begin{equation}\label{eq:intro-deficiency}
 h_F=\max_{J\subseteq\partial F}
       \bigl(c_F(J)-|\partial F\setminus J|\bigr).
\end{equation}
This follows from the independent-representative argument proved in
Lemma~\ref{lem:linear-representatives}.

For $U=V\setminus S$, put
\begin{equation}\label{eq:intro-complex}
 K_S^\bullet=[R_G^U\xrightarrow{\bA_{V,U}}R_G^V],
 \qquad M_S=H^1(K_S^\bullet)=\coker\bA_{V,U}.
\end{equation}
The Fitting ideals of $M_S$ record all restricted-kernel rank failures. We
resolve their reduced first failure locus by support primes:
\[
 \sqrt{I_{|U|}(\bA_{V,U})}
 =\bigcap_{\substack{F\subseteq U\text{ a fort}\\h_F=1}}\mathfrak p_F.
\]
Each $\mathfrak p_F$ is obtained by saturated elimination from the
exact-support equations. The inclusion-minimal primes in this finite family
are the irreducible components (Theorem~\ref{thm:prime-decomp}). Higher
nullities are treated by diagonal deformation and Grassmann incidence charts.
Example~\ref{ex:crossing} gives nested supports with incomparable
coefficient loci whose intersection is singular at constant nullity.

The module in \eqref{eq:intro-complex} is the fundamental module of Jekel,
Levy, Dana, Stromme, and Litterell \cite{JLDSL}, after the change of variables
\begin{equation}\label{eq:intro-jekel}
 w_{uv}=-t_{uv},\qquad \eta_u=d_u+\sum_{v\sim u}t_{uv}.
\end{equation}
Their universal flatness theorem characterizes zero forcing. We give its
forcing-rule proof and the corresponding chain contraction in Section~2.
We describe the exceptional coefficient loci of this module; related
graph-matrix loci are studied in \cite{AL,DK}.

\subsection{A common kernel and its coefficient sheaf}
A prescribed subspace $W\subseteq\kk^V$ requires one symmetric matrix to
annihilate several vectors simultaneously. If $\ell_u\in W^*$ is the
coordinate functional, its local coefficient space is
\[
 \mathcal C_W(u)=
 \left\{(d,(t_{uv})_{v\sim u}):
 d\ell_u+\sum_{v\sim u}t_{uv}\ell_v=0\right\}.
\]
Restriction to an edge records its coefficient. Global sections therefore
identify the coefficients used by its two ends and are exactly the symmetric
matrices satisfying $AW=0$.

Theorem~\ref{thm:stress} identifies these sections, after a diagonal
normalization, with the direct sum of internal equilibrium stresses,
independent exterior affine dependencies, and free exterior coefficients.
It also characterizes when every edge coefficient can be nonzero. The stress equations of \cite{CGT,CGH} describe the internal coefficients;
the exterior dependencies account for coordinates vanishing identically
on the prescribed kernel. When the local coefficient spaces are lines, their restriction
ratios give a multiplicative transport class; its cycle products are the
remaining obstruction to a global coefficient choice.

\subsection{Open assembly before boundary closure}
For a block with interior $I$ and boundary $B$, a local vector has a boundary
value $b$ and a residual $j$, the block's contribution to the boundary rows.
We retain both in a free two-term complex $\mathcal L_H^\bullet$. Its open
assembly identifies values and adds residuals:
\[
 b_1=\cdots=b_s=b,\qquad j=j_1+\cdots+j_s.
\]
Proposition~\ref{prop:open-port-assembly} gives an isomorphism from this
assembly to the open complex of the union plus free identity complexes.
Closing the single total residual $j=0$ then yields the global equation
complex. For two blocks,
\[
 K_{G,S}^\bullet\simeq
 \operatorname{Cone}\!\left(
 \mathcal L_{H_1}^\bullet\oplus\mathcal L_{H_2}^\bullet
 \xrightarrow{(b_1-b_2,\,j_1+j_2)}\mathcal W_B
 \right)[-1].
\]
The source and target substitutions use identity blocks, so the gluing
maps and chain contractions commute with coefficient base change. Invisible interior solutions survive in the kernel
of the port observation. For an unseeded symmetric block over a field, the
boundary image is Lagrangian, even when the interior matrix is singular;
this is the network interpretation of \cite{BF} in the present notation.

The same assembly has a finite support model. For one seed set we retain
\emph{all} its surviving local support signatures: support membership and a
neighbor count truncated at two at each boundary vertex, together with a
nonemptiness bit. Exact composition of these complete profiles gives a
minimum-cost recurrence. Equal cost functions permit replacement of a
block in every compatible exterior graph.

\subsection{Organization}
Section~2 constructs the equation sheaf and its causal unit contractions.
Sections~3 and~4 prove the support and common-kernel results. Sections~5 and~6
establish open assembly and complete-profile composition; Section~7 gives
calculations on triangulations. The appendices give the profile algorithm
and two matrix calculations. The capacity--phase formula for maximal
biconnected outerplanar graphs is developed in \cite{HLouter}.
The related constructions of finite boundary models, admissible cohomology,
and realization families appear in
\cite{HLcomparison,HLadmissible,HLfamilies}.

\paragraph{Notations.}
Graphs are finite, simple, and undirected unless stated otherwise. We write
$N[u]=N(u)\cup\{u\}$; singleton paths are permitted and $Z(\varnothing)=0$.
Linear statements use a characteristic-zero field $\kk$; geometric closures
and dimensions are over $\CC$. Two-term complexes $[P\to Q]$ have degrees zero
and one. For a matrix $B$, $I_j(B)$ denotes its $j$-minor ideal, with $I_j(B)=R$
for $j\le0$ and $I_j(B)=0$ when $j$ exceeds the matrix dimensions.

\section{Equation sheaves and zero forcing}\label{sec:equations}

The rows of $Ax=0$ define an equation sheaf. Its two-term complex relates
zero forcing to the flatness of a universal coefficient module. A legal
force selects an invertible edge coefficient and gives a chain cancellation.

\subsection{The forcing closure and forts}
For $u\in V$, write $N(u)=\{v:uv\in E\}$ and $N[u]=N(u)\cup\{u\}$. For a current black set $B$, put
\begin{equation}\label{eq:closure-step}
 \Phi(B)=B\cup\{v\notin B:\text{some }u\in B\text{ satisfies }N(u)\setminus B=\{v\}\}.
\end{equation}
Starting with $B_0=S$, define $B_{i+1}=\Phi(B_i)$. Finiteness gives a stationary set, denoted $\cl_G(S)$.

A force remains legal when further vertices are colored black, unless its target has already become black. This monotonicity identifies the set produced by \eqref{eq:closure-step} with every maximal sequential execution.

\begin{lemma}\label{lem:closure}
The set $\cl_G(S)$ is independent of the order of legal forces. It is the least fixed point of $\Phi$ containing $S$.
\end{lemma}
\begin{proof}
Let $B\subseteq C$. If $u\to v$ is legal at $B$ and $v\notin C$, then
\[
 \{v\}\subseteq N(u)\setminus C
 \subseteq N(u)\setminus B=\{v\}.
\]
Hence
\[
 B\subseteq C\quad\Longrightarrow\quad\Phi(B)\subseteq\Phi(C).
\]
Put $C_*=\bigcup_{i\ge0}\Phi^i(S)$. Finiteness gives $\Phi(C_*)=C_*$. If $D$ is any fixed point with $S\subseteq D$, induction gives
\[
 \Phi^i(S)\subseteq\Phi^i(D)=D,
 \qquad C_*\subseteq D.
\]
For a sequential legal execution $S=C_0\subset C_1\subset\cdots$, the fixed-point property of $C_*$ gives $C_i\subseteq C_*$ by induction. If the execution is maximal, its last set $C_r$ satisfies $\Phi(C_r)=C_r$, whence
\[
 C_*\subseteq C_r\subseteq C_*.
\]
Thus every maximal execution ends at $C_*=\cl_G(S)$.
\end{proof}

If the closure is proper, its complement has no unique neighbor from a black vertex. We give this stopping condition a name, together with the associated minimum-size problem.

\begin{definition}
A \emph{zero-forcing set} is a set $S$ with $\cl_G(S)=V$, and
\[
 Z(G)=\min\{|S|:S\subseteq V,\ \cl_G(S)=V\}.
\]
A \emph{fort} is a nonempty set $F\subseteq V$ such that
\begin{equation}\label{eq:fort}
 |N(u)\cap F|\ne1\qquad(u\in V\setminus F).
\end{equation}
Write $\mathfrak F(G)$ for the finite family of forts.
\end{definition}

A set $F$ satisfying \eqref{eq:fort} cannot receive a first force while all its vertices are white. Conversely, the white set at termination satisfies exactly the same condition.

\begin{proposition}\label{prop:fort-hit}
A set $S$ is zero forcing if and only if it meets every fort. Consequently,
\begin{equation}\label{eq:fort-cover}
 Z(G)=\min\{|S|:S\cap F\ne\varnothing\text{ for every }F\in\mathfrak F(G)\}.
\end{equation}
\end{proposition}
\begin{proof}
Suppose $S\cap F=\varnothing$ and $F$ is a fort. If $u\to v$ were the first force into $F$, the black set $B$ immediately before that force would satisfy
\[
 B\cap F=\varnothing,\qquad u\in B,\qquad v\in F,
 \qquad N(u)\setminus B=\{v\}.
\]
Consequently $N(u)\cap F=\{v\}$, contrary to \eqref{eq:fort}.

Conversely, let $C=\cl_G(S)\ne V$ and $F=V\setminus C$. Since $\Phi(C)=C$,
\[
 F\ne\varnothing,\qquad S\cap F=\varnothing,\qquad
 |N(u)\cap F|=|N(u)\setminus C|\ne1\quad(u\in C).
\]
Thus $F\in\mathfrak F(G)$. We have proved
\[
 \cl_G(S)=V
 \quad\Longleftrightarrow\quad
 S\cap F\ne\varnothing\quad\text{for every }F\in\mathfrak F(G).
\]
Taking the minimum of $|S|$ gives \eqref{eq:fort-cover}.
\end{proof}

\subsection{Symmetric witnesses}
For a characteristic-zero field $\kk$, let $\cA_\kk(G)$ be the symmetric matrices over $\kk$ for which, off the diagonal,
\[
 a_{uv}\ne0\quad\Longleftrightarrow\quad uv\in E.
\]
The diagonal entries are arbitrary. At a zero coordinate, the row equation
excludes exactly one nonzero neighboring term. Conversely, a fort permits
a symmetric choice of nonzero edge coefficients, as in the support
results of \cite{Spark,Compatible}.

\begin{proposition}\label{prop:support-witness}
The supports of nonzero vectors in kernels of matrices in $\cA_\kk(G)$ are exactly the forts. More precisely, every fort $F$ has a rational, indeed integer, symmetric witness $A$ with $A\one_F=0$ and the prescribed nonzero edge pattern.
\end{proposition}
\begin{proof}
Let $Ax=0$, $x\ne0$, and $F=\supp x$. If $u\notin F$ and $N(u)\cap F=\{v\}$, then
\[
 0=(Ax)_u=a_{uv}x_v\ne0.
\]
Thus $F$ is a fort.

Conversely, let $F$ be a fort and set $x=\one_F$. For $u\notin F$, write $N(u)\cap F=\{v_1,\ldots,v_k\}$. When $k\ne0$, \eqref{eq:fort} gives $k\ge2$. Choose
\[
 a_{uv_i}=a_{v_i u}=
 \begin{cases}1,&1\le i<k,\\-(k-1),&i=k.\end{cases}
\]
Then
\[
 a_{uv_i}\ne0,\qquad
 \sum_{v\in N(u)\cap F}a_{uv}=(k-1)-(k-1)=0.
\]
Every crossing edge has exactly one endpoint outside $F$, so these assignments are consistent. Put $a_{uv}=a_{vu}=1$ on all remaining edges and set
\[
 a_{uu}=
 \begin{cases}
 -\displaystyle\sum_{v\in N(u)\cap F}a_{uv},&u\in F,\\
 0,&u\notin F.
 \end{cases}
\]
For every $u\in V$,
\[
 (A\one_F)_u=
 \begin{cases}
 a_{uu}+\displaystyle\sum_{v\in N(u)\cap F}a_{uv}=0,&u\in F,\\
 \displaystyle\sum_{v\in N(u)\cap F}a_{uv}=0,&u\notin F.
 \end{cases}
\]
All chosen coefficients are integers, and each prescribed edge coefficient remains nonzero in a characteristic-zero field.
\end{proof}

Extending a vector on $U=V\setminus S$ by zero converts the equation $A_{V,U}x_U=0$ into a full null vector whose support avoids $S$.

\begin{corollary}\label{cor:uniform-kernel}
For $U=V\setminus S$,
\begin{equation}\label{eq:uniform-kernel}
 S\text{ is zero forcing}
 \quad\Longleftrightarrow\quad
 \ker A_{V,U}=0\quad\text{for every }A\in\cA_\kk(G).
\end{equation}
\end{corollary}
\begin{proof}
Propositions~\ref{prop:fort-hit} and \ref{prop:support-witness} give
\begin{align*}
 S\text{ is not zero forcing}
 &\Longleftrightarrow
 \exists F\in\mathfrak F(G),\quad F\subseteq U\\
 &\Longleftrightarrow
 \exists A\in\cA_\kk(G),\ \exists x\ne0,
 \quad Ax=0,\quad\supp x\subseteq U\\
 &\Longleftrightarrow
 \exists A\in\cA_\kk(G),\quad\ker A_{V,U}\ne0.
\end{align*}
\end{proof}

\subsection{The equation sheaf}
To retain the individual rows of $Ax=0$, we introduce one variable object for each $x_v$ and one equation object for each row. Equality of their shared coordinates will be imposed by restriction maps. We first specify the finite topology used for these maps.
\begin{definition}\label{def:cellular}
Let $X$ be a finite graph, allowing parallel edges but no loops. Its \emph{face space} has one point for each vertex and each edge. The order is $v\le e$ when $v$ is an endpoint of $e$, and its topology consists of upward-closed subsets. A cellular sheaf of modules is a collection of modules $\sF(v),\sF(e)$ and maps
\[
 \rho_{v,e}:\sF(v)\longrightarrow\sF(e)\qquad(v\in e).
\]
For an upward-closed set $O$, its sections are the families $(s_c)_{c\in O}$ satisfying $\rho_{v,e}s_v=s_e$ at every incidence contained in $O$.
\end{definition}
For open sets $O'\subseteq O$, the restriction
$\Gamma(O;\sF)\to\Gamma(O';\sF)$ forgets the coordinates indexed by cells
outside $O'$. The compatibility conditions on the retained cells are unchanged.
For a cover $O=\bigcup_iO_i$, compatible sections $(s^{(i)})$ determine
\[
 s_c=s^{(i)}_c\quad(c\in O_i).
\]
Agreement on $O_i\cap O_j$ makes this assignment well-defined; every incidence relation holds in an open set containing its vertex. Hence the assignment is the unique section on $O$ restricting to the given family. The same construction applies to sheaves of sets.

Choose an orientation of each edge. The cellular cochain complex is
\begin{equation}\label{eq:cell-cochains}
 C^0(X;\sF)=\bigoplus_v\sF(v),\qquad
 C^1(X;\sF)=\bigoplus_e\sF(e),\qquad
 (\delta s)_e=\rho_{t(e),e}s_{t(e)}-\rho_{s(e),e}s_{s(e)}.
\end{equation}
Reversing an orientation multiplies one edge coordinate by $-1$ and does not change its cohomology.

The minimal open stars and their intersections are acyclic for sections.
Their \v{C}ech complex is therefore the complex in \eqref{eq:cell-cochains}.

\begin{lemma}\label{lem:cell-derived}
The cohomology of \eqref{eq:cell-cochains} is the sheaf cohomology of the face space. In particular, $H^0=\ker\delta$, $H^1=\coker\delta$, and higher cohomology vanishes.
\end{lemma}
\begin{proof}
For each vertex $v$, let $U_v=\{v\}\cup\{e:v\in e\}$. A section on $U_v$ is determined by its value at $v$, since its value on an incident edge must be the image of that value. Thus
\[
 \Gamma(U_v;\sF)=\sF(v).
\]
A short exact sequence of sheaves on this finite face space is exact at every cell. Evaluation at $v$ is therefore exact, and $H^q(U_v;\sF)=0$ for $q>0$.
For distinct vertices $v,w$, the intersection $U_v\cap U_w$ is the discrete set of edges joining them. Consequently
\[
 \Gamma(U_v\cap U_w;\sF)
 =\bigoplus_{e:\,e=vw}\sF(e),\qquad
 H^q(U_v\cap U_w;\sF)=0\quad(q>0).
\]
Three distinct vertex stars have empty intersection, because a graph edge has exactly two endpoints. Isolated vertices cause no exception: their stars are singleton open sets.

To see why this cover computes derived sections, take an injective resolution of $\sF$ and its augmented \v{C}ech double complex. At a fixed cell $c$, the indices of stars containing $c$ form a nonempty simplex. Inserting any one of these indices contracts its augmented \v{C}ech complex, so the augmentation is exact on stalks. In the other direction, the cohomology on every nonempty intersection is concentrated in degree zero by the preceding exact evaluation calculations. Hence the total complex has the same cohomology as the ordinary \v{C}ech complex of the cover. After orienting each edge this is
\[
 R\Gamma(X;\sF)\simeq
 \left[\bigoplus_v\sF(v)\xrightarrow{\delta}
       \bigoplus_e\sF(e)\right],\qquad
 (\delta s)_e=\rho_{t(e),e}s_{t(e)}-\rho_{s(e),e}s_{s(e)}.
\]
There are no higher \v{C}ech terms. This proves the claim, including the empty face space, whose cochain complex is zero. See also~\cite{CGN} for the cellular formulation.
\end{proof}

Construct a graph $X_G$ with variable vertices $p_v$, equation vertices $q_u$, and zero-anchor vertices $o_u$. It has an edge $p_vq_u$ for each $v\in N[u]$, and an edge $o_uq_u$ for every $u$. This graph uses only the adjacency relation of $G$.

Over any commutative coefficient ring and for a matrix $A$ with the graph zero pattern, assign
\[
 \sF_A(p_v)=R,\qquad \sF_A(q_u)=R^{N[u]},\qquad \sF_A(o_u)=0,
\]
and assign $R$ to every edge. On $p_vq_u$, use the identity from $p_v$ and the $v$th coordinate projection from $q_u$. On $o_uq_u$, use zero from $o_u$ and the row functional
\[
 (y_{uv})_{v\in N[u]}\longmapsto\sum_{v\in N[u]}a_{uv}y_{uv}
\]
from $q_u$. Let $\sF_{A,S}$ be obtained by replacing $\sF_A(p_s)$ by zero for every $s\in S$.

A section of $\sF_{A,S}$ assigns a local copy $y_{uv}$ of $x_v$ to the row at $u$. The equalities $y_{uv}=x_v$ introduce no extra cohomology, as the following change of coordinates shows.

\begin{proposition}\label{prop:equation-sheaf}
There is an isomorphism of complexes
\begin{equation}\label{eq:sheaf-splitting}
 C^\bullet(X_G;\sF_{A,S})
 \cong
 [R^{\mathcal I}\xrightarrow{\id}R^{\mathcal I}]
 \oplus[R^U\xrightarrow{A_{V,U}}R^V],
 \qquad \mathcal I=\{(u,v):v\in N[u]\}.
\end{equation}
It is valid over any commutative ring and commutes with base change.
\end{proposition}
\begin{proof}
Define
\[
 J:R^U\longrightarrow R^{\mathcal I},\qquad
 (Jx)_{uv}=\begin{cases}x_v,&v\in U,\\0,&v\in S,\end{cases}
\]
and
\[
 L_A:R^{\mathcal I}\longrightarrow R^V,\qquad
 (L_Ay)_u=\sum_{v\in N[u]}a_{uv}y_{uv}.
\]
Then $L_AJ=A_{V,U}$. Orient each edge $p_vq_u$ from $p_v$ to $q_u$ and
each edge $o_uq_u$ from $o_u$ to $q_u$. With these orientations,
\[
 \delta(x,y)=(y-Jx,L_Ay).
\]
Use the source coordinates $y'=y-Jx$ and the target map
$Q(z,r)=(z,r-L_Az)$. Thus
\begin{align*}
 Q\delta(x,y'+Jx)
 &=Q(y',L_Ay'+A_{V,U}x)\\
 &=(y',L_Ay'+A_{V,U}x-L_Ay')\\
 &=(y',A_{V,U}x).
\end{align*}
Both coordinate changes are invertible over $R$, with inverses obtained by changing the minus signs to plus signs. They use no division, and commute with every homomorphism $R\to R'$.
\end{proof}
Consequently,
\begin{equation}\label{eq:sheaf-H}
 H^0(X_G;\sF_{A,S})=\{x:Ax=0,\ x|_S=0\},\qquad
 H^1(X_G;\sF_{A,S})=\coker A_{V,U}.
\end{equation}

Prescribed values $b\in R^S$ extend to $Ax=0$ precisely when
$A_{V,S}b$ vanishes in $\coker A_{V,U}$. We use the connecting homomorphism
$\partial(b)=[A_{V,S}b]$; a lift then satisfies
\[
 A_{V,U}x_U=-A_{V,S}b.
\]
The corresponding exact sequence also describes the difference of two extensions.

\begin{proposition}\label{prop:extension}
There is an exact sequence
\begin{equation}\label{eq:boundary-les}
\begin{split}
0\longrightarrow H^0(\sF_{A,S})\longrightarrow\ker A
&\xrightarrow{\mathrm{res}_S}R^S
\xrightarrow{\partial}H^1(\sF_{A,S})\\
&\longrightarrow\coker A\longrightarrow0,
\qquad \partial(b)=[A_{V,S}b].
\end{split}
\end{equation}
Thus $b$ extends to a solution of $Ax=0$ if and only if $\partial(b)=0$; two extensions differ by an element of $H^0(\sF_{A,S})$.
\end{proposition}
\begin{proof}
Put $U=V\setminus S$. There is a short exact sequence of two-term complexes
\[
 0\longrightarrow[R^U\xrightarrow{A_{V,U}}R^V]
 \longrightarrow[R^V\xrightarrow A R^V]
 \longrightarrow[R^S\longrightarrow0]\longrightarrow0.
\]
The first map extends unknowns by zero and is the identity in degree one; the last map is restriction to $S$ in degree zero. Under Proposition~\ref{prop:equation-sheaf}, this is also the sequence obtained from the quotient sheaf $\mathcal Q=\sF_A/\sF_{A,S}$, whose only nonzero stalks are $\mathcal Q(p_s)=R$.

Lift $b\in R^S$ to $\widetilde b\in R^V$ by setting its $U$ coordinates to zero. The differential of the lift is
\[
 A\widetilde b=A_{V,S}b,
\]
so the connecting homomorphism is $\partial(b)=[A_{V,S}b]$ with this convention. Its kernel consists exactly of the extendable boundary values:
\[
 \partial(b)=0
 \Longleftrightarrow A_{V,S}b\in\im A_{V,U}
 \Longleftrightarrow \exists x_U,\quad
 A_{V,U}x_U=-A_{V,S}b.
\]
The map $\coker A_{V,U}\to\coker A$ is induced by the identity on $R^V$. It is onto, and its kernel is
\[
 \frac{\im A_{V,U}+\im A_{V,S}}{\im A_{V,U}}
 =\im\partial.
\]
At the other terms, a null vector has zero trace exactly when it is the zero extension of a vector in $\ker A_{V,U}$. These computations verify exactness at every term of~\eqref{eq:boundary-les}. The difference of two extensions has zero trace and lies in that same kernel.
\end{proof}

\subsection{The universal module and specialization}
Return to $R=R_G$ and the universal $\bA$ from \eqref{eq:intro-ring}.

Before specializing the coefficients, the independent diagonal variables prevent a column relation. This gives the free resolution needed to compute the effect of specialization.

\begin{lemma}\label{lem:universal-injective}
The map $\bA_{V,U}:R^U\to R^V$ is injective.
\end{lemma}
\begin{proof}
In the determinant expansion,
\[
 \det\bA=\prod_{v\in V}d_v+
 \sum_{\substack{\sigma\in\operatorname{Sym}(V)\\\sigma\ne\id}}
 \operatorname{sgn}(\sigma)\prod_{v\in V}\bA_{v,\sigma(v)}.
\]
No summand in the sum contains $\prod_vd_v$. Hence $\det\bA\ne0$.
If $\bA_{V,U}x=0$, extend $x$ by zero on $S$ to $\widetilde x$. Then
\[
 \bA\widetilde x=0
 \quad\Longrightarrow\quad
 \operatorname{adj}(\bA)\bA\widetilde x
 =(\det\bA)\widetilde x=0
 \quad\Longrightarrow\quad\widetilde x=0,
\]
since $R$ is an integral domain. If $U=\varnothing$, injectivity is immediate.
\end{proof}
We therefore have a free resolution
\begin{equation}\label{eq:resolution}
 0\longrightarrow R^U\xrightarrow{\bA_{V,U}}R^V
 \longrightarrow M_S\longrightarrow0.
\end{equation}
For any $R$-module $N$, tensoring this resolution gives
\begin{equation}\label{eq:tor-kernel}
 \Tor_1^R(M_S,N)=\ker(\bA_{V,U}:N^U\to N^V).
\end{equation}
Thus specialization produces the kernel measured by this Tor group.

For a finite presentation $R^a\xrightarrow{B}R^b\to M\to0$, our convention is
\[
 \Fitt_j(M)=I_{b-j}(B),
\]
with the same conventions for nonpositive and oversized minor indices as above. The ideals are independent of the presentation~\cite[\stacks{07Z6}]{Stacks}. In the presentation~\eqref{eq:resolution}, $b=|V|$ and $a=|U|$, so for $1\le q\le |U|$,
\begin{equation}\label{eq:fitting-jump}
 \Fitt_{|S|+q-1}(M_S)
 =I_{|V|-|S|-q+1}(\bA_{V,U})
 =I_{|U|-q+1}(\bA_{V,U}).
\end{equation}
On a field fiber, the last ideal vanishes precisely when
\[
 \rank\bA_{V,U}(\alpha)\le |U|-q
 \quad\Longleftrightarrow\quad
 \dim\Tor_1^R(M_S,k_\alpha)\ge q.
\]
These Fitting ideals define the kernel-jump schemes. Section~\ref{sec:support}
decomposes the first such locus by the supports of its null vectors.

A module $M$ over $R$ is \emph{flat} if tensoring with $M$ preserves short exact sequences of $R$-modules. For the length-one free resolution \eqref{eq:resolution}, its failure in a coefficient module is measured explicitly by \eqref{eq:tor-kernel}.

A forcing sequence selects a triangular minor whose diagonal entries are edge
units. Its cokernel is therefore free over the coefficient ring. The following
is the zero-forcing form of the universal flatness theorem in
\cite[Theorem~5.16]{JLDSL}.

\begin{theorem}\label{thm:flatness}
For every finite simple graph and $S\subseteq V$, the following are equivalent:
\begin{enumerate}
\item $S$ is a zero-forcing set;
\item $M_S\cong R^{|S|}$;
\item $M_S$ is flat over $R$;
\item $I_{|U|}(\bA_{V,U})=R$.
\end{enumerate}
Forcing also implies the kernel vanishing in \eqref{eq:tor-kernel} for every $R$-module $N$.
\end{theorem}
\begin{proof}
Write $m=|U|$. If $m=0$, the complex is $[0\to R^V]$ and all four
conditions hold by the minor convention. Suppose first that $m>0$ and take
a complete chronology $u_1\to v_1,\ldots,u_m\to v_m$. Every nonseed is a
target exactly once. A vertex cannot force twice: immediately after its first
force all its neighbours are black. Thus the $u_i$ are also distinct.

At step $i$, every later target $v_j$, $j>i$, is white. It is distinct from
the black source $u_i$ and cannot be another neighbour of that source.
Consequently
\[
 \bA_{u_i,v_j}=0\ (j>i),\qquad
 \bA_{u_i,v_i}=t_{\{u_i,v_i\}}\in R^\times.
\]
The selected row--column minor $B=(\bA_{u_i,v_j})$ is lower triangular with
\[
 \det B=\prod_{i=1}^m t_{\{u_i,v_i\}}\in R^\times.
\]
Permute the source rows first and write $\bA_{V,U}=\binom{B}{C}$. The source
map $x\mapsto B^{-1}x$ and target map
\[
 Q(y,z)=(y,z-CB^{-1}y),\qquad
 Q^{-1}(y,z)=(y,z+CB^{-1}y)
\]
are invertible over $R$, and
\[
 Q\binom{B}{C}B^{-1}=\binom{I_m}{0}.
\]
Hence the equation complex is the direct sum of an identity complex of rank
$m$ and $[0\to R^{|S|}]$. Its cokernel is $R^{|S|}$, and the displayed
unit minor gives $I_m(\bA_{V,U})=R$. In particular the cokernel is flat.

For the converse assume that $S$ does not force. A fort $F\subseteq U$ and
the integer witness of Proposition~\ref{prop:support-witness} give an
$R$-algebra evaluation $\alpha:R\to\kk$ with every edge parameter nonzero,
and a nonzero $x\in\kk^U$ such that $\bA_{V,U}(\alpha)x=0$. Applying the
resolution~\eqref{eq:resolution} to the $R$-module $\kk_\alpha$ yields
\[
 \Tor_1^R(M_S,\kk_\alpha)\ne0.
\]
Thus $M_S$ is not flat and cannot be free. The same rank failure annihilates
every $m$-minor under $\alpha$, so
$I_m(\bA_{V,U})\subseteq\ker\alpha\ne R$. This proves the converse
implications.

Finally, for any $R$-module $N$, a row used by a legal force reduces to
$t_{uv}x_v=0$. Multiplication by the unit $t_{uv}$ is invertible on $N$.
Induction over the chronology therefore gives $x_v=0$ for every nonseed,
which is the asserted kernel vanishing over arbitrary coefficient modules.
\end{proof}

To compare the last theorem with the graph-with-boundary formulation, view black vertices as boundary vertices. Removing edges between them turns a force $u\to v$ into a boundary spike.

\begin{proposition}\label{prop:layerability}
With boundary $S$ and interior $U$, ordinary zero forcing is equivalent to complete reduction by deleting boundary--boundary edges, deleting isolated boundary vertices, and contracting a boundary spike whose boundary endpoint has degree one. Under \eqref{eq:intro-jekel}, $M_S$ is the fundamental module of \cite[Definition 2.3]{JLDSL}.
\end{proposition}
\begin{proof}
Let $B$ be the black set. Delete the edges whose two endpoints lie in $B$. Then
\[
 N(u)\setminus B=\{v\}
 \quad\Longleftrightarrow\quad
 \text{the remaining boundary vertex }u\text{ has only the edge }uv.
\]
Contract this edge, discard the used source $u$, and regard $v$ as boundary. No later force needs $u$: all its neighbors were black after $u\to v$. Iteration, together with deletion of isolated boundary vertices, converts a complete forcing sequence into the stated reduction.

Conversely, a deleted boundary--boundary edge has two black endpoints, and a deleted isolated boundary vertex has no white neighbor. Immediately before a spike contraction, every original neighbor of its boundary endpoint other than its retained endpoint is therefore black. The contraction is a legal force. Reading the contractions in order gives a complete forcing sequence.

For the matrix identification, the generalized Laplacian has
\[
 L_{uv}=-w_{uv}\quad(u\ne v),\qquad
 L_{uu}=\eta_u+\sum_{v\sim u}w_{uv}.
\]
Using \eqref{eq:intro-jekel},
\[
 L_{uv}=t_{uv},\qquad
 L_{uu}=d_u+\sum_{v\sim u}t_{uv}-\sum_{v\sim u}t_{uv}=d_u.
\]
Hence $L=\bA$ and
$RV/L(RU)=\coker\bA_{V,U}=M_S$.
\end{proof}

The flatness condition in Theorem~\ref{thm:flatness} cannot be replaced by torsion-freeness. The path with its middle vertex initially black already distinguishes the two properties.

\begin{example}\label{ex:path-ideal}
For the path $1-2-3$ with $S=\{2\}$,
\[
 \bA_{V,\{1,3\}}=
 \begin{pmatrix}d_1&0\\a&b\\0&d_3\end{pmatrix},\qquad a,b\in R^\times.
\]
Using the unit $a$ to remove one variable--equation pair gives
\[
 M_S\cong\coker\left(R\xrightarrow{(-ba^{-1}d_1,\ d_3)^{\mathsf T}}R^2\right)
 \cong(d_1,d_3).
\]
For the second isomorphism use $(r,s)\mapsto d_3r+ba^{-1}d_1s$. Its image is $(d_1,d_3)$ and its kernel is the displayed cyclic relation, since $d_1,d_3$ are relatively prime polynomial variables. The ideal is torsion-free of rank one, but at the prime $(d_1,d_3)$ it needs two generators and is not free. A finitely presented flat module is locally free \cite[\stacks{00NX}]{Stacks}; hence this module is not flat. The rank-drop locus is $d_1=d_3=0$, where $(b,0,-a)$ is an invisible kernel vector.
\end{example}

\subsection{Cancellation along the forcing closure}
The proof of Theorem~\ref{thm:flatness} can be stopped before all vertices are forced. Each completed force still gives a unit cancellation, and the unforced columns are left unchanged.
For $C=\cl_G(S)$ put $T=C\setminus S$, $W=V\setminus C$, and $k=|T|$.

\begin{theorem}\label{thm:closure-cancel}
For a chosen legal sequence from $S$ to $C$, there is a free complex $K_{\mathrm{res}}$ with
\begin{align}
 K_S&\cong[R^k\xrightarrow{\id}R^k]\oplus K_{\mathrm{res}},\label{eq:closure-split1}\\
 K_C&\cong K_{\mathrm{res}}\oplus[0\longrightarrow R^k].\label{eq:closure-split2}
\end{align}
In particular $M_C\cong M_S\oplus R^k$. For every $q\ge1$,
\begin{equation}\label{eq:closure-minors}
 I_{|V\setminus S|-q+1}(\bA_{V,V\setminus S})
 =I_{|V\setminus C|-q+1}(\bA_{V,V\setminus C}).
\end{equation}
All chain operations commute with arbitrary base change.
\end{theorem}
\begin{proof}
Put $T=C\setminus S$, $W=V\setminus C$, and let $|T|=k$. Order $T$ by the
chosen forces $u_i\to v_i$, and put $Q=\{u_1,\ldots,u_k\}$. The sources are
distinct, and the argument of Theorem~\ref{thm:flatness} gives a triangular
matrix
\[
 P=\bA_{Q,T},\qquad
 \det P=\prod_{i=1}^k t_{\{u_i,v_i\}}\in R^\times.
\]
Every $w\in W$ stays white throughout this partial chronology. It is not a
source and cannot be another white neighbour of a source at its force.
Hence $Q\cap W=\varnothing$ and $\bA_{Q,W}=0$.

With rows $(Q,V\setminus Q)$ and columns $(T,W)$, write
\[
 \bA_{V,V\setminus S}=\begin{pmatrix}P&0\\D&B\end{pmatrix},
 \qquad \bA_{V,W}=\binom0B.
\]
The invertible target operation
\[
 U_0=\begin{pmatrix}P^{-1}&0\\-DP^{-1}&I\end{pmatrix},
 \qquad U_0^{-1}=\begin{pmatrix}P&0\\D&I\end{pmatrix}
\]
sends the first matrix to $\operatorname{diag}(I_k,B)$ and leaves the second
matrix equal to $\binom0B$. Thus the same residual complex
$K_{\mathrm{res}}=[R^W\xrightarrow B R^{V\setminus Q}]$ occurs in
\eqref{eq:closure-split1} and~\eqref{eq:closure-split2}. Taking cokernels
gives $M_C\cong M_S\oplus R^k$.

For the minors let $s=|W|-q+1$. If $s>0$, expansion along the selected rows
and columns of the identity block expresses any $(s+k)$-minor of
$\operatorname{diag}(I_k,B)$ as zero or a signed minor of $B$ of size at
least $s$. A larger minor belongs to $I_s(B)$ by repeated Laplace expansion.
Conversely, adjoining all $k$ identity rows and columns to an $s$-minor of
$B$ produces that same minor. Hence
\[
 I_{s+k}(\operatorname{diag}(I_k,B))=I_s(B)
 =I_s\!\binom0B.
\]
For $s\le0$, the right side is $R$; the left side either has nonpositive
index or contains a minor of the identity block equal to one. If $s$ exceeds
the dimensions of $B$, both ideals are zero. Invertible source and target
operations preserve determinantal ideals, proving~\eqref{eq:closure-minors}.
Every inverse above uses only the already invertible edge parameters, so the
complex isomorphisms and the ideal identity commute with any base change.
\end{proof}

For every $R$-module $N$, the same row induction gives the stronger equality inside $N^V$:
\[
 \{x\in N^V:\bA x=0,\ x|_S=0\}
 =\{x\in N^V:\bA x=0,\ x|_C=0\}.
\]
The equality of minor ideals in \eqref{eq:closure-minors} retains the scheme structure as well as the pointwise kernel dimensions.

\subsection{Forcing paths and positive cochains}\label{sec:causal}
A complete forcing sequence selects at most one outgoing arrow from each vertex. We first identify the paths formed by these arrows, and then express all prerequisites for a proposed path as inequalities on its coloring times. The closed neighborhood in these inequalities includes the source itself.

\begin{lemma}\label{lem:forcing-paths}
The arrows of a complete forcing process form a vertex-disjoint induced path cover. Its initial vertices are the seeds, and its number of paths is the number of seeds.
\end{lemma}
\begin{proof}
Let $\Gamma$ be the digraph of the forces. Every nonseed has one incoming arrow, every seed has none, and every vertex has at most one outgoing arrow. If $t_v$ is its coloring time, then
\[
 u\to v\quad\Longrightarrow\quad t_u<t_v.
\]
Thus $\Gamma$ has no directed cycle and is a disjoint union of directed paths, including singleton seeds.

Suppose a path $v_0\to\cdots\to v_k$ has an edge $v_iv_j$ with $j\ge i+2$. Immediately before $v_i\to v_{i+1}$, both $v_{i+1}$ and $v_j$ are white, so
\[
 |N(v_i)\setminus B|\ge2,
\]
contrary to the forcing rule. Hence each path is induced. Its unique vertex with no incoming arrow is a seed, so the number of paths is $|S|$.
\end{proof}
Write $P(G)$ for the minimum size of a vertex-disjoint induced path cover. In particular $P(G)\le Z(G)$.

A proposed path arrow $u\to v$ can be executed only after every vertex of $N[u]\setminus\{v\}$ is black. We encode all these prerequisites, including edges not belonging to the path cover.

\begin{definition}
For an arbitrary oriented induced path cover $\cP$, let $R_{\cP}$ be its initial vertices. For each path arrow $u\to v$, insert the dependencies
\begin{equation}\label{eq:cover-dependencies}
 w\longrightarrow v\qquad(w\in N[u]\setminus\{v\}).
\end{equation}
The resulting digraph on $V$ is $D_{\cP}$. Each directed edge is a
separate one-cell in its cell realization, including a pair of opposite edges.
\end{definition}

The digraph $D_{\cP}$ separates the choice of paths from the question of whether those paths admit a legal ordering.

This formulation of the forcing-arc and chain-twist criterion
in~\cite{CJMZ} uses closed neighborhoods. Parallel increasing path covers
give a related formulation~\cite{PIP}.

\begin{proposition}\label{prop:causal-cover}
The cover $\cP$ occurs as the full set of forcing paths of a legal process if and only if $D_{\cP}$ is acyclic. Equivalently, there is a function $t:V\to\RR$ such that
\begin{equation}\label{eq:causal-time}
 t|_{R_{\cP}}=0,\qquad t_v-t_w>0\quad(w\to v\text{ in }D_{\cP}).
\end{equation}
\end{proposition}
\begin{proof}
A legal coloring time, with initial times zero, satisfies
\[
 w\in N[u]\setminus\{v\},\quad u\to v
 \quad\Longrightarrow\quad t_w<t_v.
\]
For a directed cycle $v_0\to\cdots\to v_r=v_0$, this would give
\[
 0=t_{v_r}-t_{v_0}
   =\sum_{i=0}^{r-1}(t_{v_{i+1}}-t_{v_i})>0.
\]
Hence $D_{\cP}$ is acyclic.

Conversely, the sources of $D_{\cP}$ are exactly $R_{\cP}$: no dependency ends at an initial vertex, and every other vertex receives an arrow from its path predecessor. For an acyclic $D_{\cP}$, put
\[
 t_v=\max\{\text{length of a directed path from }R_{\cP}\text{ to }v\}.
\]
Then
\[
 t|_{R_{\cP}}=0,\qquad t_v\ge t_w+1\quad(w\to v).
\]
Process the noninitial vertices in increasing order of $t$. If $u$ is the path predecessor of $v$, every vertex of $N[u]\setminus\{v\}$ has already been colored. Thus $u\to v$ is legal. The same argument applies to any function satisfying \eqref{eq:causal-time}.
\end{proof}

The time inequalities can also be expressed in the cohomology of a sheaf. To impose the initial values $t|_{R_{\cP}}=0$, we use zero vertex stalks at the path starts.
On the face space of $D_{\cP}$ put a time sheaf $\sT_{\cP}$ with stalk zero at $R_{\cP}$ and stalk $\RR$ at every other vertex and at every edge. Nonzero incidence maps are identities. Its complex is
\[
 C^0=\{t\in\RR^V:t|_{R_{\cP}}=0\},\qquad
 C^1=\RR^{E(D_{\cP})},\qquad
 (\delta t)_{w\to v}=t_v-t_w.
\]
Therefore Proposition~\ref{prop:causal-cover} says
\begin{equation}\label{eq:positive-obstruction}
 \cP\text{ is realizable}
 \quad\Longleftrightarrow\quad
 \text{some }\ell\in(0,\infty)^{E(D_{\cP})}
 \text{ has }[\ell]=0\text{ in }H^1(D_{\cP};\sT_{\cP}).
\end{equation}
For an empty edge set, the positive cone is the singleton in the zero-dimensional vector space.

For a directed cycle $\gamma$, every positive cochain and every coboundary satisfy, respectively,
\[
 \langle\ell,\gamma\rangle>0,\qquad
 \langle\delta t,\gamma\rangle=0.
\]
If $D_{\cP}$ is acyclic, Proposition~\ref{prop:causal-cover} gives a positive coboundary $\ell=\delta t$. The edge differences need not be equal: for $a\to b$, $b\to c$, and $a\to c$ one has
\[
 \ell_{a\to c}=\ell_{a\to b}+\ell_{b\to c}.
\]
Thus the realizability condition is
\[
 (0,\infty)^{E(D_{\cP})}\cap\im\delta\ne\varnothing.
\]

\subsection{The additional cost of a path cover}
If $D_{\cP}$ has a directed cycle, declaring an additional vertex initially black removes the need to force that vertex. Cutting the selected path there removes every dependency entering it. This operation relates the additional seed cost to directed feedback vertices.
For a digraph $D$, let
\[
 \tau_{\mathrm{dir}}(D)=\min\{|T|:D-T\text{ has no directed cycle}\}
\]
be its directed feedback-vertex number.

Fix the directions of $\cP$. A cut set $T$ means that the selected predecessor arrow of each $v\in T$ is removed and $v$ is made initial; every other selected arrow is retained. We may take $T\cap R_{\cP}=\varnothing$, since an initial vertex has no incoming dependency and lies on no directed cycle. The new dependency graph has edge set
\[
 \{w\to v\in E(D_{\cP}):v\notin T\}.
\]
Vertices of $T$ become sources in this graph. Its directed cycles are therefore exactly the directed cycles of $D_{\cP}-T$. Thus $\tau_{\mathrm{dir}}(D_{\cP})$ is the minimum number of cuts with
the remaining predecessor arrows fixed. The next theorem minimizes over
the choice of those arrows as well.

\begin{theorem}\label{thm:feedback}
For every finite simple graph,
\begin{equation}\label{eq:feedback-formula}
 Z(G)=\min_{\cP}\bigl(|\cP|+\tau_{\mathrm{dir}}(D_{\cP})\bigr),
\end{equation}
where the minimum ranges over all oriented induced path covers. In particular,
\begin{equation}\label{eq:ZP}
 Z(G)=P(G)
\end{equation}
if and only if some minimum induced path cover admits an orientation with acyclic dependency graph.
\end{theorem}
\begin{proof}
Fix $\cP$ and choose $T\subseteq V$ with
\[
 |T|=\tau_{\mathrm{dir}}(D_{\cP}),\qquad
 D_{\cP}-T\text{ acyclic}.
\]
No initial vertex lies on a directed cycle, so $T\cap R_{\cP}=\varnothing$ may be assumed. Cut the path edge entering each $v\in T$. The resulting induced path cover $\cP_T$ satisfies
\[
 R_{\cP_T}=R_{\cP}\sqcup T,\qquad
 |\cP_T|=|\cP|+|T|,
\]
and its dependency edges are
\[
 E(D_{\cP_T})
 =\{w\to v\in E(D_{\cP}):v\notin T\}.
\]
A directed cycle of $D_{\cP_T}$ would avoid $T$, and would therefore be a cycle of $D_{\cP}-T$. Hence $D_{\cP_T}$ is acyclic. Proposition~\ref{prop:causal-cover} gives
\[
 Z(G)\le|\cP_T|
 =|\cP|+\tau_{\mathrm{dir}}(D_{\cP}).
\]

Conversely, an optimal forcing process gives, by Lemma~\ref{lem:forcing-paths}, an induced path cover $\cP$ with
\[
 |\cP|=Z(G),\qquad \tau_{\mathrm{dir}}(D_{\cP})=0.
\]
This proves \eqref{eq:feedback-formula}. Finally,
\[
 |\cP|\ge P(G),\qquad \tau_{\mathrm{dir}}(D_{\cP})\ge0.
\]
Their sum equals $P(G)$ precisely when both inequalities are equalities, proving the assertion about \eqref{eq:ZP}.
\end{proof}
Equivalently,
\[
 Z(G)-P(G)=\min_{\cP}
 \bigl(|\cP|-P(G)+\tau_{\mathrm{dir}}(D_{\cP})\bigr).
\]
For example, if $n\ge2$, every induced path in $K_n$ has at most two vertices, and any black vertex with at least two white neighbors is unable to force. Therefore
\[
 P(K_n)=\lceil n/2\rceil,\qquad Z(K_n)=n-1,\qquad
 Z(K_n)-P(K_n)=\lfloor n/2\rfloor-1.
\]
For a fixed $\cP$, the number $\tau_{\mathrm{dir}}(D_{\cP})$ counts the
prescribed predecessor cuts. Formula~\eqref{eq:feedback-formula} also minimizes
over $\cP$. The next example shows why that minimization is necessary.

\begin{example}\label{ex:prescribed-repair}
On vertices $0,\ldots,4$, take the three triangular faces $(0,1,4),(1,2,4),(2,3,4)$. Consider the directed induced paths
\[
 (0,1),\qquad(2,4),\qquad(3).
\]
Their starts are $S=\{0,2,3\}$. The prescribed arrows are $0\to1$ and $2\to4$. The first waits for vertex 4, while the second waits for vertex 1. In the dependency graph this gives the directed two-cycle $1\to4\to1$. Every other dependency starts at a seed, so its feedback-vertex number is one. Thus one predecessor cut is necessary if these path arrows are retained.

Nevertheless the same initial set already forces the whole graph by different arrows:
\[
 3\to4,\qquad0\to1.
\]
The source 3 has neighbours 2 and 4, and 2 is initial; after this force the source 0 has only 1 white. No extra seed is required after changing the selected arrows.
\end{example}

\subsection{Acyclic matchings and chain contractions}
The temporal dependencies also describe the unit entries that can be eliminated in $\bA_{V,U}$. This is the form of the matching construction that will be used for the equation complex; related constrained matchings and algebraic Morse contractions are treated in \cite{TD,Skoldberg,CGN}.
Let $U=V\setminus S$. An \emph{edge-unit matching} is an injection $\phi:U\to V$ with $\phi(v)\sim v$. Its dependency graph on $U$ has
\begin{equation}\label{eq:matching-dependencies}
 w\to v\quad\text{when}\quad w\in U\cap(N[\phi(v)]\setminus\{v\}).
\end{equation}
The matched entries are edge units. The closed neighborhood in
\eqref{eq:matching-dependencies} also imposes the prerequisite that a
matched source be black.

An acyclic matching orders the chosen row--column pairs so that every off-diagonal prerequisite precedes its target. Conversely, a forcing sequence supplies exactly such an order.

\begin{theorem}\label{thm:matching}
The set $S$ is zero forcing if and only if there is an edge-unit matching covering $U$ whose dependency graph is acyclic.
\end{theorem}
\begin{proof}
For a complete forcing sequence, let $\phi(v)$ be the source that forces $v$. Distinct targets have distinct sources, and
\[
 w\in U\cap(N[\phi(v)]\setminus\{v\})
 \quad\Longrightarrow\quad t_w<t_v.
\]
Thus $\phi$ is an edge-unit matching with acyclic dependency graph.

Conversely, choose a topological order for an acyclic matching. Immediately before the step for $v$, its earlier vertices and all seeds are black. By \eqref{eq:matching-dependencies},
\[
 N[\phi(v)]\setminus\{v\}\subseteq B,
 \qquad v\notin B.
\]
In particular $\phi(v)\in B$ and $N(\phi(v))\setminus B=\{v\}$. Hence $\phi(v)\to v$ is legal. Executing the order forces every vertex of $U$.
\end{proof}

In the forcing order, the matched block has the form $B=\Delta(I-L)$,
where $\Delta$ is invertible diagonal and $L$ is strictly triangular.
Thus its inverse is the finite path sum in~\eqref{eq:path-inverse}.

\begin{theorem}\label{thm:contraction}
For such a matching put $Q=\phi(U)$, $J=V\setminus Q$, and order the rows of $Q$ by their matched columns. Write
\[
 \bA_{V,U}=\binom BC,\qquad
 \Delta=\diag(t_{\phi(v),v}),\qquad L=I-\Delta^{-1}B.
\]
If $m=|U|>0$, then $L^m=0$ and
\begin{equation}\label{eq:path-inverse}
 B^{-1}=(I+L+\cdots+L^{m-1})\Delta^{-1}.
\end{equation}
The maps
\begin{equation}\label{eq:contraction-maps}
 h(y,z)=B^{-1}y,\qquad
 p(y,z)=z-CB^{-1}y,\qquad i(z)=(0,z)
\end{equation}
satisfy
\begin{equation}\label{eq:contraction-identities}
 h\bA_{V,U}=I,\qquad
 \bA_{V,U}h+ip=I,\qquad pi=I.
\end{equation}
They give a chain contraction onto $[0\to R^J]$, compatible with every coefficient base change. If $\mathfrak m,\mathfrak n$ are two such matchings, their coordinate changes
\[
 T_{\mathfrak n\mathfrak m}=p_{\mathfrak n}i_{\mathfrak m}
\]
are inverse isomorphisms and satisfy
\begin{equation}\label{eq:coherent-changes}
 T_{\mathfrak p\mathfrak n}T_{\mathfrak n\mathfrak m}=T_{\mathfrak p\mathfrak m}.
\end{equation}
\end{theorem}
\begin{proof}
Every nonzero entry $L_{v,w}$ follows an arrow $w\to v$. For $r\ge1$,
\[
 (L^r)_{v,w}
 =\sum_{w=v_0\to v_1\to\cdots\to v_r=v}
 L_{v_r,v_{r-1}}\cdots L_{v_1,v_0}.
\]
There is no directed walk of length $m$ in an acyclic graph on $m$ vertices, so $L^m=0$. As $B=\Delta(I-L)$,
\[
 (I-L)\sum_{r=0}^{m-1}L^r=I-L^m=I,
 \qquad B^{-1}=\sum_{r=0}^{m-1}L^r\Delta^{-1}.
\]
For the maps in \eqref{eq:contraction-maps},
\begin{align*}
 h\bA_{V,U}x&=h(Bx,Cx)=B^{-1}Bx=x,\\
 (\bA_{V,U}h+ip)(y,z)
 &=(y,CB^{-1}y)+(0,z-CB^{-1}y)=(y,z),\\
 pi(z)&=p(0,z)=z.
\end{align*}
Also $p\bA_{V,U}=0$ and $hi=0$, proving the contraction onto $[0\to R^J]$. For $m=0$, use the identity map on $[0\to R^V]$.

For two matchings, write the identity on the common degree-one module as
\[
 i_{\mathfrak n}p_{\mathfrak n}
 =I-\bA_{V,U}h_{\mathfrak n}.
\]
Then
\begin{align*}
 T_{\mathfrak p\mathfrak n}T_{\mathfrak n\mathfrak m}
 &=p_{\mathfrak p}i_{\mathfrak n}p_{\mathfrak n}i_{\mathfrak m}\\
 &=p_{\mathfrak p}(I-\bA_{V,U}h_{\mathfrak n})i_{\mathfrak m}\\
 &=p_{\mathfrak p}i_{\mathfrak m}=T_{\mathfrak p\mathfrak m}.
\end{align*}
Taking $\mathfrak p=\mathfrak m$ gives the inverse identities. All matrices are defined over $R$, with inverses only of edge units; every coefficient base change preserves the displayed identities.
\end{proof}

The same matched block can be eliminated without fixing the values on $S$. The remaining equations relate the seed values to the rows at vertices which were not used as sources.

\begin{corollary}\label{cor:seed-terminal}
For a forcing matching as above, put
\begin{equation}\label{eq:psi}
 \Psi_S=\bA_{J,S}-CB^{-1}\bA_{Q,S}.
\end{equation}
Then the full equation complex has a decomposition, up to a contractible identity summand,
\begin{equation}\label{eq:seed-terminal}
 [R^V\xrightarrow{\bA}R^V]\simeq[R^S\xrightarrow{\Psi_S}R^J].
\end{equation}
In every field fiber,
\[
 \ker\bA\cong\ker\Psi_S,\qquad
 \rank\bA=|U|+\rank\Psi_S.
\]
\end{corollary}
\begin{proof}
Order the rows by $(Q,J)$ and the columns by $(U,S)$. Then
\[
 \begin{pmatrix}I&0\\-CB^{-1}&I\end{pmatrix}
 \begin{pmatrix}B&\bA_{Q,S}\\C&\bA_{J,S}\end{pmatrix}
 \begin{pmatrix}I&-B^{-1}\bA_{Q,S}\\0&I\end{pmatrix}
 =\begin{pmatrix}B&0\\0&\Psi_S\end{pmatrix}.
\]
The block $B$ is invertible, and hence gives a contractible summand. On kernels the remaining identification is
\[
 b\longmapsto(-B^{-1}\bA_{Q,S}b,b),\qquad
 b\in\ker\Psi_S.
\]
This proves \eqref{eq:seed-terminal} and
$\rank\bA=|U|+\rank\Psi_S$ after specialization to any field.
\end{proof}
The transition identities \eqref{eq:coherent-changes} show that different complete forcing sequences give compatible coordinates on $M_S$. The noncontractible equations of the full matrix are those of $\Psi_S$ in \eqref{eq:psi}.

\section{Support varieties and rank strata}\label{sec:support}

We now determine the coefficient loci admitting a prescribed kernel support.
Work over $\CC$, and set $n=|V|$, $e=|E|$,
$R=\CC[d_v,t_a^{\pm1}]$, and
\[
 \Omega_G=\Spec R=\AA^n\times\Gm^e.
\]
Fix $S\subseteq V$, write $U=V\setminus S$, and set $m=|U|$. A coefficient point is denoted by $\alpha$; its matrix is $\bA(\alpha)$.

\subsection{The incidence variety of a fort}
Proposition~\ref{prop:support-witness} identifies the possible nonzero supports. To keep track of the coefficients realizing one such support, we retain both the matrix parameter and the projective null vector.
For a nonempty subset $F\subseteq U$, define
\begin{equation}\label{eq:exact-incidence}
 \cI_F=\{(\alpha,[x]):\bA(\alpha)x=0,\ \supp x=F\}
 \subseteq\Omega_G\times\PP(\CC^F).
\end{equation}
The vector is extended by zero on $V\setminus F$. The condition of exact support is open in its projective coordinate space. Let
\[
 \cD_F=\overline{\pi(\cI_F)}\subseteq\Omega_G,
\]
where $\pi$ forgets $[x]$ and the closure is Zariski closure. Write
\[
 \partial F=N(F)\setminus F,\quad k_u=|N(u)\cap F|\ (u\in\partial F),\quad
 e_\times(F)=|E(F,V\setminus F)|.
\]

At a vertex outside $F$, the equation $\bA x=0$ is a zero sum of nonzero neighboring terms. The following hyperplane complements describe those equations.

\begin{definition}
For $j\ge2$, let
\[
 \mathcal U_j=\{(z_1,\ldots,z_j)\in\Gm^j:z_1+\cdots+z_j=0\}.
\]
Thus $\mathcal U_j$ is the complement of the coordinate hyperplanes in the hyperplane $z_1+\cdots+z_j=0$.
\end{definition}
Since
\[
 (1,\ldots,1,-(j-1))\in\mathcal U_j,
\]
this is a nonempty open subset of an affine space of dimension $j-1$. In particular, it is smooth and irreducible.

The equations at vertices of $F$ determine their diagonal coefficients. The remaining equations use disjoint groups of crossing-edge variables, one group for each $u\in\partial F$.

\begin{theorem}\label{thm:incidence-product}
The variety $\cI_F$ is nonempty exactly when $F$ is a fort. For a fort,
\begin{equation}\label{eq:incidence-product}
 \cI_F\cong
 \Gm^{|F|-1+e-e_\times(F)}\times\AA^{n-|F|}
 \times\prod_{u\in\partial F}\mathcal U_{k_u}.
\end{equation}
Consequently it is smooth and irreducible, and
\begin{equation}\label{eq:incidence-dim}
 \dim\cI_F=n+e-|\partial F|-1.
\end{equation}
\end{theorem}
\begin{proof}
Choose $r\in F$ and normalize $x_r=1$. Since the support is exactly $F$, this is a global chart on the projective support torus. Its free coordinates are
\[
 (x_v)_{v\in F\setminus\{r\}}\in\Gm^{|F|-1}.
\]
Change the edge coordinates by
\[
 c_{uv}=t_{uv}x_ux_v\quad(u,v\in F),\qquad
 z_{uv}=t_{uv}x_v\quad(u\notin F,\,v\in F).
\]
The inverse substitutions are $t_{uv}=c_{uv}/(x_ux_v)$ and $t_{uv}=z_{uv}/x_v$. All denominators are units on the support torus.

For an exterior vertex $u$,
\begin{equation}\label{eq:outside-zero-sum}
 0=(\bA x)_u
   =\sum_{v\in N(u)\cap F}t_{uv}x_v
   =\sum_{v\in N(u)\cap F}z_{uv}.
\end{equation}
This has no solution if $k_u=1$, and gives $\mathcal U_{k_u}$ if $k_u\ge2$. The coordinate groups for different $u$ are disjoint.

For $u\in F$,
\[
 0=x_u(\bA x)_u
   =d_ux_u^2+\sum_{v\in N(u)\cap F}c_{uv},
\]
and hence
\begin{equation}\label{eq:inside-diag}
 d_u=-x_u^{-2}\sum_{v\in N(u)\cap F}c_{uv}.
\end{equation}
Thus the internal diagonals are determined, the $n-|F|$ exterior diagonals are free, and the $e-e_\times(F)$ noncrossing edge variables are free and nonzero. These coordinate substitutions give the isomorphism \eqref{eq:incidence-product}, with regular inverse.

The product is nonempty exactly when every nonzero $k_u$ is at least two, equivalently when $F$ is a fort. For such $F$,
\begin{align*}
 \dim\cI_F
 &=|F|-1+e-e_\times(F)+n-|F|
     +\sum_{u\in\partial F}(k_u-1)\\
 &=n+e-|\partial F|-1,
\end{align*}
because $\sum_uk_u=e_\times(F)$. Smoothness and irreducibility follow from the same properties of the factors.
\end{proof}

The dimension of $\cD_F$ is obtained by subtracting the generic fiber dimension of $\pi$ from~\eqref{eq:incidence-dim}. We next compute that fiber through the restricted kernel on $F$.

\subsection{Support components and the dimension of the coefficient image}
Let $F_1,\ldots,F_r$ be the connected components of $G[F]$. For $u\in\partial F$, write
\[
 C(u)=\{i:N(u)\cap F_i\ne\varnothing\}.
\]

Each component $F_i$ supplies one multiple of an exact-support null vector that solves all internal rows. An exterior row gives a relation among the components it meets. Since a single row contributes at most one independent relation, its exterior vertex must be used as a label at most once.

\begin{definition}\label{def:labeled-forest}
A \emph{support-component forest} selects, for some exterior vertices $u$, one pair of distinct indices in $C(u)$, and regards it as an edge labeled by $u$. Each exterior label is used at most once, and the selected edges form a forest on $\{1,\ldots,r\}$. Define
\begin{equation}\label{eq:rho-h}
 \rho_F=\max\{\text{number of edges in such a forest}\},\qquad h_F=r-\rho_F.
\end{equation}
Parallel edges with different labels are allowed as candidates, but two parallel edges form a cycle and cannot both belong to the selected forest. A label with $|C(u)|=1$ supplies no edge. An empty forest is allowed.
\end{definition}

For example, let $F$ be the three leaves of a star. Its support graph has
three components, but the center supplies only one exterior label. Hence
$r(F)=3$, $\rho_F=1$, and $h_F=2$. A vector of exact support $F$ forces
all three leaf diagonals to vanish. The center row then gives one relation
among the three leaf coordinates, with all three coefficients nonzero.
Its two-dimensional kernel contains a vector with all coordinates nonzero:
none of the three coordinate hyperplanes contains the kernel, and their
union cannot cover it over $\CC$. Thus the restricted nullity is two and
the coefficient locus has codimension three, as
Theorem~\ref{thm:fort-codimension} predicts.

For every exterior vertex, define the subspace
\begin{equation}\label{eq:component-space}
 L_u=\{z\in\CC^r:\supp z\subseteq C(u),\ \sum_i z_i=0\}.
\end{equation}
It is generated by the differences $\mathbf e_i-\mathbf e_j$ for $i,j\in C(u)$.

The difference vectors generating $L_u$ are the rows of an incidence matrix on the support components. Their independent choices are precisely the labeled forests just defined.

\begin{lemma}\label{lem:row-forest}
The maximum rank of a matrix obtained by selecting one row in each $L_u$ is $\rho_F$.
\end{lemma}
\begin{proof}
Suppose selected rows $c_{u_1},\ldots,c_{u_s}$ have a nonzero $s$-minor. Write
\[
 c_{u_a}=\sum_{\{i,j\}\subseteq C(u_a)}
 \lambda_{a,ij}(\mathbf e_i-\mathbf e_j).
\]
By multilinearity, that minor is a sum of minors obtained by choosing one difference vector from each row. At least one summand is nonzero. The chosen labels $u_a$ are distinct.

For difference vectors, a cycle gives the relation
\[
 (\mathbf e_{i_1}-\mathbf e_{i_2})+\cdots+
 (\mathbf e_{i_k}-\mathbf e_{i_1})=0.
\]
Conversely, if their edges form a forest, the coordinate at a leaf forces the coefficient of its unique incident edge in any linear relation to vanish. Repeating at the remaining leaves proves independence. Thus the nonzero summand has $s$ forest edges, and $s\le\rho_F$.

Choose the difference vectors of a maximum labeled forest, and zero rows for unused exterior vertices. They have rank $\rho_F$ and belong to the required spaces. Hence the maximum rank equals $\rho_F$.
\end{proof}

The rank bound for exterior rows gives a lower bound for every restricted kernel. It remains to choose internal weights with no additional internal kernel and exterior weights attaining that rank bound.

\begin{theorem}\label{thm:fort-codimension}
For every fort $F\subseteq U$, each point $(\alpha,[x])\in\cI_F$ satisfies
\begin{equation}\label{eq:lower-nullity}
 \dim\ker\bA_{V,F}(\alpha)\ge h_F.
\end{equation}
Equality holds on a nonempty dense open subset of $\cI_F$. Moreover,
\begin{equation}\label{eq:fort-codimension}
 \codim_{\Omega_G}\cD_F=|\partial F|+h_F.
\end{equation}
\end{theorem}
\begin{proof}
Write $F_1,\ldots,F_r$ for the components of $G[F]$. At a point
$(\alpha,[x])\in\cI_F$, every coordinate $x_v$, $v\in F$, is nonzero.
Define an injection $\iota_x:\CC^r\to\CC^F$ by
$(\iota_x(a))_v=a_ix_v$ for $v\in F_i$. Each row in $F_i$ vanishes on this
vector because it is the scalar $a_i$ times the same row evaluated on $x$;
there are no edges between distinct support components. The remaining
possibly nonzero rows are
\begin{equation}\label{eq:grouped-rows}
 (\bA\iota_x(a))_u=\sum_{i=1}^rC_{u,i}a_i,
 \qquad C_{u,i}=\sum_{v\in N(u)\cap F_i}\bA_{uv}(\alpha)x_v
 \quad(u\in\partial F).
\end{equation}
Each grouped row belongs to $L_u$: it is supported on the components met by
$u$, and its total sum is $(\bA x)_u=0$. Lemma~\ref{lem:row-forest} gives
$\rank C\le\rho_F$, whence
\[
 \iota_x(\ker C)\subseteq\ker\bA_{V,F}(\alpha),\qquad
 \dim\ker\bA_{V,F}(\alpha)\ge r-\rho_F=h_F.
\]
To attain equality, set $x=\one_F$. Choose positive rational weights
on all edges internal to $F$, and let the diagonal at $v\in F$ be minus the
sum of those internal weights at $v$. For real $y$ on a component,
\[
 y^{\mathsf T}\bA_{F_i,F_i}y
 =-\sum_{vw\in E(F_i)}t_{vw}(y_v-y_w)^2.
\]
Thus its kernel consists of constant vectors. The same is true over $\CC$,
since the rank of this real matrix is the size of a nonzero minor and does
not change under scalar extension. Singleton components have the zero
one-by-one matrix and satisfy the same assertion. Hence
\[
 \ker\bA_{F,F}=\bigoplus_i\CC\one_{F_i}.
\]

For each $u\in\partial F$ let
\[
 H_u=\{(z_v)_{v\in N(u)\cap F}:\sum_vz_v=0\},
 \qquad g_u(z)_i=\sum_{v\in N(u)\cap F_i}z_v.
\]
The map $g_u:H_u\to L_u$ is onto: for a prescribed row $c\in L_u$, choose
one neighbour in each component met by $u$, give it value $c_i$, and set
the other coordinates to zero. Their total sum is zero. Therefore the
product map $g:\prod_uH_u\to\prod_uL_u$ is a surjective linear map of
affine spaces.

A forest attaining $\rho_F$ in Lemma~\ref{lem:row-forest} gives a nonzero
$\rho_F$-minor in the grouped-row coordinates. The nonvanishing of that
minor defines a nonempty open subset of $\prod_uL_u$, and its inverse image
under $g$ is nonempty and open. If $\rho_F=0$, use the constant minor one.
The subset in $\prod_uH_u$ where every crossing coordinate is nonzero is
also nonempty and open: on each factor one can use
$(1,\ldots,1,-(k_u-1))$, since $k_u=|N(u)\cap F|\ge2$.
The product $\prod_uH_u$ is irreducible, so these two opens intersect.
Choose the crossing weights in their intersection. Then every crossing weight is nonzero and $\rank C=\rho_F$. Give edges wholly outside $F$ arbitrary nonzero values and choose
outside diagonals arbitrarily. The equations on $F$ force any restricted
null vector to be componentwise constant, while the external rows impose
exactly $Ca=0$. Consequently
\[
 \ker\bA_{V,F}=\iota_{\one_F}(\ker C),\qquad
 \dim\ker\bA_{V,F}=h_F.
\]

To pass from this point to generic coefficients, put $s=|F|-h_F$. All
$(s+1)$-minors vanish on the coefficient image of $\cI_F$, by the lower
bound. They therefore vanish on its closure $\cD_F$. At the point just
constructed some $s$-minor $\mu$ is nonzero (use $\mu=1$ if $s=0$).
Thus
\[
 O=\cD_F\cap\{\mu\ne0\}
\]
is a nonempty open on which the restricted rank is exactly $s$. Its inverse
image in the irreducible variety $\cI_F$ is a nonempty dense open, proving
the generic equality in~\eqref{eq:lower-nullity}. The same coefficient open set $O$ determines the fibers of $\pi$.

The image of $\pi:\cI_F\to\cD_F$ is constructible and dense. Since
$\cD_F$ is irreducible, it contains a nonempty open subset $O'$. Restrict to
$O\cap O'$. Pivot elimination using $\mu$ makes the restricted kernels a
vector bundle of rank $h_F$. For every $\alpha$ there,
\[
 \pi^{-1}(\alpha)=
 \PP(\ker\bA_{V,F}(\alpha))\cap
 \{x_v\ne0\text{ for every }v\in F\}.
\]
This is nonempty because $\alpha\in O'$, and is open in a projective space
of dimension $h_F-1$. The generic fiber therefore has dimension $h_F-1$.
By the incidence dimension already computed in
Theorem~\ref{thm:incidence-product},
\[
 \begin{aligned}
 \dim\cD_F
 &=\dim\cI_F-(h_F-1)\\
 &=(n+e-|\partial F|-1)-(h_F-1)
  =n+e-|\partial F|-h_F.
 \end{aligned}
\]
Subtracting from $\dim\Omega_G=n+e$ proves~\eqref{eq:fort-codimension}.
\end{proof}

Connected support leaves only one componentwise multiple. Minimal fort support has the same effect, since two independent restricted null vectors would produce a smaller support by canceling one coordinate.

\begin{corollary}\label{cor:connected-minimal}
If $G[F]$ is connected, then $h_F=1$ and $\codim\cD_F=|\partial F|+1$. The same conclusion holds if $F$ is an inclusion-minimal fort, even when $G[F]$ is disconnected.
\end{corollary}
\begin{proof}
If $G[F]$ is connected, then
\[
 r(F)=1,\qquad\rho_F=0,\qquad h_F=1.
\]
Now suppose that $F$ is an inclusion-minimal fort. Choose an exact-support witness $x$. If $\dim\ker\bA_{V,F}\ge2$, choose a kernel vector $y\notin\CC x$ and $v\in F$. Then
\[
 z=x_vy-y_vx\ne0,\qquad z_v=0,\qquad
 \supp z\subsetneq F.
\]
Proposition~\ref{prop:support-witness} makes $\supp z$ a smaller fort, a contradiction. Thus the restricted nullity is one, and Theorem~\ref{thm:fort-codimension} gives $h_F=1$ and the stated codimension.
\end{proof}

\subsection{A second formula for the component nullity}
The integer $\rho_F$ can also be computed by the deficit of independent exterior equations on subsets of the labels. We derive this form because it describes $h_F$ without choosing a maximizing forest.
For $J\subseteq\partial F$, join all component indices within each $C(u)$, $u\in J$. Let $c_F(J)$ denote the number of connected classes of the resulting graph on all $r$ indices, including isolated ones.

For a family of row spaces, the sum of the spaces indexed by $J$ bounds how
many independent rows can be selected from those indices. The following is the
linear independent-representative form of Rado's theorem~\cite{Rado}.
Its min--max expression gives the second formula for $h_F$.

\begin{lemma}\label{lem:linear-representatives}
For subspaces $L_1,\ldots,L_m$ of a finite-dimensional vector space, the maximum number of independent vectors chosen from distinct $L_i$ is
\begin{equation}\label{eq:linear-representatives}
 \rho=\min_{J\subseteq\{1,\ldots,m\}}
 \left(m-|J|+\dim\sum_{j\in J}L_j\right).
\end{equation}
\end{lemma}
\begin{proof}
Write $L_J=\sum_{j\in J}L_j$. If independent vectors have been selected at the distinct indices in $I$, then for every $J\subseteq\{1,\ldots,m\}$,
\[
 |I|=|I\setminus J|+|I\cap J|
 \le m-|J|+\dim L_J.
\]
This proves the upper bound.

We first prove the following full-selection statement by induction on $m$: if $\dim L_J\ge|J|$ for every $J$, then all $m$ spaces admit independent representatives. The assertion for $m=0$ is empty, and for $m=1$ choose a nonzero vector in $L_1$.

Suppose $m>1$. If a nonempty proper subset $J$ satisfies $\dim L_J=|J|$, apply induction to its subfamily and choose $|J|$ independent representatives. They span $L_J$. For a set $K$ disjoint from $J$, the images of the other spaces in the quotient by $L_J$ satisfy
\[
 \dim\frac{L_J+L_K}{L_J}
 =\dim L_{J\cup K}-\dim L_J
 \ge |J|+|K|-|J|=|K|.
\]
Induction supplies independent representatives for the remaining spaces in that quotient. Lift them to their original spaces. A linear relation among all chosen representatives first projects to a zero relation in the quotient, killing the lifted coefficients; independence inside $L_J$ kills the rest.

If no such proper subset is tight, then $\dim L_J\ge|J|+1$ for every nonempty proper $J$. Choose $0\ne v_m\in L_m$ and quotient by its line. For a nonempty $J\subseteq\{1,\ldots,m-1\}$, passing to this quotient decreases $\dim L_J$ by at most one, leaving dimension at least $|J|$. Induction in the quotient, followed by lifts and $v_m$, again gives $m$ independent representatives. This completes the full-selection statement.

Let $k$ be the minimum on the right of~\eqref{eq:linear-representatives}. The choice $J=\varnothing$ gives $k\le m$, and every expression in the minimum is nonnegative. Introduce a vector space $E$ of dimension $m-k$, the same auxiliary summand for every index, and replace $L_i$ by $L_i\oplus E$. For each nonempty $J$,
\[
 \dim\sum_{j\in J}(L_j\oplus E)=\dim L_J+m-k\ge|J|.
\]
The full-selection result provides $m$ independent vectors in these enlarged spaces. Their $m$-dimensional span has a projection to the original ambient space whose kernel has dimension at most $m-k$. Hence its image has dimension at least $k$. It is spanned by the $m$ projected representatives, so a subset of at least $k$ of them is independent. Their original indices are distinct, and each belongs to its original $L_i$. This proves the required lower bound.
\end{proof}
For the spaces $L_u$ in \eqref{eq:component-space}, the dimension of a sum is determined by the connected classes joined by its labels.

\begin{proposition}\label{prop:deficiency}
The integer in \eqref{eq:rho-h} satisfies
\begin{equation}\label{eq:deficiency}
 h_F=\max_{J\subseteq\partial F}\bigl(c_F(J)-|\partial F\setminus J|\bigr).
\end{equation}
\end{proposition}
\begin{proof}
Within a connected class of $s$ indices, the differences $\mathbf e_i-\mathbf e_j$ span a space of dimension $s-1$. Summing over the $c_F(J)$ classes gives
\[
 \dim\sum_{u\in J}L_u=r-c_F(J).
\]
Lemmas~\ref{lem:row-forest} and \ref{lem:linear-representatives} now give
\begin{align*}
 \rho_F
 &=\min_{J\subseteq\partial F}
   \left(|\partial F\setminus J|+
          \dim\sum_{u\in J}L_u\right)\\
 &=\min_{J\subseteq\partial F}
   \bigl(|\partial F\setminus J|+r-c_F(J)\bigr),\\
 h_F=r-\rho_F
 &=\max_{J\subseteq\partial F}
   \bigl(c_F(J)-|\partial F\setminus J|\bigr).
\end{align*}
\end{proof}

\subsection{The prime ideals of support images}
The exact-support condition is removed from the equations only after the vector coordinates have been made invertible. This gives the ideal of the coefficient-image closure and permits a comparison of different supports.
Fix $r\in F$ and introduce variables $x_v$ for $v\in F$. In $R[x_v:v\in F]$ put
\begin{equation}\label{eq:support-prime}
 \mathfrak p_F=
 \left(\bigl(\bA_{V,F}x,\ x_r-1\bigr):
 \left(\prod_{v\in F}x_v\right)^\infty\right)\cap R.
\end{equation}
Saturation removes components on coordinate hyperplanes; contraction to $R$ eliminates the vector variables. These are finite polynomial-ideal operations. Laurent variables can equivalently be treated by adjoining $s_e$ with $s_et_e-1=0$.

The coordinate ring in Theorem~\ref{thm:incidence-product} is an integral domain. Its contraction to the coefficient ring is therefore prime. Taking the union over all possible supports recovers the full first rank-drop locus.

\begin{theorem}\label{thm:prime-decomp}
For every fort $F$, $\mathfrak p_F$ is prime and $V(\mathfrak p_F)=\cD_F$. Moreover,
\begin{equation}\label{eq:radical-support}
 \sqrt{I_m(\bA_{V,U})}
 =\bigcap_{\substack{F\subseteq U\\F\text{ fort}}}\mathfrak p_F
 =\bigcap_{\substack{F\subseteq U\\F\text{ fort},\ h_F=1}}\mathfrak p_F.
\end{equation}
Empty intersections are understood to equal $R$. The irreducible components of the reduced failure locus correspond to the inclusion-minimal primes among the last finite family, after repetitions are removed.
\end{theorem}
\begin{proof}
Fix the normalization coordinate $r\in F$, put $x_F=\prod_{v\in F}x_v$, and let $J_F=(\bA_{V,F}x,x_r-1)$ in $R[x_v:v\in F]$. The ring
\[
 B_F=R[x_v:v\in F]_{x_F}/J_F
\]
is the coordinate ring of the explicit product in Theorem~\ref{thm:incidence-product}. It is an integral domain. The kernel of the coefficient homomorphism $R\to B_F$ is
\[
 \ker(R\to B_F)=J_F R[x]_{x_F}\cap R
             =(J_F:x_F^\infty)\cap R=\mathfrak p_F.
\]
The equality with saturation follows directly from localization: a polynomial maps to zero exactly when some power of $x_F$ times it belongs to $J_F$. A kernel of a homomorphism to an integral domain is prime. The vanishing ideal of the closure of the image on coefficient space is that same kernel, so $V(\mathfrak p_F)=\cD_F$.

Now suppose $m=|U|>0$. A coefficient point belongs to $V(I_m(\bA_{V,U}))$ exactly when the column map has a nonzero kernel vector. Extend it by zero outside $U$ and take its support $F$. Proposition~\ref{prop:support-witness} makes $F$ a fort, and normalizing one of its nonzero coordinates gives a point of $\cI_F$ above the coefficient point. Conversely every such incidence point produces a deficient column map. Hence, on complex points,
\[
 V(I_m(\bA_{V,U}))=
 \bigcup_{\substack{F\subseteq U\\F\text{ fort}}}\pi(\cI_F).
\]
The set on the left is closed. Every image closure $\cD_F$ is contained in it, and every image is contained in its closure. These two inclusions give
\[
 V(I_m(\bA_{V,U}))=
 \bigcup_{\substack{F\subseteq U\\F\text{ fort}}}\cD_F.
\]
There are only finitely many subsets of $U$, so this is a finite union. The Nullstellensatz over $\CC$, applied in the Laurent coefficient ring, turns equality of these closed sets into the first radical intersection in~\eqref{eq:radical-support}.

To restrict the family to $h_F=1$, fix a coefficient point in the deficient locus and choose a nonzero kernel vector $x$ with inclusion-minimal support among vectors of this particular kernel. This is a finite minimization. If $y$ were a linearly independent vector supported in the same $F$, choose $v\in F$ and form
\[
 z=x_vy-y_vx.
\]
Since $x_v\ne0$, independence implies $z\ne0$; its $v$ coordinate is zero and its support is a proper subset of $F$. This contradicts the choice of $x$. Therefore $\dim\ker\bA_{V,F}=1$. The uniform lower bound and the fact that all row spaces annihilate $\one$ give
\[
 1\ge h_F=r(F)-\rho_F\ge1.
\]
Thus this point is already covered by an $h_F=1$ support. The same finite-union argument proves the second radical intersection.

Finally every $\cD_F$ is irreducible, so the irreducible components of their finite union are exactly the inclusion-maximal members after repetitions are removed. Their ideals are the inclusion-minimal primes in the stated family. If $U=\varnothing$, the column map has no nonzero kernel, $I_0=R$, and there is no nonempty support; the empty intersection convention gives the same result.
\end{proof}
In particular,
\begin{equation}\label{eq:image-inclusion}
 \cD_F\subseteq\cD_H\quad\Longleftrightarrow\quad
 \mathfrak p_H\subseteq\mathfrak p_F.
\end{equation}
Equation~\eqref{eq:radical-support} gives the prime decomposition of the reduced rank-drop locus; the original minors specify its determinantal scheme structure. For results concerning submaximal minors of full sparse symmetric matrices, see \cite{DK}. Example~\ref{ex:crossing} below will also show that the inclusion test \eqref{eq:image-inclusion} cannot be replaced by inclusion of the supports.

\subsection{Diagonal changes and kernel jumps}
The edge parameters must remain nonzero, but the diagonal parameters are unrestricted. Symmetry allows a change of one diagonal entry to remove exactly one kernel direction while preserving the remaining directions.
\begin{lemma}\label{lem:diagonal}
Let $A$ be a symmetric matrix over a field, and let $K=\ker A_{V,U}$. If some $x\in K$ has $x_v\ne0$ for $v\in U$, then for every nonzero $\lambda$,
\begin{equation}\label{eq:diagonal-kernel}
 \ker(A+\lambda E_{vv})_{V,U}=K\cap\{y:y_v=0\}.
\end{equation}
\end{lemma}
\begin{proof}
Extend $x$ by zero on $S$. Since $Ax=0$ and $A=A^{\mathsf T}$,
\[
 x^{\mathsf T}A_{V,U}=0.
\]
For $y\in\ker(A+\lambda E_{vv})_{V,U}$,
\[
 0=x^{\mathsf T}(A+\lambda E_{vv})_{V,U}y
   =\lambda x_vy_v.
\]
Thus $y_v=0$, and
\[
 A_{V,U}y=(A+\lambda E_{vv})_{V,U}y-\lambda y_v\mathbf e_v=0.
\]
Conversely, $A_{V,U}y=0$ and $y_v=0$ imply
$(A+\lambda E_{vv})_{V,U}y=0$.
\end{proof}

To apply this operation to every kernel dimension, we use the minors of $\bA_{V,U}$ that impose nullity at least $q$.

\begin{definition}
For $1\le q\le m$, the \emph{kernel-jump scheme} is
\begin{equation}\label{eq:jump-scheme}
 \cD_{S,q}=V(\mathfrak d_{S,q}),\qquad
 \mathfrak d_{S,q}=I_{m-q+1}(\bA_{V,U}).
\end{equation}
Set $\cD_{S,m+1}=\varnothing$. Its closed points have column-kernel dimension at least $q$. The parameter $q$ is the column nullity.
\end{definition}

Choosing independent coordinate functionals on a $k$-dimensional kernel allows Lemma~\ref{lem:diagonal} to be iterated exactly $k-q$ times.

\begin{proposition}\label{prop:nullity-deformation}
If $\dim\ker A_{V,U}=k\ge q\ge1$, there is a set $T\subseteq U$ of size $k-q$ such that
\[
 A(\eps)=A+\eps\sum_{v\in T}E_{vv}
\]
has kernel dimension exactly $q$ for every $\eps\ne0$, and $A(0)=A$. Thus, as reduced spaces,
\begin{equation}\label{eq:jump-dense}
 |\cD_{S,q}|=
 \overline{|\cD_{S,q}|\setminus|\cD_{S,q+1}|}.
\end{equation}
Every irreducible component of $\cD_{S,q}$ has generic nullity exactly $q$.
\end{proposition}
\begin{proof}
Put $K=\ker A_{V,U}$. The restrictions of the coordinate functionals span $K^*$. Choose indices $T=\{v_1,\ldots,v_{k-q}\}$ for which these restrictions are independent, and write
\[
 K_i=K\cap\{y_{v_1}=\cdots=y_{v_i}=0\}.
\]
Then
\[
 \dim K_i=k-i,\qquad
 y\mapsto y_{v_i}\text{ is nonzero on }K_{i-1}.
\]
For each fixed $\eps\ne0$, successive application of Lemma~\ref{lem:diagonal} gives
\[
 \ker\left(A+\eps\sum_{j=1}^iE_{v_jv_j}\right)_{V,U}=K_i.
\]
At $i=k-q$, the kernel therefore has dimension $q$. No edge entry changes. The resulting curve specializes at $\eps=0$ to $A$, proving \eqref{eq:jump-dense}.

Let $Y$ be an irreducible component of the reduced jump locus and remove its intersections with the other components. On this nonempty open subset, \eqref{eq:jump-dense} forces the exact-nullity-$q$ locus to be dense. Thus its generic nullity is $q$.
\end{proof}

The restricted matrix $\bA_{V,F}$ is unchanged by diagonal changes at vertices outside $F$. We may therefore remove kernel directions outside $F$ without removing its exact-support vectors.

\begin{proposition}\label{prop:full-generic-kernel}
On a dense open subset of $\cD_F$, the full unknown-column kernel has dimension $h_F$, is supported in $F$, and contains a vector of exact support $F$. If $h_F=1$, then
\[
 \cD_F\subseteq\cD_H\quad\Longrightarrow\quad F\subseteq H.
\]
\end{proposition}
\begin{proof}
By Theorem~\ref{thm:fort-codimension}, choose $\alpha\in\pi(\cI_F)$ for which
\[
 K_F:=\ker\bA_{V,F}(\alpha)\text{ has dimension }h_F
\]
and contains an exact-support vector $x$. Put $A^{(0)}=\bA(\alpha)$ and
$K_U^{(j)}=\ker A^{(j)}_{V,U}$, and regard $K_F$ as a subspace of
$K_U^{(0)}$ by zero extension. If $K_U^{(j)}$ is not supported in $F$,
choose $v_j\in U\setminus F$ whose coordinate functional is nonzero on
$K_U^{(j)}$, and put
\[
 A^{(j+1)}=A^{(j)}+\lambda_jE_{v_jv_j},\qquad \lambda_j\ne0.
\]
Lemma~\ref{lem:diagonal} gives
\[
 K_U^{(j+1)}=K_U^{(j)}\cap\{y_{v_j}=0\},\qquad
 \dim K_U^{(j+1)}=\dim K_U^{(j)}-1.
\]
Since $v_j\notin F$, the submatrix $A^{(j)}_{V,F}$ does not change.
Hence $K_F$, its dimension $h_F$, and the vector $x$ are preserved.
After at most $\dim K_U^{(0)}-h_F$ steps, the full kernel is supported
in $F$ and therefore equals $K_F$.

The closed minor conditions $\dim\ker\bA_{V,F}\ge h_F$ hold on $\cD_F$. The full kernel always contains that restricted kernel. At the constructed point its dimension is exactly $h_F$. A corresponding nonzero minor of size $|U|-h_F$ gives a nonempty open subset $O_1\subseteq\cD_F$ where the full kernel has dimension at most $h_F$, and hence exactly $h_F$. On $O_1$ the full and restricted kernels coincide by inclusion and equality of dimensions.

The morphism $\cI_F\to\cD_F$ has constructible image by the finite-type image theorem, and its image is dense by the definition of $\cD_F$. A dense constructible subset of an irreducible variety contains a nonempty open subset: write it as a finite union of locally closed sets; one of their closures must be the whole irreducible variety, and that locally closed set then contains an open subset of it. Thus there is a nonempty open $O_2\subseteq\pi(\cI_F)$. On the dense open $O_1\cap O_2$, the coinciding kernels contain a vector of exact support $F$, as required.

Now assume $h_F=1$ and $\cD_F\subseteq\cD_H$. At a point of this open subset the full kernel is the line $\CC x$ with $\supp x=F$. The minor condition $\ker\bA_{V,H}\ne0$ holds on $\pi(\cI_H)$ and therefore on its closure $\cD_H$. It gives a nonzero full kernel vector supported in $H$. This vector is a nonzero scalar multiple of $x$, so $F\subseteq H$.
\end{proof}

\subsection{Incidence spaces for higher nullity}
We parametrize subspaces annihilated by a single symmetric matrix.
Define
\begin{equation}\label{eq:grass-incidence}
 \mathcal J_{S,q}=
 \{(\alpha,W)\in\Omega_G\times\Gr(q,\CC^U):\bA(\alpha)W=0\}.
\end{equation}
On the Grassmannian this is the vanishing of the map from its tautological bundle to the trivial bundle $\CC^V$. Its fiber over $\alpha$ is
\[
 \Gr(q,\ker\bA_{V,U}(\alpha)),
\]
empty when the kernel is too small. The Grassmannian is projective, so this projection is proper and its image is the closed set $|\cD_{S,q}|$; see \cite[\stacks{089R}, \stacks{01W7}]{Stacks}.

A $q$-plane has a basis matrix with an invertible $q$-row minor. Normalizing that minor to $I_q$ gives finitely many affine charts, on each of which the common-kernel equations are polynomial.

\begin{proposition}\label{prop:grass-elimination}
For $P\subseteq U$ of cardinality $q$, let $X^{(P)}$ be a $V\times q$ matrix whose $P$-rows form $I_q$, whose $S$-rows are zero, and whose other rows consist of free chart variables. Then
\begin{equation}\label{eq:grass-elimination}
 \sqrt{\mathfrak d_{S,q}}
 =\bigcap_{\substack{P\subseteq U\\|P|=q}}
 \sqrt{\bigl(\bA X^{(P)}\bigr)\cap R}.
\end{equation}
The ideal inside the radical is contracted from the polynomial ring containing the chart variables.
\end{proposition}
\begin{proof}
Let $\mathcal J_{S,q}$ be the incidence of coefficient points $\alpha$ and
$q$-planes $W\subseteq\ker\bA_{V,U}(\alpha)$. A $q$-plane has a full-rank
basis matrix $X$ with zero $S$-rows. Some $q$-row block $X_P$ is invertible.
Replacing the basis by $XX_P^{-1}$ leaves its column space unchanged and
makes exactly those rows the identity. Conversely every chart matrix
$X^{(P)}$ has rank $q$, because its $P$-rows are $I_q$. Hence on this chart
\[
 W\subseteq\ker\bA_{V,U}
 \quad\Longleftrightarrow\quad \bA X^{(P)}=0.
\]
Put $E_P=(\bA X^{(P)})\cap R$, where contraction is taken from the polynomial
ring containing the remaining entries of $X^{(P)}$. Elimination identifies
$V(E_P)$ with the closure of the coefficient image of the chart. The
projection of the entire incidence is closed, since the Grassmannian is
projective; its image is precisely $|\cD_{S,q}|$. There are finitely many
charts, so
\[
 |\cD_{S,q}|=
 \overline{\bigcup_{|P|=q}\pi(\mathcal J_{S,q}|_P)}
 =\bigcup_{|P|=q}\overline{\pi(\mathcal J_{S,q}|_P)}
 =\bigcup_{|P|=q}V(E_P).
\]
Taking radical ideals over $\CC$ gives~\eqref{eq:grass-elimination}.
If there is no $q$-subset of $U$, both the incidence and rank-jump locus
are empty, consistently with the determinantal conventions.
\end{proof}

\subsection{First-order equations and examples}
At a point of exact nullity $q$, we compute the tangent space to the
determinantal scheme defined by the minors, before passing to its reduced
support. The first variation is a map from the kernel to the cokernel.
\begin{proposition}\label{prop:tangent}
At a point $\alpha$ of exact nullity $q$, set
\[
 K=\ker\bA_{V,U}(\alpha),\qquad
 C=\coker\bA_{V,U}(\alpha).
\]
Define the linear map on coefficient variations
\[
 \chi_\alpha:T_\alpha\Omega_G\longrightarrow\Hom_{\CC}(K,C),
 \qquad \chi_\alpha(\dot A)(x)=[\dot A_{V,U}x].
\]
The Zariski tangent space of the determinantal scheme is
\begin{equation}\label{eq:tangent}
 T_\alpha\cD_{S,q}=\ker\chi_\alpha.
\end{equation}
\end{proposition}
\begin{proof}
Let $B_0=\bA_{V,U}(\alpha)$ have rank $r=m-q$. A vector in
$T_\alpha\Omega_G$ varies the diagonal and edge coefficients while
preserving the symmetric zero pattern.

Choose an invertible $r$-minor of $B_0$ and use its row and column order. On the resulting coefficient open set write
\[
 B=\begin{pmatrix}P&Q_0\\R_0&D_0\end{pmatrix},\qquad
 T=D_0-R_0P^{-1}Q_0.
\]
The invertible row and column changes
\[
 \begin{pmatrix}P^{-1}&0\\-R_0P^{-1}&I\end{pmatrix}
 B
 \begin{pmatrix}I&-P^{-1}Q_0\\0&I\end{pmatrix}
 =\begin{pmatrix}I&0\\0&T\end{pmatrix}
\]
show that $I_{r+1}(B)$ is generated on this chart by the entries of $T$. At $\alpha$, $T=0$. The kernel and cokernel identifications at that point are
\[
 \CC^q\longrightarrow K,\quad
 u\longmapsto\binom{-P^{-1}Q_0u}{u},
\]
\[
 C\longrightarrow\CC^{n-r},\quad
 [(y,z)]\longmapsto z-R_0P^{-1}y.
\]
All unmarked block entries in these two maps are evaluated at $\alpha$.

Over $\CC[\eps]/(\eps^2)$,
\[
 (P+\eps\dot P)^{-1}
 =P^{-1}-\eps P^{-1}\dot P P^{-1},
\]
so the first variation of the Schur complement is
\[
 \dot T=\dot D_0-\dot R_0P^{-1}Q_0
       +R_0P^{-1}\dot P P^{-1}Q_0-R_0P^{-1}\dot Q_0.
\]
Applying the displayed cokernel map to
$\dot B\binom{-P^{-1}Q_0u}{u}$ gives precisely $\dot T u$.
Hence the linearized minor equations are
\[
 \dot T=0
 \Longleftrightarrow\chi_\alpha(\dot A)=0.
\]
This proves~\eqref{eq:tangent}. If $r=0$, the empty pivot block has the
usual empty inverse and the calculation reduces to $T=B$.
\end{proof}

The next example separates three notions: inclusion of forts, inclusion of their coefficient images, and an increase in kernel dimension. Two distinct support components can meet singularly even when the nullity remains one.

\begin{example}\label{ex:crossing}
Take seeds $S=\{s,t\}$, unknown vertices $U=\{1,2,3\}$, and edges $s1,s2,s3,t1,t2$ with respective unit coefficients $a,b,h,c,d$. Write the unknown diagonals as $\delta_1,\delta_2,\delta_3$. Then
\[
 B_S=\begin{pmatrix}
 \delta_1&0&0\\0&\delta_2&0\\0&0&\delta_3\\a&b&h\\c&d&0
 \end{pmatrix},\qquad \Delta=ad-bc.
\]
Three maximal minors are unit multiples of $\delta_1$, $\delta_2$, and $\delta_3\Delta$. Expanding the other minors shows that they belong to the resulting ideal. Hence
\begin{equation}\label{eq:crossing-ideal}
 I_3(B_S)=(\delta_1,\delta_2,\delta_3\Delta)
 =(\delta_1,\delta_2,\delta_3)\cap(\delta_1,\delta_2,\Delta).
\end{equation}
For the last equality, quotient by $(\delta_1,\delta_2)$. The independent variable $\delta_3$ and the prime element $\Delta$ are relatively prime, so their principal ideals intersect in their product. The relation $\Delta=0$ is prime in the Laurent ring since it eliminates $a=bc/d$.

The only forts in $U$ are $F=\{1,2\}$ and $H=\{1,2,3\}$. A witness of support $F$ is a multiple of $(b,-a,0)$ and requires $\delta_1=\delta_2=\Delta=0$. A witness of support $H$ is a multiple of $(-dh,ch,\Delta)$ and requires all three diagonals to vanish and $\Delta\ne0$. Therefore the two primes in \eqref{eq:crossing-ideal} are exactly $\mathfrak p_H$ and $\mathfrak p_F$. They are incomparable even though $F\subsetneq H$. In particular, the nonminimal fort $H$ contributes a component.

The $2$-minor on rows $(s,t)$ and columns $(1,3)$ equals $-ch$, a unit. Thus $I_2(B_S)=R$ and $\cD_{S,2}=\varnothing$. Nevertheless, at $\delta_1=\delta_2=\delta_3=\Delta=0$ the reduced failure variety has a crossing. Indeed $\partial\Delta/\partial a=d$ is a unit, so locally its equations are two coordinate equations and $xy=0$. The intersection is singular at constant nullity one. Together with the nested supports above, this separates support minimality from the geometry of coefficient-image components; compare~\cite{Compatible}.
\end{example}

For a star whose center belongs to $S$, the centre row is a single nonzero linear relation among the leaf coordinates. The higher jump ideals can consequently be computed in every degree.

\begin{example}\label{ex:star}
Let $S$ be the center of a star with $\ell$ leaves. Then
\[
 B_S=\begin{pmatrix}\diag(\delta_1,\ldots,\delta_\ell)\\t_1\ \cdots\ t_\ell\end{pmatrix},
 \qquad t_i\in R^\times.
\]
If precisely $z$ leaf diagonals vanish, only those $z$ coordinates can support a kernel vector, and they satisfy one nonzero linear equation. Thus the nullity is $\max(z-1,0)$.

For $1\le q\le\ell-1$, expansion of minors containing the last row gives
\begin{equation}\label{eq:star-ideal}
 \mathfrak d_{S,q}
 =\left(\prod_{i\in J}\delta_i:|J|=\ell-q\right)
 =\bigcap_{\substack{T\subseteq\{1,\ldots,\ell\}\\|T|=q+1}}
 (\delta_i:i\in T).
\end{equation}
To verify the ideal equality, every product of $\ell-q$ distinct diagonals occurs in a minor with a remaining leaf column supplying a unit $t_j$. Minors not using the last row are higher-degree products and add no generators. A monomial belongs to every coordinate prime on the right exactly when its set of variables meets every $(q+1)$-set, equivalently when it contains at least $\ell-q$ distinct variables. This is exactly divisibility by a generator on the left. The ideals are monomial, so this proves equality. For $q=\ell$, the $1$-minor ideal contains a unit and the jump scheme is empty.
\end{example}

\section{Common kernels and coefficient sheaves}\label{sec:common-kernels}

Fix a subspace $W\subseteq\kk^V$. The matrices satisfying $AW=0$ form
a linear coefficient space. We describe this space by a sheaf of local
row solutions and then by equilibrium stresses and exterior dependencies.

\subsection{The coefficient sheaf of a common kernel}\label{sec:coefficients}

For fixed $W$, the equations in \eqref{eq:grass-incidence} are linear
in each row's coefficients. The shared coefficients at the two ends of an
edge must agree. We impose this agreement on the face space of $G$.

In the rest of this section, $\kk$ is an infinite characteristic-zero field and $0\ne W\subseteq\kk^V$ has dimension $q$. Prescribed seeds mean $W\subseteq\kk^{V\setminus S}$, with zero extension to $V$.

Let $\ell_v\in W^*$ be the restriction of the $v$th coordinate functional.

The $u$th row annihilates $W$ precisely when $d\ell_u+\sum_{v\sim u}t_{uv}\ell_v=0$ in $W^*$. We take that local solution space as the vertex stalk and use the shared edge coefficient as its restriction.

\begin{definition}\label{def:coefficient-sheaf}
The \emph{coefficient sheaf} on the face space of the original graph $G$ is
\begin{equation}\label{eq:coefficient-stalk}
 \sC_W(u)=\left\{(d,(t_{uv})_{v\sim u}):
 d\ell_u+\sum_{v\sim u}t_{uv}\ell_v=0\text{ in }W^*\right\},
 \qquad \sC_W(e)=\kk.
\end{equation}
The restriction to an incident edge extracts its coordinate $t_{uv}$.
\end{definition}
The topology and sheaf axiom are those in Definition~\ref{def:cellular}. This definition uses $W$ itself and is independent of a basis. Edge coefficients are allowed to vanish in the linear stalks; excluding their zero coordinate hyperplanes is a separate open condition.

Choose a basis matrix $X$ of $W$, with $v$th row $X_v\in\kk^q$. A coefficient vector $a\in\kk^{n+e}$ determines a symmetric matrix $A(a)$ with zero entries at nonedges. Define
\begin{equation}\label{eq:theta}
 \Theta_X:\kk^{n+e}\longrightarrow\kk^{nq},
 \qquad a\longmapsto\vecop(A(a)X).
\end{equation}
Changing the basis of $W$ multiplies the equations by an invertible matrix and does not change the coefficient solution space.

Agreement of the two restrictions on $uv$ is the symmetry condition $t_{uv}=t_{vu}$. Thus global sections can be computed either from the stalks or from the single matrix $\Theta_X$.

\begin{proposition}\label{prop:coefficient-global}
There are natural identifications
\begin{equation}\label{eq:coefficient-global}
 H^0(G;\sC_W)=\ker\Theta_X
 =\{A=A^{\mathsf T}:A_{uv}=0\text{ at nonedges},\ AW=0\}.
\end{equation}
Put $r_u=\rank X_{N[u]}$. Then
\begin{equation}\label{eq:coefficient-H1}
 \dim H^1(G;\sC_W)=\sum_{u\in V}r_u-\rank\Theta_X.
\end{equation}
An allowed realization of $W$ exists precisely when this common coefficient space contains a point with all edge coordinates nonzero.
\end{proposition}
\begin{proof}
A family of vertex sections satisfies
\[
 d_uX_u+\sum_{v\sim u}t_{uv}X_v=0\quad(u\in V).
\]
It is a global section if and only if $t_{uv}=t_{vu}$ for every edge. Hence
\[
 H^0(G;\sC_W)
 =\{a:A(a)X=0\}=\ker\Theta_X,
\]
which proves \eqref{eq:coefficient-global}.

The local coefficient map has rank $r_u=\rank X_{N[u]}$. Therefore
\[
 \dim\sC_W(u)=1+\deg(u)-r_u,
\]
and
\begin{align*}
 \dim C^0&=n+2e-\sum_ur_u,&
 \dim C^1&=e,\\
 \dim H^0&=n+e-\rank\Theta_X.&
\end{align*}
The Euler identity gives
\[
 \dim H^1=\dim H^0-\dim C^0+\dim C^1
 =\sum_ur_u-\rank\Theta_X.
\]
Finally, the prescribed graph pattern requires exactly that the edge coordinates of a global section be nonzero.
\end{proof}

An edge that is zero on every vector of $\ker\Theta_X$ prevents realization. Over an infinite field this is the only additional obstruction: finitely many proper edge-coordinate hyperplanes cannot cover the solution space.

\begin{proposition}\label{prop:theta-rank}
For an edge $e$, let $\epsilon_e^*$ be its coefficient coordinate functional on $\kk^{n+e}$. Then $W$ has an allowed symmetric realization if and only if
\begin{equation}\label{eq:theta-rank}
 \rank\begin{pmatrix}\Theta_X\\\epsilon_e^*\end{pmatrix}
 =\rank\Theta_X+1\qquad(e\in E).
\end{equation}
\end{proposition}
\begin{proof}
Let $K=\ker\Theta_X$. For a finite-dimensional linear map, the annihilator $K^\perp$ equals its row space. Indeed, each row vanishes on $K$, and both spaces have dimension $\rank\Theta_X$ by rank-nullity. Thus for each edge $e$,
\[
 \epsilon_e^*|_K\ne0
 \ \Longleftrightarrow\ \epsilon_e^*\notin\operatorname{row}(\Theta_X)
 \ \Longleftrightarrow\
 \rank\begin{pmatrix}\Theta_X\\\epsilon_e^*\end{pmatrix}
       =\rank\Theta_X+1.
\]
An allowed coefficient vector lies in $K$ and has nonzero value under every $\epsilon_e^*$, so it implies these conditions.

Conversely, choose coordinates on the vector space $K\cong\kk^d$. Under the rank conditions every restricted coordinate $\epsilon_e^*|_K$ is a nonzero linear polynomial. Their product is nonzero in the polynomial ring $\kk[K]$, which is an integral domain:
\[
 f=\prod_{e\in E}\epsilon_e^*|_K\ne0.
\]
Over an infinite field a nonzero polynomial cannot vanish at every point. One proves this by induction on $d$: expand in the last variable, choose the first $d-1$ coordinates where a nonzero coefficient does not vanish, then avoid the finitely many roots in the last coordinate. Choose $a\in K$ with $f(a)\ne0$. Its matrix is symmetric, annihilates $W$, vanishes at nonedges, and has all edge coefficients nonzero, as required. If $E=\varnothing$, the product is one and the zero matrix is allowed. 
\end{proof}

\subsection{Local coefficient solutions and forts}
For a single vector, Proposition~\ref{prop:support-witness} gives a global symmetric realization from its fort support. For a subspace $W$, the same support condition on every vector first controls only the local coefficient rows.
\begin{theorem}\label{thm:local-forts}
The following are equivalent:
\begin{enumerate}
\item every nonzero vector of $W$ has fort support;
\item at each vertex $u$, the stalk $\sC_W(u)$ contains a row with all incident edge coefficients nonzero;
\item for every incidence $u\sim v$,
\begin{equation}\label{eq:local-span}
 \ell_v\in\spanlin\bigl(\ell_u,\{\ell_w:w\in N(u)\setminus\{v\}\}\bigr).
\end{equation}
\end{enumerate}
\end{theorem}
\begin{proof}
Fix an incidence $u\sim v$. A coefficient vector in $\sC_W(u)$ with
$t_{uv}\ne0$ satisfies
\[
 \ell_v=-\frac{d}{t_{uv}}\ell_u
 -\sum_{w\in N(u)\setminus\{v\}}
                    \frac{t_{uw}}{t_{uv}}\ell_w.
\]
Conversely any such expression supplies a local row by taking $t_{uv}=1$
and moving the other terms to the left. Thus the $uv$ coordinate is not
identically zero on the local solution space exactly when
\eqref{eq:local-span} holds. If this is true for every incident edge, the
product of the corresponding nonzero coordinate linear functions is a
nonzero polynomial on that finite-dimensional solution space. Since $\kk$
is infinite, it has a nonzero value somewhere. At that one row all incident
edge coefficients are nonzero. This proves the equivalence of (b) and (c),
including an isolated vertex, where the empty product is one.

For (a), suppose first that \eqref{eq:local-span} fails and put
\[
 L=\spanlin\bigl(\ell_u,\{\ell_w:w\in N(u)\setminus\{v\}\}\bigr)
 \subseteq W^*.
\]
The image of $\ell_v$ in $W^*/L$ is nonzero. Choose a linear functional on
$W^*/L$ nonzero at that image and pull it back to $W^*$. The finite-dimensional
identification $W^{**}=W$ gives $x\in W$ with
\[
 \ell_u(x)=0,\quad \ell_w(x)=0\ (w\in N(u)\setminus\{v\}),
 \quad\ell_v(x)\ne0.
\]
Thus $x\ne0$ and its support has the unique exterior neighbour incidence
$N(u)\cap\supp x=\{v\}$, so it is not a fort. Conversely, a nonfort support
of a nonzero $x\in W$ gives exactly these equations at its witnessing
incidence, and evaluation on $x$ rules out \eqref{eq:local-span}. This proves
(a) equivalent to (c).
\end{proof}
The coefficients obtained in Theorem~\ref{thm:local-forts} need not agree at the two ends of an edge. The remaining condition is the existence of a global section with all edge coordinates nonzero. See \cite{Compatible} for compatible support families and the distinction between symmetric and combinatorially symmetric realizations.

For $W=\kk x$, the image of $a\mapsto A(a)x$ determines the dimension
of $H^1(G;\sC_W)$.

\begin{proposition}\label{prop:q1-H1}
If $q=1$, then $H^1(G;\sC_W)=0$, whether or not the nonzero vector spanning $W$ has fort support.
\end{proposition}
\begin{proof}
Let $W=\kk x$ and $F=\supp x$. The local rank is
\[
 r_u=\begin{cases}1,&u\in F\cup\partial F,\\0,&u\notin F\cup\partial F.
 \end{cases}
\]
For $u\in F$, varying only $d_u$ changes $A(a)x$ by a multiple of
$x_u\mathbf e_u$, with $x_u\ne0$. For $u\in\partial F$, choose $v\in N(u)\cap F$. Varying only the coefficient of $uv$ changes the output by
\[
 x_v\mathbf e_u+x_u\mathbf e_v=x_v\mathbf e_u.
\]
Every other output coordinate is identically zero. Hence
\[
 \im\Theta_X=\kk^{F\cup\partial F},\qquad
 \rank\Theta_X=|F|+|\partial F|=\sum_ur_u.
\]
Equation~\eqref{eq:coefficient-H1} now gives $H^1(G;\sC_W)=0$.
\end{proof}
By Proposition~\ref{prop:theta-rank}, realization is equivalent to
\[
 H^0(G;\sC_W)\cap\{t_e\ne0\text{ for all }e\}\ne\varnothing.
\]

\subsection{Equilibrium stresses and exterior dependencies}\label{sec:stress}
A subspace $W$ contains a vector nonzero at every vertex occurring in any of its supports. Rescaling by this vector turns one basis vector into a constant on the support, so that the other row equations become equilibrium equations.
Let
\[
 F=\bigcup_{x\in W}\supp x.
\]
Choose $x^0\in W$ nonzero at every $v\in F$. Such a vector exists by the finite-hyperplane argument, since each coordinate functional indexed by $F$ is nonzero. Define $D_{vv}=x^0_v$ for $v\in F$ and $D_{vv}=1$ outside $F$. Then
\begin{equation}\label{eq:normalization}
 W'=D^{-1}W,\qquad A'=DAD,
 \qquad AW=0\Longleftrightarrow A'W'=0.
\end{equation}
This is an invertible congruence on coefficient matrices, preserving symmetry and the nonzero edge pattern. The normalized space contains $\one_F$. Choose its basis matrix with rows
\begin{equation}\label{eq:framework-rows}
 X_v=(1,p_v)\quad(v\in F),\qquad X_v=0\quad(v\notin F),
 \qquad p_v\in\kk^{q-1}.
\end{equation}

Write
\[
 \mathcal S_G(W)=
 \{A=A^{\mathsf T}:A_{uv}=0\text{ at nonedges},\ AW=0\}.
\]
Congruence by $D$ identifies $\mathcal S_G(W)$ with
$\mathcal S_G(W')$. We keep $A',d'_u,t'_{uv}$ for the normalized matrix
and its entries, and $X$ for the basis matrix of $W'$ in
\eqref{eq:framework-rows}. The original matrix is $D^{-1}A'D^{-1}$.

For $u\in F$, the first coordinate of $(A'X)_u=0$ determines $d'_u$.
Substitution in the other coordinates gives the equilibrium equation
$\sum_vt'_{uv}(p_v-p_u)=0$.

\begin{definition}\label{def:rigidity}
Put $G_F=G[F]$. The rigidity matrix $\mathcal R_{G_F}(p)$ has one row
for each edge $uv\in E(G_F)$ and a $(q-1)$-column block for each vertex
of $F$. Its edge row has $p_u-p_v$ in the $u$ block, $p_v-p_u$ in the
$v$ block, and zero elsewhere. Its equilibrium stress space is
$\ker\mathcal R_{G_F}(p)^{\mathsf T}$. For $e\in E(G_F)$, the matrix
$\mathcal R_{G_F-e}(p)$ is obtained by deleting the row indexed by $e$.
Write $N_F(u)=N(u)\cap F$ for every $u\in V$. For $u\in\partial F$,
let $X_{N_F(u)}$ have rows $(1,p_v)$, $v\in N_F(u)$.
\end{definition}
The rigidity matrix and its stress space are used in \cite{CGT}; related
sheaf and cosheaf constructions appear in \cite{CGH}. Here each exterior
vertex contributes a dependency among the rows of $X_{N_F(u)}$.

\begin{theorem}\label{thm:stress}
In the normalized coordinates \eqref{eq:normalization}, there is a linear
isomorphism
\begin{equation}\label{eq:stress-decomp}
\begin{split}
 \mathcal S_G(W')
 \cong{}&\ker\mathcal R_{G_F}(p)^{\mathsf T}\\
 &\oplus\bigoplus_{u\in\partial F}\ker X_{N_F(u)}^{\mathsf T}\\
 &\oplus\kk^{|V\setminus F|+|E(G[V\setminus F])|}.
\end{split}
\end{equation}
Composing its inverse with $A'\mapsto D^{-1}A'D^{-1}$ parametrizes
$\mathcal S_G(W)$. This space contains a matrix with every edge coefficient
nonzero if and only if
\begin{equation}\label{eq:stress-edge}
 \rank\mathcal R_{G_F}(p)=\rank\mathcal R_{G_F-e}(p)
 \quad(e\in E(G_F))
\end{equation}
and, for each $u\in\partial F$ and $v\in N_F(u)$,
\begin{equation}\label{eq:exterior-affine}
 (1,p_v)\in\spanlin\{(1,p_w):w\in N_F(u)\setminus\{v\}\}.
\end{equation}
\end{theorem}
\begin{proof}
For $u\in F$, the first and remaining coordinates of $(A'X)_u=0$ are
\[
 d'_u+\sum_{v\in N_F(u)}t'_{uv}=0,\qquad
 d'_up_u+\sum_{v\in N_F(u)}t'_{uv}p_v=0.
\]
Eliminating $d'_u$ gives
\[
 d'_u=-\sum_{v\in N_F(u)}t'_{uv},\qquad
 \sum_{v\in N_F(u)}t'_{uv}(p_v-p_u)=0.
\]
Thus the internal edge coefficients form
$\ker\mathcal R_{G_F}(p)^{\mathsf T}$ and determine every internal diagonal.

For $u\notin F$, one has $X_u=0$, so
\[
 (A'X)_u=\sum_{v\in N_F(u)}t'_{uv}(1,p_v)=0.
\]
When $u\in\partial F$, this says
$t'_{u,F}\in\ker X_{N_F(u)}^{\mathsf T}$; otherwise the equation is zero.
Each crossing edge has a unique exterior endpoint, so these variable sets
are disjoint for distinct $u$. Exterior diagonals and edges in
$G[V\setminus F]$ do not appear in $A'X$ and give the last summand of
\eqref{eq:stress-decomp}.

Conversely, take a stress $\sigma$, a vector
$\zeta_u\in\ker X_{N_F(u)}^{\mathsf T}$ for each $u\in\partial F$, and
arbitrary exterior coefficients. Define the normalized symmetric matrix by
\[
 t'_{vw}=\sigma_{vw}\quad(vw\in E(G_F)),\qquad
 t'_{uv}=t'_{vu}=(\zeta_u)_v
       \quad(u\in\partial F,\ v\in N_F(u)),
\]
\[
 d'_v=-\sum_{w\in N_F(v)}\sigma_{vw}\quad(v\in F).
\]
Every coefficient is now assigned once. The stress equation verifies the
rows in $F$, and $X_{N_F(u)}^{\mathsf T}\zeta_u=0$ verifies the exterior
rows. Hence $A'X=0$. This construction is linear and inverse to coefficient
extraction. Since
\[
 (D^{-1}A'D^{-1})W=D^{-1}A'W'=0,
\]
it also gives the stated parametrization of $\mathcal S_G(W)$.

For an internal edge $e$, a stress with nonzero $e$-coordinate exists
precisely when
\[
 \mathcal R_{G_F}(p)_{e,*}
 \in\spanlin\{\mathcal R_{G_F}(p)_{f,*}:f\ne e\},
\]
which is \eqref{eq:stress-edge}. The same row-dependence argument for
$X_{N_F(u)}$ gives \eqref{eq:exterior-affine} at a crossing edge.
If all these conditions hold, the finite-hyperplane argument of
Proposition~\ref{prop:theta-rank} supplies a vector with all required
coordinates nonzero in each summand. Choose nonzero values on the free
exterior edges. The inverse construction gives an allowed $A'$, and
congruence by $D^{-1}$ preserves the nonzero edge pattern. Necessity follows
by extracting these same coordinates from any allowed realization.
\end{proof}
When $q=1$, the rigidity matrix has no columns, so the internal
condition holds. Each exterior row is the scalar $1$; hence
\eqref{eq:exterior-affine} is equivalent to $|N_F(u)|\ne1$.
This recovers Proposition~\ref{prop:support-witness}.

\subsection{One-dimensional stalks and multiplicative transport}
If each coefficient stalk is a line and every incident restriction is nonzero, choosing a local generator reduces all gluing equations to scalar ratios. The product of these ratios around a cycle is then the obstruction to a nonzero global section.
\begin{proposition}\label{prop:line-transport}
Suppose $G$ is connected, each vertex stalk of $\sC_W$ is one-dimensional, and every restriction to an incident edge is nonzero. Choose a generator $f_u$ of each vertex stalk and denote its edge coordinate by $f_{u,e}\in\kk^\times$. For an oriented edge $e:u\to v$, set
\[
 r_e=f_{u,e}/f_{v,e}.
\]
A global coefficient section with every edge coefficient nonzero exists if and only if
\begin{equation}\label{eq:transport}
 r_e=\lambda_v/\lambda_u\quad(e:u\to v)
\end{equation}
for some $\lambda_u\in\kk^\times$. Equivalently, the product of $r_e$ along every oriented closed walk, using inverse factors when an edge is traversed backwards, is one.
\end{proposition}
\begin{proof}
Write the local section at $u$ as $\lambda_uf_u$. On an edge $e:u\to v$,
\[
 \lambda_uf_{u,e}=\lambda_vf_{v,e}
 \quad\Longleftrightarrow\quad
 \frac{\lambda_v}{\lambda_u}
 =\frac{f_{u,e}}{f_{v,e}}=r_e
\]
when $\lambda_u,\lambda_v\ne0$. Multiplying along a closed walk gives a product of one by cancellation.

Conversely, choose a root $u_0$, a spanning tree, and $\lambda_{u_0}\in\kk^\times$. For its unique tree path $\gamma_{u_0,u}$, define
\[
 \lambda_u=\lambda_{u_0}\prod_{e\in\gamma_{u_0,u}}r_e^{\epsilon(e)},
\]
where $\epsilon(e)=1$ for the chosen orientation and $-1$ for the opposite orientation. This solves every tree-edge equation. Each nontree edge completes a closed walk, and the product-one assumption solves its equation as well. Thus the coefficients glue and are all nonzero.

Since every edge restriction is nonzero and $G$ is connected, $\lambda_u=0$ at one vertex of a linear global section implies $\lambda_v=0$ at all vertices.
\end{proof}
Under the hypotheses of Proposition~\ref{prop:line-transport}, remove zero
from every vertex and edge stalk of $\sC_W$. The resulting stalks, with
induced restrictions and scalar multiplication, form a torsor $\mathcal T_W$
under the constant sheaf $\underline{\kk^\times}$ on the face space of $G$.
A global section is a compatible choice of nonzero local coefficients.
With local generators fixed, its transition functions are the edge ratios
$r_e$. Replacing $f_u$ by $c_uf_u$ replaces $r_e$ by
$(c_u/c_v)r_e$, a vertex coboundary. Thus
$[r]\in H^1(G;\underline{\kk^\times})$ is intrinsic, and it vanishes
exactly when $\mathcal T_W$ has a global section
\cite[\stacks{03AG}]{Stacks}.

We give a subspace $W$ for which every nonzero vector has fort support, while the multiplicative transport class is nontrivial. Thus the local criterion of Theorem~\ref{thm:local-forts} does not imply a common symmetric realization.

\begin{example}\label{ex:cube}
On vertices $0,\ldots,7$ take the cube edges
\[
 01,12,23,30,\quad04,17,26,35,\quad45,56,67,74.
\]
Let $X_u=(1,u,u^2)$ and $W=\operatorname{col}X$ over $\QQ$. Every $0\ne x\in W$ is the evaluation of a nonzero polynomial of degree at most two at eight distinct numbers. Therefore
\[
 |V\setminus\supp x|\le2,
 \qquad
 u\notin\supp x\ \Longrightarrow\
 |N(u)\cap\supp x|\ge3-1=2.
\]
Thus every nonzero vector of $W$ has fort support.

The local equation at $u$ is
\[
 d_u+\sum_vt_{uv}=0,\quad
 ud_u+\sum_vvt_{uv}=0,\quad
 u^2d_u+\sum_vv^2t_{uv}=0.
\]
Eliminating $d_u$ and using the first resulting equation gives the equivalent pair
\[
 \sum_{v\in N(u)}(v-u)t_{uv}=0,\qquad
 \sum_{v\in N(u)}(v-u)^2t_{uv}=0.
\]
For distinct neighbors $a,b$ of $u$, the corresponding two-column minor is
\[
 \det\begin{pmatrix}a-u&b-u\\(a-u)^2&(b-u)^2\end{pmatrix}
 =(a-u)(b-u)(b-a)\ne0.
\]
Thus the two equations have rank two in the three edge variables, and the
local coefficient space is one-dimensional. A nonzero generator has
\begin{equation}\label{eq:cube-weights}
 f_{uv}=\frac{1}{(v-u)\prod_{w\in N(u)\setminus\{v\}}(v-w)}.
\end{equation}
To verify \eqref{eq:cube-weights}, write $N(u)=\{a,b,c\}$ and put $s_v=(v-u)f_{uv}$. With $D=(a-b)(a-c)(b-c)$,
\[
 D\sum_vs_v=(b-c)-(a-c)+(a-b)=0,
\]
\[
 D\sum_vvs_v=a(b-c)-b(a-c)+c(a-b)=0.
\]
Hence $\sum_v(v-u)f_{uv}=0$ and
$\sum_v(v-u)^2f_{uv}=\sum_v(v-u)s_v=0$. The vertices are distinct, so no denominator vanishes.

Symmetry would require $\lambda_v/\lambda_u=f_{uv}/f_{vu}$. On the square $0\to1\to2\to3\to0$, these ratios are
\[
 -\frac73,\qquad1,\qquad-1,\qquad\frac15.
\]
Their product is $7/15\ne1$. Proposition~\ref{prop:line-transport} excludes a nonzero global coefficient section. Since the cube is connected and all local restrictions are nonzero, its only global coefficient section is zero. With eight one-dimensional vertex stalks and twelve edge stalks,
\[
 H^0(G;\sC_W)=0,\qquad \dim H^1(G;\sC_W)=12-8=4.
\]
\end{example}

For the following triangle, the transport class vanishes and
$\dim H^1(G;\sC_W)=1$.

\begin{example}\label{ex:triangle-H1}
Let $G=K_3$, and let $X$ have rows $(1,0),(1,1),(1,2)$. For $s=(1,-2,1)^{\mathsf T}$ set
\[
 A=-\tfrac12ss^{\mathsf T}
 =\begin{pmatrix}-1/2&1&-1/2\\1&-2&1\\-1/2&1&-1/2\end{pmatrix}.
\]
Then $AX=0$ and every edge entry is nonzero. Each local coefficient space is one-dimensional. Globally, a symmetric matrix annihilating this two-dimensional space has image contained in $\kk s$ and must be a scalar multiple of $ss^{\mathsf T}$; hence $\dim H^0=1$. Since $C^0$ and $C^1$ both have dimension three, $\dim H^1=1$ as well. This realizes the given kernel with $\dim H^1=1$.
\end{example}

\section{Open boundary complexes}\label{sec:ports}

Under a decomposition of $G$, each boundary row of $Ax=0$ is a sum
of contributions from the incident blocks. Keeping these contributions
as residual variables gives the gluing complex. We first recall the
restriction sequence for a cellular sheaf.

\subsection{Restriction of a fixed sheaf}
Let $X=X_1\cup X_2$ be a subcomplex decomposition and $Y=X_1\cap X_2$. Two restrictions of the same cellular sheaf have identical stalks on $Y$, so their sections can be compared by taking the difference of their traces.

\begin{proposition}\label{prop:MV}
Let $X=X_1\cup X_2$ be a finite graph decomposed into subcomplexes, put $Y=X_1\cap X_2$, and let $\sF$ be one cellular sheaf on $X$. There is a short exact sequence of cochain complexes
\begin{equation}\label{eq:MV-complex}
 0\to C^\bullet(X;\sF)
 \to C^\bullet(X_1;\sF)\oplus C^\bullet(X_2;\sF)
 \xrightarrow{r_1-r_2}C^\bullet(Y;\sF)\to0.
\end{equation}
Over a field, set $N_i=\ker(H^0(X_i)\to H^0(Y))$ and $B_i=\im(H^0(X_i)\to H^0(Y))$. Then
\begin{equation}\label{eq:MV-kernel}
 0\to N_1\oplus N_2\to H^0(X;\sF)\to B_1\cap B_2\to0.
\end{equation}
The full long exact sequence ends with
\begin{equation}\label{eq:MV-long}
\begin{split}
 H^0(X_1)\oplus H^0(X_2)&\to H^0(Y)\to H^1(X)\\
 &\to H^1(X_1)\oplus H^1(X_2)\to H^1(Y)\to0.
\end{split}
\end{equation}
\end{proposition}
\begin{proof}
Let $A^\bullet=C^\bullet(X;\sF)$, $B^\bullet=C^\bullet(X_1;\sF)\oplus C^\bullet(X_2;\sF)$ and $C^\bullet=C^\bullet(Y;\sF)$. In each degree, a cochain on $X$ is exactly a pair agreeing on every common cell. Hence the first map is injective and its image is $\ker(r_1-r_2)$. A cochain on $Y$ extends to $X_1$ by zero on cells outside $Y$, proving degreewise surjectivity. Restrictions commute with the differential, because both endpoints of every edge of a subcomplex belong to that subcomplex. This gives the short exact sequence of complexes.

The connecting map can be written directly. For $a\in H^0(Y;\sF)$ choose a degree-zero lift $(u_1,u_2)$ with $u_1|_Y-u_2|_Y=a$. Since $\delta a=0$, the pair $(\delta u_1,\delta u_2)$ agrees on common edges and is the restriction of a unique one-cochain $w$ on $X$. Set $\partial(a)=[w]\in H^1(X;\sF)$. Another lift changes $(u_1,u_2)$ by a cochain on $X$, so it changes $w$ by a coboundary. Thus $\partial$ is well-defined. Chasing lifts gives the long exact sequence; no term beyond degree one occurs because all three complexes have only degrees zero and one.

For the second displayed sequence, a global section is a pair of local sections with equal trace. Its common trace lies in $B_1\cap B_2$. Every element of that intersection has a local lift on each side, and the lifts agree on $Y$, so the trace map is onto the intersection. Its kernel consists precisely of pairs with both traces zero, namely $N_1\oplus N_2$. This proves the asserted exact sequence.
\end{proof}
Restriction of an equation sheaf keeps the original row functional.
For separately specified block matrices, we instead record each block's
contribution to the boundary rows as follows.

\subsection{Boundary values and residuals}
Let $R$ be a commutative ring. A block $H$ has vertex set $I\sqcup B$, where $I$ is its interior and $B$ its boundary, and a symmetric matrix $A_H$ supported on the block's edges. The following gluing identities hold for symmetric block matrices over $R$; the graph-matrix realization problem additionally takes every edge coefficient to be a unit.

Choose interior seeds $P\subseteq I$ and boundary seeds $S_B\subseteq B$. Write $U=I\setminus P$, $B_0=B\setminus S_B$, and set
\begin{equation}\label{eq:port-blocks}
 D=A_{I,U},\quad E_0=A_{I,B_0},\quad
 F_0=A_{B,U},\quad J_0=A_{B,B_0}.
\end{equation}
A vector $z\in R^U$ represents the unknown interior values, and $b\in R^{B_0}$ the unspecified boundary values. Missing seed coordinates are zero. The interior equations are $Dz+E_0b=0$. Introduce a residual $j\in R^B$ for the contribution of this block to the boundary equations.

A second block may cancel the boundary contribution $F_0z+J_0b$.
We therefore introduce the coordinate $j=F_0z+J_0b$.

\begin{definition}\label{def:port-complex}
The \emph{port complex} is
\begin{equation}\label{eq:port-complex}
 \cL_H^\bullet:
 R^U\oplus R^{B_0}\oplus R^B
 \xrightarrow{d_H}R^I\oplus R^B,
 \qquad
 d_H(z,b,j)=(Dz+E_0b,\ F_0z+J_0b-j).
\end{equation}
The port module is $\cW_B=R^{B_0}\oplus R^B$, in degree zero. The port map is
\[
 \rho_H^0(z,b,j)=(b,j),\qquad \rho_H^1=0.
\]
It is a cochain map because the target has zero differential and no positive-degree term.
\end{definition}
The degree-zero cohomology consists of interior solutions with their
boundary values and residuals. Closing the boundary imposes $j=0$.

For two blocks, a global vector gives one boundary value and a zero sum of the two residuals. These are linear conditions on the direct sum of the port complexes. Their kernel contains only an additional identity summand beyond the global equation complex.

\begin{theorem}\label{thm:port-gluing}
Let $G=H_1\cup_B H_2$, with disjoint interiors and no edge from one interior to the other. Assign every edge to exactly one block and split boundary diagonal contributions so that the zero-extended local matrices satisfy $A_G=A_{H_1}+A_{H_2}$. Give the two blocks the same boundary seed set. Define
\begin{equation}\label{eq:q-port}
 q:\cL_{H_1}^\bullet\oplus\cL_{H_2}^\bullet\longrightarrow\cW_B,
 \qquad q^0=(b_1-b_2,\ j_1+j_2),\quad q^1=0.
\end{equation}
For the combined seed set $S$, there are homotopy equivalences
\begin{equation}\label{eq:derived-glue}
 K_{G,S}^\bullet\simeq\ker q\simeq\Cone(q)[-1].
\end{equation}
They commute with every homomorphism of coefficient rings. More precisely,
\begin{equation}\label{eq:kernel-unit}
 \ker q\cong K_{G,S}^\bullet\oplus[R^B\xrightarrow{\id}R^B].
\end{equation}
\end{theorem}
\begin{proof}
The map $q$ is degreewise split surjective: for $(b,j)\in\cW_B$, take $b_1=b$, $j_1=j$, and all other source coordinates zero. In $\ker q$, write
\[
 b_1=b_2=b,\qquad j_1=j,\qquad j_2=-j.
\]
Its differential has the four components
\[
 \begin{pmatrix}
 D_1z_1+E_1b\\
 D_2z_2+E_2b\\
 F_1z_1+J_1b-j\\
 F_2z_2+J_2b+j
 \end{pmatrix}.
\]
Add the last two row blocks and put $j'=j-F_1z_1-J_1b$. Then the differential becomes, after reordering and changing one row sign,
\[
 (z_1,z_2,b,j')\longmapsto
 \left(
 \begin{pmatrix}
 D_1z_1+E_1b\\
 D_2z_2+E_2b\\
 F_1z_1+F_2z_2+(J_1+J_2)b
 \end{pmatrix},\ j'\right).
\]
The first component is $(A_G)_{V,V\setminus S}$ applied to the global unknown vector. The second is the identity on $R^B$, proving \eqref{eq:kernel-unit}.

We use the mapping-fiber model for $\Cone(q)[-1]$ with differential $(d,q)$:
\[
 \cL_{H_1}^0\oplus\cL_{H_2}^0
 \xrightarrow{(d,q)}
 (\cL_{H_1}^1\oplus\cL_{H_2}^1)\oplus\cW_B.
\]
Choose a degree-zero section $s$ of $q$ and write a source as $k+s(w)$, where $k\in\ker q^0$. Applying the target change $(a,w)\mapsto(a-ds(w),w)$ gives
\[
 (d,q)(k+s(w))=(dk+ds(w),w)\longmapsto(dk,w).
\]
Thus this model is $\ker q$ plus an identity complex on $\cW_B$.
Reversing the sign of the last target summand gives the shifted-cone
convention of \cite[\stacks{014D}]{Stacks}. Every change is defined over $R$
and every removed block is an identity, so the equivalences commute with
coefficient base change.
\end{proof}

In either displayed splitting, let $\pi$ and $\iota$ be projection onto and
inclusion of the residual complex. The chain homotopy from $\iota\pi$ to
the identity is zero on the residual summand and is the identity map from
degree one to degree zero on each deleted identity block. Thus
$dh+hd=\id-\iota\pi$ over $R$. Tensoring the explicit maps preserves this
equality for every $R$-algebra $R'$. These are bounded complexes of free
modules, so their derived tensor products are computed term by term. In
particular,
\[
 H^0(K_{G,S}^\bullet\otimes_R^{\mathbf L}R')
 =\ker\bigl((A_G)_{V,V\setminus S}\otimes_R R'\bigr).
\]
For the universal complex over $R_G$, the injective free resolution
\eqref{eq:resolution} gives in addition
\[
 H^0(K_S^\bullet\otimes_{R_G}^{\mathbf L}R')
 \cong\Tor_1^{R_G}(M_S,R').
\]

For several blocks, retain the total boundary residual as an output until the final closure. Let $H_1,\ldots,H_r$ have the same boundary $B$, common boundary seed set, pairwise disjoint interiors, and additive coefficient assignments as in Theorem~\ref{thm:port-gluing}. Define the open assembly complex $\mathcal A^\bullet$ by taking tuples
\[
 (z_1,\ldots,z_r,b,j_1,\ldots,j_r)
\]
in degree zero and all local interior and residual rows in degree one. Its differential is the list
\[
 (D_i z_i+E_i b,\ R_i-j_i)_{i=1}^r,
 \qquad R_i=F_i z_i+J_i b,
\]
and its output port is
\[
 (b,j),\qquad j=\sum_{i=1}^r j_i.
\]
The common value is $b$, and the total residual is $j$.

\begin{proposition}\label{prop:open-port-assembly}
With its output port retained, the open assembly has an isomorphism of complexes
\begin{equation}\label{eq:open-port-assembly}
 \mathcal A^\bullet\cong
 \mathcal L_{\cup_iH_i}^\bullet
 \oplus[R^{(r-1)B}\xrightarrow{\id}R^{(r-1)B}],
\end{equation}
where the output map on the contractible summand is zero. The isomorphism commutes with every coefficient base change.
\end{proposition}
\begin{proof}
Use $j=\sum_i j_i$ and, for $i<r$, the source coordinates
\[
 \eta_i=j_i-R_i.
\]
The inverse is
\[
 j_i=\eta_i+R_i\quad(i<r),\qquad
 j_r=j-\sum_{i<r}(\eta_i+R_i).
\]
The first $r-1$ residual rows become $-\eta_i$. The final residual row is
\[
 R_r-j_r=\sum_iR_i-j+\sum_{i<r}\eta_i.
\]
Add the first $r-1$ rows to the last and negate those first rows. The residual part is now
\[
 \left(\sum_iR_i-j,\eta_1,\ldots,\eta_{r-1}\right).
\]
The interior rows are unchanged. The first component, together with those interior rows, is the port differential of the union. The output is exactly $(b,j)$ and does not involve the $\eta_i$. Both source and target substitutions are invertible and use only addition and multiplication in $R$, so they survive arbitrary base change.
\end{proof}
At each binary gluing, the output is $(b,j_1+j_2)$. The final closure
sets the total residual equal to zero.

\begin{corollary}\label{cor:many-ports}
For $r$ blocks with common boundary, disjoint interiors, and consistently assigned coefficients, impose $b_i=b_1$ for $2\le i\le r$ and $\sum_i j_i=0$. The resulting kernel complex is the global equation complex plus $r-1$ contractible identity complexes on $R^B$. In particular, iterated open gluing followed by final residual closure gives the same global complex in different parenthesizations, up to the explicit unit cancellations.
\end{corollary}
\begin{proof}
Apply Proposition~\ref{prop:open-port-assembly} while retaining the common output $(b,j)$. It identifies the open assembly with the port complex of the union and $r-1$ identity pairs, without changing the output. Now impose $j=0$. The union port complex restricted to this condition is exactly $K_{G,S}^\bullet$, so the result is
\[
 K_{G,S}^\bullet\oplus
 [R^{(r-1)B}\xrightarrow{\id}R^{(r-1)B}].
\]
For an iterated binary assembly, the residual at each node is the sum of the residuals of its descendant blocks. Both parenthesizations give the same expressions
\[
 b_i=b_1,\qquad j=\sum_i j_i,
\]
because addition is associative. The displayed coordinate changes eliminate the intermediate residual
differences over $R$. Setting the final $j$ equal to zero gives the same
global complex for every parenthesization.
\end{proof}

\begin{example}\label{ex:premature-closure}
Let three one-edge blocks have common boundary vertex $b$ and respective interior vertices $1,2,3$. Give their edge coefficients the value one, their interior diagonals zero, and fix $b=0$. Each block has no remaining interior equation on $z_i$, and its boundary residual is $j_i=z_i$. The assembled closed equation is
\[
 z_1+z_2+z_3=0.
\]
The vector $(1,1,-2)$ is a valid global interior solution. If the first two blocks were closed before attaching the third, their imposed equation would instead be $z_1+z_2=0$, which excludes this solution. Retaining $j_{12}=z_1+z_2$ and imposing $j_{12}+j_3=0$ gives the
global equation, as in Corollary~\ref{cor:many-ports}.
\end{example}

\subsection{Boundary relations and invisible interior solutions}
Fix a field-valued coefficient fiber. Define
\[
 \cN_H=\ker(H^0(\cL_H)\xrightarrow{\rho_H}\cW_B),\qquad
 \cR_H=\im(H^0(\cL_H)\xrightarrow{\rho_H}\cW_B).
\]
Thus $\cN_H$ is the space of invisible interior solutions and $\cR_H$ is the boundary relation. If $b=0$ and $j=0$, the invisible space is explicitly
\begin{equation}\label{eq:invisible-interior}
 \cN_H=\ker D\cap\ker F_0.
\end{equation}
Let $\iota(b,j)=(b,-j)$.

Two local solutions glue when their boundary values agree and their residuals are opposite. Solutions with both port coordinates zero can be chosen independently inside either block; the remaining solutions are measured by an intersection of boundary relations.

\begin{proposition}\label{prop:port-H0}
There is an exact sequence
\begin{equation}\label{eq:port-H0}
 0\to\cN_{H_1}\oplus\cN_{H_2}
 \to H^0(K_{G,S})
 \to\cR_{H_1}\cap\iota(\cR_{H_2})\to0.
\end{equation}
Furthermore,
\begin{equation}\label{eq:port-LES}
\begin{split}
0\to H^0(K_{G,S})&\to H^0(\cL_{H_1})\oplus H^0(\cL_{H_2})\to\cW_B\\
&\to H^1(K_{G,S})\to H^1(\cL_{H_1})\oplus H^1(\cL_{H_2})\to0.
\end{split}
\end{equation}
In particular,
\begin{equation}\label{eq:port-dimension}
 \dim H^0(K_{G,S})
 =\dim\cN_{H_1}+\dim\cN_{H_2}
 +\dim\bigl(\cR_{H_1}\cap\iota(\cR_{H_2})\bigr).
\end{equation}
\end{proposition}
\begin{proof}
Theorem~\ref{thm:port-gluing} identifies a global null vector with a pair
of local sections $s_1,s_2$ having ports
\[
 \rho_{H_1}(s_1)=(b,j),\qquad
 \rho_{H_2}(s_2)=(b,-j).
\]
Its common signed port lies in
$\cR_{H_1}\cap\iota(\cR_{H_2})$. Conversely, for every point in that
intersection choose one local section on each side realizing it. Their
values agree and residuals sum to zero, so they give a global vector. Thus
the port map is onto the stated intersection. A pair maps to zero exactly
when both local sections have zero value and residual, so its kernel is
$\cN_{H_1}\oplus\cN_{H_2}$. This proves~\eqref{eq:port-H0}; taking dimensions
proves~\eqref{eq:port-dimension}.

For the longer sequence, $q$ in~\eqref{eq:q-port} is split surjective in
every degree. The short exact sequence of complexes
\[
 0\to\ker q\to\cL_{H_1}^\bullet\oplus\cL_{H_2}^\bullet
       \xrightarrow q\cW_B[0]\to0
\]
has connecting map obtained as follows. Lift $(b,j)\in\cW_B$ to a
degree-zero vector $s$ in the middle complex. Its differential $ds$ lies
in $(\ker q)^1$ because $q$ is a chain map. Send $(b,j)$ to $[ds]$.
Another lift differs by $(\ker q)^0$, so this changes $ds$ by a boundary.
Theorem~\ref{thm:port-gluing} identifies the cohomology of $\ker q$ with
that of $K_{G,S}$. The resulting long exact sequence is
\eqref{eq:port-LES}; it ends at degree one because both port complexes have
only two terms and $\cW_B$ has only degree zero.
\end{proof}
Equation~\eqref{eq:port-dimension} separates the invisible interior
solutions from the boundary relation. In coefficient families, the free
complex \eqref{eq:port-complex} also retains the kernels arising under
specialization.

\subsection{Elimination on a coefficient open set}
An invertible internal submatrix permits a smaller port presentation. The residual variable is kept during the elimination, so that the effect of the eliminated variables on the boundary is included.
Suppose a square subblock $P_0$ of the interior matrix $D$ is invertible on a coefficient open set. Arrange the variables and equations as
\[
 D=\begin{pmatrix}P_0&Q_0\\R_0&D_0\end{pmatrix},\quad
 E_0=\binom{E_1}{E_2},\quad F_0=(F_1\ F_2).
\]

\begin{proposition}\label{prop:schur-port}
On this open set, eliminating the pivot variables removes a contractible identity complex and leaves the port equations
\begin{align}
 \overline D z_2+\overline E b&=0,\label{eq:schur-equation}\\
 j&=\overline Fz_2+\overline Jb,\label{eq:schur-residual}
\end{align}
where
\begin{align*}
 \overline D&=D_0-R_0P_0^{-1}Q_0,&
 \overline E&=E_2-R_0P_0^{-1}E_1,\\
 \overline F&=F_2-F_1P_0^{-1}Q_0,&
 \overline J&=J_0-F_1P_0^{-1}E_1.
\end{align*}
If $P_0$ has size $r$ and one restricts to the rank-$r$ subscheme defined on this open set by $I_{r+1}(D)$, then $\overline D=0$, but the variables $z_2$ and the conditions $\overline Eb=0$ remain.
\end{proposition}
\begin{proof}
The pivot equation is
\[
 P_0z_1+Q_0z_2+E_1b=0,
 \qquad z_1=-P_0^{-1}(Q_0z_2+E_1b).
\]
Substitution gives
\begin{align*}
 R_0z_1+D_0z_2+E_2b
 &=(D_0-R_0P_0^{-1}Q_0)z_2
   +(E_2-R_0P_0^{-1}E_1)b,\\
 F_1z_1+F_2z_2+J_0b
 &=(F_2-F_1P_0^{-1}Q_0)z_2
   +(J_0-F_1P_0^{-1}E_1)b.
\end{align*}
These are \eqref{eq:schur-equation}--\eqref{eq:schur-residual}. At the level of complexes, replace the source coordinate by
\[
 z_1'=z_1+P_0^{-1}(Q_0z_2+E_1b),
\]
normalize the pivot row by $P_0^{-1}$, and subtract its contributions from the other rows. The new differential is the identity on $z_1'$ direct summed with the displayed residual complex.

The same operations give $D\sim\diag(I_r,\overline D)$ and therefore
\[
 I_{r+1}(D)=I_1(\overline D).
\]
On the stated rank-$r$ subscheme, $\overline D=0$. The remaining variables $(z_2,b,j)$ satisfy
\[
 \overline Eb=0,\qquad j=\overline Fz_2+\overline Jb.
\]
\end{proof}

\subsection{Symmetric equations and Lagrangian relations}
Suppose there are no seeds in the block, and specialize to a field. Then $D=D^{\mathsf T}$, $E_0=F_0^{\mathsf T}$, and $J_0=J_0^{\mathsf T}$. On $\cW_B=\kk^B\oplus\kk^B$ use the alternating form
\[
 \omega((b,j),(b',j'))=b^{\mathsf T}j'-j^{\mathsf T}b'.
\]
These are the boundary relations of linear networks~\cite{BF}.

Symmetry cancels the boundary pairing of two local solutions. The dimension of the relation is obtained by subtracting the invisible interior kernel from the space of all interior solutions.

\begin{proposition}\label{prop:lagrangian}
For an unseeded symmetric block over a field, the relation
\[
 \cR_H=\{(b,j):Dz+E_0b=0,\ j=F_0z+J_0b\text{ for some }z\}
\]
is Lagrangian: it is isotropic and has dimension $|B|$.
\end{proposition}
\begin{proof}
Write $m=|I|$ and $b_0=|B|$. With no seed coordinates removed, symmetry says
$D=D^{\mathsf T}$, $E_0=F_0^{\mathsf T}$, and $J_0=J_0^{\mathsf T}$.
For local solutions $(z,b)$ and $(z',b')$, substitution gives
\[
 \begin{aligned}
 \omega((b,j),(b',j'))
 &=b^{\mathsf T}(F_0z'+J_0b')
                  -(F_0z+J_0b)^{\mathsf T}b'\\
 &=(E_0b)^{\mathsf T}z'-z^{\mathsf T}E_0b'\\
 &=(-Dz)^{\mathsf T}z'-z^{\mathsf T}(-Dz')=0.
 \end{aligned}
\]
Thus $\mathcal R_H$ is isotropic.

Its dimension must account for solutions invisible at the ports. Let
$N=\ker D\cap\ker F_0$. The transpose of the $m$-row matrix $(D\ E_0)$
is the column matrix $\binom D{F_0}$, whose kernel is exactly $N$.
Therefore $\rank(D\ E_0)=m-\dim N$, and
\[
 \dim\ker(D\ E_0)=m+b_0-(m-\dim N)=b_0+\dim N.
\]
The map from this solution space to the boundary relation is
$(z,b)\mapsto(b,F_0z+J_0b)$. Its kernel consists of $b=0$, $Dz=0$,
$F_0z=0$, hence is $N$. Subtracting this kernel gives
\[
 \dim\cR_H=(b_0+\dim N)-\dim N=b_0.
\]
The alternating boundary form is nondegenerate on a $2b_0$-dimensional
space. Hence $\mathcal R_H$ is Lagrangian.
\end{proof}

The Lagrangian dimension remains constant when the projection to boundary values drops rank. A single interior vertex already exhibits this distinction.

\begin{example}\label{ex:singular-port}
One interior vertex joined to two boundary vertices, with no boundary diagonal term, has equations
\[
 dz+pb_1+qb_2=0,\qquad (j_1,j_2)=(pz,qz),\qquad p,q\ne0.
\]
When $d\ne0$, every boundary value determines a unique $z$, and the relation is the graph of a linear map in $b$. When $d=0$, the value condition is $pb_1+qb_2=0$, while the residual may be any multiple of $(p,q)$. The relation still has dimension two, but its projection to boundary values drops rank. This projection degeneration contributes to an invisible global solution precisely through the intersection and kernel terms in~\eqref{eq:port-H0}.
\end{example}

For the two single-edge blocks $1-2$ and $2-3$ with the common boundary value fixed to zero, the residuals at $2$ are $az_1$ and $bz_3$. The global condition is $az_1+bz_3=0$, not the two separate conditions $az_1=0$ and $bz_3=0$. At $d_1=d_3=0$ this produces the invisible vector of Example~\ref{ex:path-ideal}. It is the smallest example of an obstruction generated by boundary cancellation rather than by a hidden mode of either isolated block.

\subsection{Forcing with a prescribed set of active boundary vertices}
Fixing all boundary values to zero does not require every boundary residual to vanish. Selecting which residuals vanish corresponds to selecting which initially black boundary vertices are permitted to force.
Let $H$ now be a finite simple graph block. All boundary vertices are initially black. Interior seeds are $P\subseteq I$. Fix $T\subseteq B$ as the boundary vertices allowed to force; vertices of $B\setminus T$ stay black but never act as sources. Every black interior vertex may force by the ordinary rule. Related restrictions on forcing sources occur in leaky forcing~\cite{Leaky}.

When the permitted process stops, the remaining white vertices have no unique neighbor from an eligible source. Only vertices of $I\setminus F$ and of the active set $T$ need to be tested.

\begin{definition}
A \emph{relative fort for $T$} is a nonempty set $F\subseteq I$ such that
\begin{equation}\label{eq:relative-fort}
 |N_H(u)\cap F|\ne1\qquad(u\in(I\setminus F)\cup T).
\end{equation}
Over the universal ring of the block define
\begin{equation}\label{eq:relative-complex}
 K_{H;P,T}^\bullet=
 [R^{I\setminus P}\xrightarrow{\bA_{I\cup T,I\setminus P}}R^{I\cup T}],
 \qquad M_{H;P,T}=H^1(K_{H;P,T}).
\end{equation}
\end{definition}

The retained rows in \eqref{eq:relative-complex} are exactly those that may propagate zero under this boundary rule. The proof of Theorem~\ref{thm:flatness} therefore has a relative form.

\begin{theorem}\label{thm:relative-forcing}
The following are equivalent:
\begin{enumerate}
\item the prescribed boundary rule colors the entire interior from $P\cup B$;
\item every relative fort for $T$ meets $P$;
\item $\bA_{I\cup T,I\setminus P}$ has zero kernel in every characteristic-zero field fiber;
\item $M_{H;P,T}$ is flat;
\item $M_{H;P,T}\cong R^{|P|+|T|}$.
\end{enumerate}
The complex \eqref{eq:relative-complex} is obtained from the port complex by imposing $b=0$, imposing $j|_T=0$, and canceling the unobserved residual coordinates on $B\setminus T$.
\end{theorem}
\begin{proof}
Suppose $F\subseteq I\setminus P$ is a relative fort. Immediately before a first force $u\to v$ into $F$, the eligible source would lie in $(I\setminus F)\cup T$ and satisfy
\[
 N_H(u)\cap F=\{v\},
\]
contrary to \eqref{eq:relative-fort}. Conversely, if the process stops with a nonempty white set $F\subseteq I\setminus P$, every eligible black vertex has a number of neighbors in $F$ different from one. Thus $F$ satisfies \eqref{eq:relative-fort}. This proves (a)$\Longleftrightarrow$(b).

A nonzero vector in the column kernel, extended by zero on $P\cup B$, has relative-fort support: a retained row seeing only one nonzero neighbor would give
\[
 0=(Ax)_u=a_{uv}x_v\ne0.
\]
Conversely, for a relative fort $F$ avoiding $P$, put $x=\one_F$. At each retained exterior row with $k\ge2$ neighbors in $F$, choose the crossing coefficients
\[
 1,\ldots,1,-(k-1),\qquad
 \sum_{v\in N(u)\cap F}a_{uv}=0.
\]
Choose nonzero weights on all other edges, preserving symmetry, and set
\[
 d_v=-\sum_{w\in N(v)\cap F}a_{vw}\quad(v\in F).
\]
Then
\[
 A_{I\cup T,I\setminus P}\one_F=0.
\]
Each crossing edge has a single exterior endpoint, so the assignments are consistent. This proves (b)$\Longleftrightarrow$(c).

The universal column map is injective: its principal submatrix on $I\setminus P$ has the nonzero diagonal monomial $\prod_{v\in I\setminus P}d_v$ in its determinant. A successful forcing sequence uses source rows in $I\cup T$ and gives a triangular unit minor on every unknown column. The elimination in Theorem~\ref{thm:flatness} yields
\[
 M_{H;P,T}\cong R^{|I|+|T|-|I\setminus P|}
 =R^{|P|+|T|}.
\]
If forcing fails, the rational vector just constructed gives
\[
 \Tor_1^R(M_{H;P,T},\kk_\alpha)
 =\ker\bA_{I\cup T,I\setminus P}(\alpha)\ne0.
\]
Thus the module is not flat. This proves the remaining equivalences.

Finally, impose $b=0$ and $j|_T=0$ in \eqref{eq:port-complex}. The rows indexed by $B\setminus T$ each contain their own residual coordinate with coefficient $-1$. Canceling these identity pairs leaves
\[
 R^{I\setminus P}\xrightarrow{\bA_{I\cup T,I\setminus P}}R^{I\cup T},
\]
which is \eqref{eq:relative-complex}.
\end{proof}

\subsection{Two-port triangular strips}
We apply the relative criterion to the rooted configurations in \cite{BranchDraft}. The graph is defined by successive triangle attachments, so that the boundary rule and the possible obstructing tail can both be written explicitly.
\begin{definition}\label{def:strip}
Start with a root edge $e_0=\{a,b\}$. Add vertices $v_1,\ldots,v_m$, $m\ge2$, by joining $v_i$ to both endpoints of $e_{i-1}$. For $i<m$, choose $e_i$ to contain $v_i$ and one endpoint of $e_{i-1}$. Fix $e_1=\{a,v_1\}$. The boundary is $\{a,b\}$ and the other vertices are interior. The strip is a \emph{fan} if $a$ remains in every frontier $e_i$; otherwise call it a \emph{non-fan strip}.
\end{definition}
In a fan the interior vertices form a path attached to the common vertex $a$. In a non-fan, let $j$ be the first frontier which loses $a$. The tail starting at $v_j$ will distinguish the two active-boundary conditions.

\begin{theorem}\label{thm:strip}
Initially color only $a,b$ black, with no interior seeds. For the partially active rule,
\begin{align}
 H\text{ a fan}:&\quad\text{the interior is forced iff }b\in T,\label{eq:fan-rule}\\
 H\text{ a non-fan}:&\quad\text{the interior is forced iff }T=\{a,b\}.
 \label{eq:nonfan-rule}
\end{align}
Consequently $M_{H;\varnothing,\{b\}}$ is free of rank one for a fan and is nonflat for a non-fan; with both ports active it is free of rank two in either case.
\end{theorem}
\begin{proof}
Let $I=\{v_1,\ldots,v_m\}$. The first-frontier choice gives
\[
 N(b)\cap I=\{v_1\},\qquad
 \{v_1,v_2\}\subseteq N(a)\cap I.
\]
If $b\notin T$, no first force is possible. For a fan, all edges have the form given by the fixed root, the consecutive interior path, and the edges from $a$ to the interior. Thus
\[
 b\to v_1\to v_2\to\cdots\to v_m
\]
is a legal sequence whenever $b\in T$, proving \eqref{eq:fan-rule}.

For a non-fan, take the first $j$ with $a\notin e_j$. Then
\[
 2\le j\le m-1,\qquad
 e_{j-1}=\{a,v_{j-1}\},\qquad
 e_j=\{v_{j-1},v_j\}.
\]
Set $F=\{v_j,\ldots,v_m\}$. Since vertices discarded from a frontier never return,
\[
 \{v_j,v_{j+1}\}\subseteq N(v_{j-1})\cap F,
 \qquad N(v_i)\cap F=\varnothing\quad(i<j-1),
 \qquad N(b)\cap F=\varnothing.
\]
Therefore $|N(u)\cap F|\ne1$ for all $u\in(I\setminus F)\cup\{b\}$. The set $F$ is a relative fort for $T=\{b\}$ and prevents complete forcing.

Suppose now that $T=\{a,b\}$. For $i<m$, denote by $u_i$ the endpoint of $e_{i-1}$ discarded in choosing $e_i$. Before coloring $v_i$, the two vertices of $e_{i-1}$ are black, and
\[
 N(u_i)\cap\{v_i,\ldots,v_m\}=\{v_i\}.
\]
Thus $u_i\to v_i$ is legal. At the last step, a vertex of $e_{m-1}$ forces $v_m$. This proves \eqref{eq:nonfan-rule}. The module statements follow from Theorem~\ref{thm:relative-forcing} with $P=\varnothing$ and $|T|=1$ or $2$.
\end{proof}

The obstructing non-fan tail is seen exactly once from $a$ and not at all from $b$. In the smallest example, the triangles are $(a,b,u)$, $(a,u,v)$, $(u,v,w)$. With interior order $(u,v,w)$ choose
\[
 Q=\begin{pmatrix}0&1&-1\\1&-1&1\\-1&1&-1\end{pmatrix},\qquad
 C=\begin{pmatrix}1&1&0\\1&0&0\end{pmatrix},\qquad z=\begin{pmatrix}0\\1\\1\end{pmatrix},
\]
where the rows of $C$ are $a,b$. Every displayed graph edge has nonzero coefficient, and
\[
 Qz=0,\qquad Cz=\binom10.
\]
Thus $j_b=0$ permits this nonzero interior solution, whereas imposing $j_a=0$ removes it. For a fan, the boundary equation and successive interior equations instead give
\[
 j_b=t_{bv_1}z_{v_1}=0
 \ \Longrightarrow\ z_{v_1}=0
 \ \Longrightarrow\ z_{v_2}=0
 \ \Longrightarrow\cdots\Longrightarrow\ z_{v_m}=0.
\]
These implications use only the nonzero edge coefficients. They give the port interpretation of the distinction in Theorem~\ref{thm:strip}.

\section{Complete supports and minimum-cost gluing}\label{sec:support-boundaries}

The preceding construction fixes the coefficients and keeps the boundary
values and residuals of a local solution. We now vary the coefficients.
Proposition~\ref{prop:support-witness} identifies the possible nonzero
supports with the forts. For a block, we impose the fort condition at
interior vertices and record the boundary data. These data are collected
for one fixed seed set before its cost is minimized.

\subsection{The sheaf of supports}\label{sec:states}

For $u\in V$ define
\begin{equation}\label{eq:support-relation}
 \mathcal Q_u=
 \left\{\sigma\in\{0,1\}^{N[u]}:
 \sigma_u=1\ \text{or}\ \sum_{v\in N(u)}\sigma_v\ne1\right\}.
\end{equation}
The value $\sigma_v=1$ means membership in a candidate kernel support.

Use the bipartite incidence graph $Y_G$ with variable vertices $p_v$, relation vertices $q_u$, and edges $p_vq_u$ for $v\in N[u]$. Define a pointed-set cellular sheaf by
\[
 \sB_{G,S}(p_v)=\begin{cases}\{0\},&v\in S,\\\{0,1\},&v\notin S,\end{cases}
 \qquad \sB_{G,S}(q_u)=\mathcal Q_u,
 \qquad \sB_{G,S}(p_vq_u)=\{0,1\}.
\]
The incidence maps are the inclusion from a variable stalk and the coordinate projection from a relation stalk. All base points are zero. Definition~\ref{def:cellular} gives the sheaf on the face space.

The incidence restrictions force the local tuples to come from a single binary assignment on $V$. The conditions at vertices outside its one-set are exactly the fort conditions.

\begin{proposition}\label{prop:support-sheaf}
Its global sections are naturally
\begin{equation}\label{eq:support-sections}
 \Gamma(Y_G;\sB_{G,S})
 =\{\varnothing\}\cup\{F\subseteq V\setminus S:F\text{ is a fort}\}.
\end{equation}
Thus this section set contains only the base point if and only if $S$ is zero forcing, equivalently if and only if $M_S$ is flat.
\end{proposition}
\begin{proof}
Let $\sigma\in\{0,1\}^V$ be the common assignment of a global section, and put $F=\{v:\sigma_v=1\}$. The variable stalks and relation stalks respectively impose
\[
 F\cap S=\varnothing,
\]
\[
 u\notin F\quad\Longrightarrow\quad
 |N(u)\cap F|=\sum_{v\in N(u)}\sigma_v\ne1.
\]
Thus $F$ is either empty or a fort avoiding $S$. Conversely, such a set gives the section by its characteristic function and the corresponding neighborhood tuples. This proves \eqref{eq:support-sections}. Propositions~\ref{prop:fort-hit} and Theorem~\ref{thm:flatness} then give
\[
 \Gamma(Y_G;\sB_{G,S})=\{\varnothing\}
 \quad\Longleftrightarrow\quad
 S\text{ is zero forcing}
 \quad\Longleftrightarrow\quad M_S\text{ is flat}.
\]
\end{proof}
The nonbasepoint sections of $\sB_{G,S}$ are realized by symmetric null vectors, by Proposition~\ref{prop:support-witness}.

\subsection{Four states at each boundary vertex}
Let $H=(I\sqcup B,E_H)$ be a finite block. Every edge incident to an interior vertex belongs to the block; future attachments occur only through $B$. Fix a seed set $S_H\subseteq V(H)$.

At an interior vertex the fort condition can already be imposed. At a boundary vertex, further neighbors may arrive from another block, so we retain membership in the support and the neighbor count truncated at two. A separate bit distinguishes the empty support from a nonempty support invisible at the boundary.

Define the open-boundary sheaf as follows. On $Y_H$, retain the variable stalks $\{0\}$ at seeds and $\{0,1\}$ elsewhere. Put
\[
 \sB^\circ_{H,S_H}(q_u)=
 \begin{cases}
 \mathcal Q_u,&u\in I,\\
 \{0,1\}^{N_H[u]},&u\in B,
 \end{cases}
 \qquad
 \sB^\circ_{H,S_H}(p_vq_u)=\{0,1\}.
\]
All maps are the same inclusions and coordinate projections as before.
This gives a pointed-set sheaf by Definition~\ref{def:cellular}. At boundary
vertices, record support membership and the truncated number of supporting
neighbors contributed by $H$. Impose the boundary relation when the vertex
is closed after assembly.

\begin{proposition}\label{prop:open-support-sections}
There is a natural bijection
\[
 \Gamma(Y_H;\sB^\circ_{H,S_H})
 \longleftrightarrow
 \{F\subseteq V(H)\setminus S_H:
 |N_H(u)\cap F|\ne1\text{ for }u\in I\setminus F\}.
\]
Each set on the right is realized, for some characteristic-zero symmetric coefficient fiber with every edge nonzero, as the exact support of a vector satisfying all interior rows and the given zero seed coordinates. Boundary residuals are left free.
\end{proposition}
\begin{proof}
The coordinate restrictions identify every local binary tuple with one assignment on $V(H)$. A seed coordinate is zero. At an interior vertex outside the one-set, membership in $\mathcal Q_u$ says exactly that the number of one-neighbours is not one; at a boundary vertex the entire binary cube imposes no condition. This proves the bijection.

For the realization take $x=\one_F$. At each $u\in I\setminus F$ with nonempty $N_H(u)\cap F$, its size is $k\ge2$. Give these crossing edges the weights $1,\ldots,1,-(k-1)$, symmetrically at their two ends. Such an edge has a unique endpoint outside $F$, so assignments at different constrained rows do not conflict. Give all other edges arbitrary nonzero rational weights. At $u\in I\cap F$, set
\[
 d_u=-\sum_{v\in N_H(u)\cap F}t_{uv}.
\]
Other diagonals may be chosen freely. Then $(Ax)_u=0$ for every $u\in I$, all seed coordinates vanish, and the boundary residual is defined by $j=(Ax)|_B$. Conversely any such vector has an admissible support, since an interior zero row cannot see a unique nonzero neighbouring term. The empty support is represented by the zero vector.
\end{proof}
We collect these supports according to their boundary memberships and
truncated neighbor counts.

\begin{definition}\label{def:profile}
An \emph{admissible local support} is any $F\subseteq V(H)\setminus S_H$, including the empty set, such that
\begin{equation}\label{eq:local-support}
 |N_H(u)\cap F|\ne1\qquad(u\in I\setminus F).
\end{equation}
The boundary coordinates have the following four states. At $b\in B$ put
\[
 \sigma_b(F)=
 \begin{cases}
 \Istate,&b\in F,\\
 \Ostate_0,&b\notin F,\ |N_H(b)\cap F|=0,\\
 \Ostate_1,&b\notin F,\ |N_H(b)\cap F|=1,\\
 \Ostate_2,&b\notin F,\ |N_H(b)\cap F|\ge2.
 \end{cases}
\]
Also put $\eps(F)=0$ if $F=\varnothing$ and $\eps(F)=1$ otherwise. Its signature is $q_H(F)=(\sigma(F),\eps(F))$. The \emph{complete support profile} of the fixed seed set is
\begin{equation}\label{eq:profile}
 \Pi_H(S_H)=\{q_H(F):F\text{ satisfies }\eqref{eq:local-support}\}.
\end{equation}
\end{definition}
A profile belongs to one fixed seed set and contains every boundary
signature realized by a surviving support. Thus the later minimization
ranges over seed sets with a prescribed complete profile. If $|B|=b$,
there are at most $2\cdot4^b$ signatures. The extra bit distinguishes an
interior nonempty support from the empty support even when their boundary
counts agree.

Under edge-disjoint gluing, the two support sets must agree on membership of a boundary vertex, and their numbers of incident support edges add. This determines the following product.

\begin{definition}
Two signatures on the same labeled boundary are compatible if, at each boundary vertex, either both states are $\Istate$ or both are outside states. Their product is defined by
\begin{equation}\label{eq:signature-product}
 \Istate\odot\Istate=\Istate,\qquad
 \Ostate_i\odot\Ostate_j=\Ostate_{\min(2,i+j)},
 \qquad \eps=\eps_1\lor\eps_2.
\end{equation}
For profiles $\Pi_1,\Pi_2$, let $\Pi_1\star\Pi_2$ be the set of products of all compatible pairs.
\end{definition}

Every admissible support on the union restricts to such a compatible pair. Conversely, the union of two compatible supports satisfies all interior conditions, since an interior vertex has no neighbors in the other interior.

\begin{theorem}\label{thm:profile-gluing}
Suppose $H=H_1\cup_B H_2$, the interiors are disjoint, the edge sets are disjoint, and there are no edges joining the two interiors. If $S_1\cap B=S_2\cap B$, then
\begin{equation}\label{eq:profile-gluing}
 \Pi_H(S_1\cup S_2)=\Pi_{H_1}(S_1)\star\Pi_{H_2}(S_2).
\end{equation}
The operation $\star$ is associative.
\end{theorem}
\begin{proof}
Let $F$ be an admissible support of $H$ and put $F_i=F\cap V(H_i)$. For every $u\in I_i\setminus F_i$,
\[
 N_H(u)\cap F=N_{H_i}(u)\cap F_i,
\]
so $F_i$ is admissible. Their memberships agree on $B$. If $b\notin F$, edge-disjointness gives
\[
 |N_H(b)\cap F|
 =|N_{H_1}(b)\cap F_1|+|N_{H_2}(b)\cap F_2|.
\]
Therefore
\[
 q_H(F)=q_{H_1}(F_1)\odot q_{H_2}(F_2),
 \qquad \eps(F)=\eps(F_1)\lor\eps(F_2).
\]
This proves one inclusion in \eqref{eq:profile-gluing}.

Conversely, choose admissible $F_1,F_2$ realizing a compatible pair. Their boundary memberships agree, and their seed sets agree on $B$. Hence $F=F_1\cup F_2$ avoids $S_1\cup S_2$. Its interior conditions hold by the same neighborhood equality, and its signature is the stated product. Thus the reverse inclusion holds.

For three signatures, compatibility requires the same support membership at each boundary vertex. On the outside states,
\[
 \min(2,\min(2,i+j)+k)
 =\min(2,i+j+k)
 =\min(2,i+\min(2,j+k)),
\]
and the nonempty bit is combined by associative disjunction. Taking all compatible witnesses therefore proves associativity of $\star$.
\end{proof}

The product can retain a support which neither separate boundary equation would permit. The path of Example~\ref{ex:path-ideal} gives the smallest calculation.

\begin{example}\label{ex:profile-path}
Take the blocks $1-2$ and $2-3$, with common boundary $2$ and seed set $\{2\}$. Each has complete profile
\[
 \{(\Ostate_0,0),(\Ostate_1,1)\}.
\]
The second state is the support consisting of its leaf. Their product contains $(\Ostate_2,1)$, representing the two leaves together. When the boundary vertex becomes internal this state is admissible, since the forbidden count is one, not two. The resulting fort $\{1,3\}$ prevents forcing. This is the support version of residual cancellation in Example~\ref{ex:path-ideal}.
\end{example}

\subsection{Minimum cost for a complete profile}
For zero forcing, one seed set must eliminate every nonempty support. We minimize its cost for each complete profile.
Fix $P\subseteq B$ and a profile $\Pi$. Define
\begin{equation}\label{eq:profile-cost}
 c_H(P,\Pi)=
 \min\{|S_H\cap I|:S_H\cap B=P,\ \Pi_H(S_H)=\Pi\},
\end{equation}
with minimum $+\infty$ when the set is empty. Boundary seeds are not charged until the corresponding vertices leave the boundary.

Disjoint interiors have additive seed costs. Theorem~\ref{thm:profile-gluing}
therefore gives the following formula.

\begin{theorem}\label{thm:cost-gluing}
Under the assumptions of Theorem~\ref{thm:profile-gluing},
\begin{equation}\label{eq:cost-gluing}
 c_H(P,\Pi)=
 \min_{\Pi_1\star\Pi_2=\Pi}
 \bigl(c_{H_1}(P,\Pi_1)+c_{H_2}(P,\Pi_2)\bigr).
\end{equation}
\end{theorem}
\begin{proof}
For a seed set $S_H$ on the union, take $S_i=S_H\cap V(H_i)$ and $\Pi_i=\Pi_{H_i}(S_i)$. Their boundary seed sets both equal $P$, and Theorem~\ref{thm:profile-gluing} gives $\Pi_1\star\Pi_2=\Pi_H(S_H)$. Since the two interiors are disjoint,
\[
 \begin{aligned}
 |S_H\cap I|
 &=|S_1\cap I_1|+|S_2\cap I_2|\\
 &\ge c_{H_1}(P,\Pi_1)+c_{H_2}(P,\Pi_2).
 \end{aligned}
\]
Thus every seed set counted on the left has cost at least one candidate on the right. Minimizing gives the lower bound.

Conversely fix a profile pair with $\Pi_1\star\Pi_2=\Pi$ and finite costs. Finiteness of the two vertex sets supplies one seed set $S_i$ attaining each whole-profile cost. Both restrict to $P$ at the boundary, so $S=S_1\cup S_2$ is a well-defined seed choice. The profile gluing theorem then gives
\[
 \Pi_H(S)=\Pi_1\star\Pi_2=\Pi,
\]
and its interior cost is exactly the sum of the two minima. Thus each finite candidate on the right is realized on the left. If there is no finite candidate, no union seed set can exist by the first implication. This proves equality also for the value $+\infty$.
\end{proof}
The resulting seed set satisfies
\[
 \exists S\quad\forall F\in\mathfrak F(G),\quad S\cap F\ne\varnothing.
\]

\subsection{Vertex and edge operations}\label{sec:state-operations}
To evaluate \eqref{eq:profile-cost}, construct a graph by introducing boundary vertices, introducing edges between boundary vertices, forgetting boundary vertices, and joining blocks. An interior vertex must have all its incident edges present before it is forgotten. We give the profile update in each case.

For a signature $(\sigma,\eps)$, write $\sigma[v\mapsto a]$ for adjoining a new boundary coordinate in state $a$. If $v$ is introduced as a seed, its only possible state is $\Ostate_0$. Otherwise it can either be outside the support or belong to it. Thus the two operations on a whole profile are
\begin{align*}
 \operatorname{add}_v^1(\Pi)
 &=\{(\sigma[v\mapsto\Ostate_0],\eps):(\sigma,\eps)\in\Pi\},\\
 \operatorname{add}_v^0(\Pi)
 &=\operatorname{add}_v^1(\Pi)
   \cup\{(\sigma[v\mapsto\Istate],1):(\sigma,\eps)\in\Pi\}.
\end{align*}
The superscript specifies whether $v$ is chosen as a seed. It updates $P$ but does not yet add to the cost.

When an edge $uv$ is introduced, an endpoint in state $\Istate$ remains in that state. For an endpoint $u$ in state $\Ostate_i$, the update is
\[
 \Ostate_i\longmapsto
 \Ostate_{\min(2,i+\mathbf1_{\{\sigma_v=\Istate\}})}.
\]
Apply the symmetric rule at $v$ and leave all other states and $\eps$ unchanged. Both endpoints are boundary vertices, so no interior condition is altered.

When $v$ is forgotten, impose its fort condition and then remove its boundary coordinate:
\[
 \operatorname{forget}_v(\Pi)
 =\{(\sigma|_{B\setminus\{v\}},\eps):
     (\sigma,\eps)\in\Pi,\ \sigma_v\ne\Ostate_1\}.
\]
The cost and boundary seed set change by
\[
 (P,c)\longmapsto
 \bigl(P\setminus\{v\},\ c+\mathbf1_{\{v\in P\}}\bigr).
\]
The nonempty bit is retained, including when the entire support has become internal.

For a join, use \eqref{eq:profile-gluing} and \eqref{eq:cost-gluing}. If several constructions give the same pair $(P,\Pi)$, retain the least cost. The initial empty graph has the single finite entry
\[
 c_{\varnothing}(\varnothing,\{((),0)\})=0.
\]
The following theorem proves the table invariant for these updates.

\begin{theorem}\label{thm:finite-algorithm}
These rules compute the exact functions \eqref{eq:profile-cost} for every constructed block. With empty final boundary, let $\Pi_0=\{((),0)\}$. Then
\begin{equation}\label{eq:profile-Z}
 Z(G)=c_G(\varnothing,\Pi_0).
\end{equation}
If a construction has $N$ operation nodes and boundary size at most $b$, a direct implementation uses at most
\begin{equation}\label{eq:state-bound}
 2^b\,2^{2\cdot4^b}
\end{equation}
states per node, and $N\,2^{2^{O(b)}}$ arithmetic operations on costs of $O(\log(2+|V|))$ bits.
\end{theorem}
\begin{proof}
We prove the table invariant: after each construction step, every finite entry at $(P,\Pi)$ is the minimum number of selected seeds already in the interior among all seed sets whose boundary restriction is $P$ and whose entire admissible support profile is $\Pi$. Every finite minimum is attained. The empty graph has the single empty support and cost zero, so the invariant holds initially.

When an isolated boundary vertex $v$ is introduced, any new seed set uniquely specifies whether $v$ is a seed. For a fixed old admissible support, its extension either omits $v$ or contains it; the latter is allowed exactly when $v$ is not a seed. These are exactly the two displayed $\operatorname{add}$ rules. No interior equation has changed and no new seed is charged yet. Both the restriction of a new support and every allowed extension are present, so the updated profile is exact in both directions.

When an edge $uv$ is introduced, both endpoints are in the boundary. Interior admissibility is unchanged. Support membership is unchanged, and the number of support neighbours at $u$ increases by one precisely when $u$ is outside and $v$ is inside the support; the symmetric statement holds at $v$. Truncation at two is exact for the later test because counts zero, one and at least two are the only possibilities distinguished by the fort condition. Therefore the stated signature update is exactly the profile of the same seed set in the new graph.

When $v$ is forgotten, all incident edges have already been introduced. If $v$ belongs to a support, its own fort condition is vacuous. If it is outside, that condition holds exactly when its state is not $\Ostate_1$. Filtering by this condition and deleting its coordinate therefore gives exactly the new admissible supports. The nonempty bit remains unchanged, even if no support vertex remains on the boundary. A seed at $v$ becomes an interior seed, and only then is one added to the cost. This proves the forget update.

At a join, the exact witness restriction and gluing in Theorem~\ref{thm:profile-gluing} give the product profile, and Theorem~\ref{thm:cost-gluing} gives precisely the summed minimum cost. If several predecessors give the same $(P,\Pi)$, the future updates
are identical. Each predecessor minimum is attained by a seed set, so the
least of their costs is the minimum for $(P,\Pi)$. This proves the invariant by induction over the construction tree.

At empty final boundary, all fort conditions have been imposed. The empty support always survives. A nonempty fort survives exactly when the remaining profile contains $((),1)$. Therefore
\[
 \Pi_G(S)=\{((),0)\}
 \Longleftrightarrow S\text{ meets every nonempty fort}
 \Longleftrightarrow S\text{ is zero forcing}.
\]
Every seed has been counted once, at its forgetting operation. Minimizing this final entry gives~\eqref{eq:profile-Z}.

For a boundary of size $b$, the four vertex states and the global
nonemptiness bit give at most $2\cdot4^b$ signatures. A profile is a subset
of this set. Including the boundary seed set gives at most
$2^b2^{2\cdot4^b}$ table keys, a doubly exponential bound in $b$. A join can examine every pair of keys and every compatible pair of signatures. Squaring this key bound and multiplying by factors exponential in $b$ still gives $2^{2^{O(b)}}$ operations per node. The other updates are no larger. Each finite cost is an integer between zero and $|V|$, and can be stored in $O(\log(2+|V|))$ bits. This proves the stated bound over $N$ operation nodes.
\end{proof}
Every finite graph can be constructed by introducing all vertices, then all edges, and then forgetting the vertices. A decomposition with smaller boundaries reduces \eqref{eq:state-bound}. Treewidth-parameterized algorithms for zero forcing are studied
in~\cite{Scheffler}.

\subsection{Replacement of a graph block}
The complete cost function can be tested against any exterior block with the same labeled boundary. To finish such a gluing, all remaining boundary fort conditions must be imposed.
For a boundary profile define
\[
 \operatorname{Close}_B(\Pi)=
 \{\eps:(\sigma,\eps)\in\Pi,\ \sigma_b\ne\Ostate_1\text{ for every }b\in B\}.
\]
This enforces all boundary fort conditions simultaneously. The empty support always survives, so this set always contains zero for a realizable profile.

Closing a realizable profile leaves the bit $1$ precisely when a nonempty fort survives. Thus the exterior graph depends on the interior block through $c_H(P,\Pi)$, which gives the following sufficient condition for replacement.

\begin{corollary}\label{cor:context-replacement}
Let $H_1,H_2$ have the same labeled boundary and suppose
\[
 c_{H_1}(P,\Pi)=c_{H_2}(P,\Pi)\quad\text{for every }P,\Pi.
\]
For any exterior block $Q$ with that boundary, disjoint interior, and compatible nonduplicated edge assignment,
\[
 Z(H_1\cup_B Q)=Z(H_2\cup_B Q).
\]
More explicitly, for any pair of blocks $H,Q$,
\begin{equation}\label{eq:closed-glue-cost}
 Z(H\cup_B Q)=
 \min_{\substack{P,\Pi_H,\Pi_Q\\
 \operatorname{Close}_B(\Pi_H\star\Pi_Q)=\{0\}}}
 \bigl(|P|+c_H(P,\Pi_H)+c_Q(P,\Pi_Q)\bigr).
\end{equation}
\end{corollary}
\begin{proof}
For a seed set $S$ in $H\cup_BQ$, put
\[
 P=S\cap B,\qquad
 S_H=S\cap V(H),\qquad S_Q=S\cap V(Q).
\]
Theorem~\ref{thm:profile-gluing} and closure of all boundary coordinates give
\[
 S\text{ is zero forcing}
 \quad\Longleftrightarrow\quad
 \operatorname{Close}_B
   (\Pi_H(S_H)\star\Pi_Q(S_Q))=\{0\}.
\]
The total cost is
\[
 |S|=|P|+|S_H\cap I_H|+|S_Q\cap I_Q|.
\]
Minimizing first over seed choices for each complete profile gives
\eqref{eq:closed-glue-cost}. If $c_{H_1}(P,\Pi)=c_{H_2}(P,\Pi)$ for all $P,\Pi$, the two minima have identical terms and conditions. Hence
$Z(H_1\cup_BQ)=Z(H_2\cup_BQ)$.
\end{proof}
Theorem~\ref{thm:strip} gives the active-boundary conditions for a
triangular strip. The full function $c_H$ in
Corollary~\ref{cor:context-replacement} determines its effect in any
compatible exterior graph.

\section{Applications to triangulations}\label{sec:outerplanar}

A maximal outerplanar graph can be constructed by starting with a triangle and repeatedly attaching a new vertex to the ends of an outer edge. Each attachment adds one vertex, two edges, and one bounded triangular face. Thus a graph obtained from $n-2$ triangles in this manner has $n$ vertices and $2n-3$ edges. The weak dual of its bounded faces is a tree.

We use forcing sequences for upper bounds, forts for lower bounds,
and complete support profiles to compare the triangular strips.

\subsection{A fan of arbitrary order}
Let $F_n$, $n\ge3$, have vertices $0,1,\ldots,n-1$ and edges
\begin{equation}\label{eq:fan-edges}
 E(F_n)=\{\{0,i\}:1\le i\le n-1\}
 \cup\{\{i,i+1\}:1\le i\le n-2\}.
\end{equation}
Its bounded triangles are $(0,i,i+1)$, $1\le i\le n-2$.

\begin{proposition}\label{prop:outer-fan}
For $n\ge3$,
\[
 P(F_n)=Z(F_n)=2.
\]
For $S=\{0,1\}$ the universal module satisfies $M_S\cong R_{F_n}^{\,2}$.
\end{proposition}
\begin{proof}
Starting from $S=\{0,1\}$, perform
\[
 1\to2,\quad2\to3,\quad\cdots,\quad(n-2)\to(n-1).
\]
Every force is legal since $0$ and the preceding path vertices are black. With
\[
 Q=(1,\ldots,n-2),\qquad U=(2,\ldots,n-1),
\]
the matrix $\bA_{Q,U}$ is triangular with diagonal
\[
 t_{12},\ t_{23},\ \ldots,\ t_{n-2,n-1}.
\]
Consequently,
\[
 \det\bA_{Q,U}=\prod_{i=1}^{n-2}t_{i,i+1}\in R_{F_n}^{\times},
 \qquad M_S\cong R_{F_n}^{\,2}.
\]
Every vertex has degree at least two, hence a single initially black vertex cannot perform a force. Thus $Z(F_n)=2$. The paths $(0)$ and $(1,2,\ldots,n-1)$ form an induced path cover, while the triangle on $\{0,1,2\}$ prevents a single induced path from covering all vertices. Hence $P(F_n)=2$.
\end{proof}

\subsection{A central triangle with three ears}
Let $G_6$ have bounded triangular faces
\begin{equation}\label{eq:sun-faces}
 (0,1,2),\quad(0,1,3),\quad(1,2,4),\quad(2,0,5).
\end{equation}
In particular, $G_6$ is maximal outerplanar, with $6$ vertices and $9$ edges.

\begin{proposition}\label{prop:outer-sun}
For the graph $G_6$,
\[
 P(G_6)=Z(G_6)=3.
\]
\end{proposition}
\begin{proof}
The set $S=\{0,1,4\}$ gives the forces
\[
 4\to2,\qquad1\to3,\qquad0\to5.
\]
Selecting the source rows $Q=(4,1,0)$ and the unknown columns $U=(2,3,5)$ gives
\[
 \bA_{Q,U}=
 \begin{pmatrix}
 t_{24}&0&0\\
 t_{12}&t_{13}&0\\
 t_{02}&t_{03}&t_{05}
 \end{pmatrix},
 \qquad
 \det\bA_{Q,U}=t_{24}t_{13}t_{05}\in R_{G_6}^{\times}.
\]
Hence $Z(G_6)\le3$ and $M_S\cong R_{G_6}^{\,3}$.

Both
\[
 C=\{0,1,2\},\qquad L=\{3,4,5\}
\]
are forts, and $C\cap L=\varnothing$. A two-element zero-forcing set would have to contain exactly one vertex of $C$ and one of $L$. Up to the dihedral symmetry of the central triangle, these pairs are $\{0,3\}$ and $\{0,4\}$. They avoid, respectively, the forts
\[
 \{2,4,5\},\qquad\{1,2,3,5\}.
\]
For the first fort, the exterior-neighborhood cardinalities at $0,1,3$ are $2,2,0$; for the second, those at $0,4$ are $4,2$. Thus both satisfy the fort condition. Proposition~\ref{prop:fort-hit} gives $Z(G_6)\ge3$.

The vertex sequences $(0)$, $(1,3)$, and $(4,2,5)$ form an induced path cover, so $P(G_6)\le3$. Suppose there were a cover by two induced paths. Since $\{0,1,2\}$ is a triangle, one path would contain exactly two central vertices, say $a,b$, and the other would contain the remaining central vertex $c$. The ear attached to $\{a,b\}$ cannot lie on the first induced path, which would otherwise contain a triangle. It must therefore lie on the second path. But that ear is not adjacent to $c$ or to either of the other two ears, contradicting connectivity of the second path. Thus $P(G_6)\ge3$.
\end{proof}

For instance, the fort $F=\{2,4,5\}$ is connected, with $\partial F=\{0,1\}$. Theorem~\ref{thm:fort-codimension} also gives
\[
 \codim_{\Omega_{G_6}}\cD_F=|\partial F|+1=3.
\]
Thus the same witness which obstructs a two-vertex seed set determines a codimension-three support locus.

\subsection{Two graphs with the same weak dual tree}
We next use a central triangle and three triangular strips. For $i=0,1,2$, with indices modulo $3$, put
\begin{equation}\label{eq:outer-arm-labels}
 a_i=i,\qquad b_i=i+1\pmod3,\qquad
 (u_i,v_i,w_i)=(3+3i,4+3i,5+3i).
\end{equation}
Starting with $(0,1,2)$, attach the three faces
\begin{equation}\label{eq:outer-snake-faces}
 (a_i,b_i,u_i),\qquad(a_i,u_i,v_i),\qquad(u_i,v_i,w_i)
\end{equation}
for each $i$. Denote the resulting graph by $G_{\mathrm s}$. Define $G_{\mathrm f}$ by replacing only the final face for $i=0$,
\begin{equation}\label{eq:outer-fan-replace}
 (3,4,5)\longmapsto(0,4,5),
 \qquad
 E(G_{\mathrm f})=
 E(G_{\mathrm s})\setminus\{\{3,5\}\}\cup\{\{0,5\}\}.
\end{equation}
Each graph has $12$ vertices, $21$ edges, and $10$ bounded faces. The weak dual in either case is a central vertex joined to three paths of length three.

\begin{proposition}\label{prop:outer-dual-covers}
The induced path covering numbers of the two graphs satisfy
\[
 P(G_{\mathrm s})=P(G_{\mathrm f})=3.
\]
\end{proposition}
\begin{proof}
The following tuples give the vertices of the paths in their traversal order:
\[
 \begin{aligned}
 \cP_{\mathrm s}&=
 \{(4,0,9),\ (5,3,1,7),\ (8,6,2,10,11)\},\\
 \cP_{\mathrm f}&=
 \{(0,9),\ (5,4,3,1,7),\ (8,6,2,10,11)\}.
 \end{aligned}
\]
Directly from \eqref{eq:outer-snake-faces}--\eqref{eq:outer-fan-replace}, each tuple is an induced path, and each family partitions the vertex set. This proves the upper bounds.

Let $C=\{0,1,2\}$. The graph obtained by deleting $C$ has three connected components
\[
 H_i=\{u_i,v_i,w_i\},\qquad i=0,1,2.
\]
An induced path can meet at most two of these components: it contains at most two vertices of the clique $C$; if it contains two, they must occur consecutively, so there can be at most two segments outside $C$.

In $G_{\mathrm s}$ every $G_{\mathrm s}[H_i]$ is a triangle. If two induced paths covered the graph, each of them would have to meet all three $H_i$, since no induced path contains all three vertices of a triangle. This is impossible. In $G_{\mathrm f}$, the components $H_1,H_2$ remain triangles. Both paths of a purported two-path cover must meet $H_1$ and $H_2$, while at least one of them must also meet $H_0$. Again one path would meet three components. Thus both covering numbers are at least three.
\end{proof}

\begin{theorem}\label{thm:outer-dual-Z}
The zero-forcing numbers of these graphs are
\[
 Z(G_{\mathrm s})=4,\qquad Z(G_{\mathrm f})=3.
\]
In particular, the zero-forcing number is not determined by the weak dual tree of a maximal outerplanar graph.
\end{theorem}
\begin{proof}
For $G_{\mathrm s}$, take $S_{\mathrm s}=\{0,4,6,8\}$. The following forces are legal in the displayed order:
\begin{equation}\label{eq:outer-snake-forces}
 \begin{gathered}
 8\to7,\quad 7\to1,\quad6\to2,\quad1\to3,\\
 0\to9,\quad2\to10,\quad3\to5,\quad9\to11.
 \end{gathered}
\end{equation}
They provide an eight-by-eight triangular unit minor whose determinant, up to sign, is
\[
 t_{87}t_{71}t_{62}t_{13}t_{09}t_{2,10}t_{35}t_{9,11}.
\]
Hence $M_{S_{\mathrm s}}\cong R_{G_{\mathrm s}}^{\,4}$ and $Z(G_{\mathrm s})\le4$.

To obtain the opposite inequality, consider
\begin{equation}\label{eq:outer-snake-forts}
 \begin{aligned}
 F_1&=\{6,7,10,11\},\\
 F_2&=\{2,6,8,9,10\},\\
 F_3&=\{2,7,8,9,11\}.
 \end{aligned}
\end{equation}
These are forts. Their exterior boundaries and corresponding numbers of neighbors in the fort are
\[
 \begin{array}{c|c|c}
 &\partial F_j&(|N(x)\cap F_j|)_{x\in\partial F_j}\\\hline
 F_1&(1,2,8,9)&(2,2,2,2)\\
 F_2&(0,1,7,11)&(2,2,2,2)\\
 F_3&(0,1,6,10)&(2,2,3,3).
 \end{array}
\]
They have the further properties
\begin{equation}\label{eq:outer-snake-forts-avoid}
 F_j\cap\{0,3,4,5\}=\varnothing\quad(j=1,2,3),
 \qquad F_1\cap F_2\cap F_3=\varnothing.
\end{equation}
Suppose that a set $S$ of three initial vertices forces all of $G_{\mathrm s}$. Its first force must start at one of $v_i,w_i$: all the other vertices have degree at least four, whereas three initial black vertices cannot activate such a source. If the source is $w_i$ of degree two, at least two seeds belong to $\{a_i,u_i,v_i,w_i\}$; if the source is $v_i$ of degree three, all three seeds belong to that set. Rotate \eqref{eq:outer-snake-forts} by $i$ arms. The three rotated forts avoid $\{a_i,u_i,v_i,w_i\}$ and have empty triple intersection. The at most one remaining seed cannot meet all three forts. One fort is therefore disjoint from $S$, contradicting Proposition~\ref{prop:fort-hit}. This proves $Z(G_{\mathrm s})\ge4$.

For $G_{\mathrm f}$, set $S_{\mathrm f}=\{0,6,8\}$. The legal sequence
\begin{equation}\label{eq:outer-fan-forces}
 \begin{gathered}
 8\to7,\quad7\to1,\quad6\to2,\quad1\to3,\quad3\to4,\\
 4\to5,\quad0\to9,\quad2\to10,\quad9\to11
 \end{gathered}
\end{equation}
shows $Z(G_{\mathrm f})\le3$ and $M_{S_{\mathrm f}}\cong R_{G_{\mathrm f}}^{\,3}$.

The only degree-two vertices of $G_{\mathrm f}$ are $5,8,11$. For these vertices define, respectively,
\begin{equation}\label{eq:outer-fan-forts}
 E_5=\{6,7,10,11\},\quad
 E_8=\{0,3,4,9,10\},\quad
 E_{11}=\{3,5,7,8\}.
\end{equation}
They are forts, as their exterior-neighbor cardinalities are
\[
 \begin{array}{c|c|c}
 &\partial E_w&(|N(x)\cap E_w|)_{x\in\partial E_w}\\\hline
 E_5&(1,2,8,9)&(2,2,2,2)\\
 E_8&(1,2,5,11)&(2,3,2,2)\\
 E_{11}&(0,1,4,6)&(2,2,2,2).
 \end{array}
\]
Moreover,
\[
 E_w\cap N[w]=\varnothing\qquad(w=5,8,11).
\]
If two initial vertices could force the graph, the first source would have to have degree two; hence the initial set would be $\{w,x\}$ for some $w\in\{5,8,11\}$ and $x\in N(w)$. But this set misses $E_w$, which is impossible for a zero-forcing set. Therefore $Z(G_{\mathrm f})\ge3$.
\end{proof}

The preceding proof also determines the gap with the induced path covering number:
\[
 Z(G_{\mathrm s})-P(G_{\mathrm s})=1,
 \qquad
 Z(G_{\mathrm f})-P(G_{\mathrm f})=0.
\]
The same weak dual tree therefore supports both an acyclic realization of a minimum induced path cover and a configuration for which every such cover has a causal obstruction.

\subsection{The two-port states}
The change of the last triangle can be detected by the support profiles of the individual strips. Let the boundary be $B=\{a,b\}$, with both boundary vertices initially black, and let $I=\{u,v,w\}$ have no seeds. Take
\[
 H_{\mathrm{fan}}:(a,b,u),(a,u,v),(a,v,w),
 \qquad
 H_{\mathrm{serp}}:(a,b,u),(a,u,v),(u,v,w).
\]
Here the notation specifies the bounded triangular faces. For $F\subseteq I$, admissibility means
\[
 |N_H(x)\cap F|\ne1\qquad(x\in I\setminus F).
\]

\begin{proposition}\label{prop:outer-profiles}
The complete profiles of the two strips at $B$ are
\begin{equation}\label{eq:outer-fan-profile}
 \Pi_{\mathrm{fan}}=
 \{(\Ostate_0,\Ostate_0;0),\ (\Ostate_2,\Ostate_1;1)\},
\end{equation}
\begin{equation}\label{eq:outer-serp-profile}
 \Pi_{\mathrm{serp}}=
 \{(\Ostate_0,\Ostate_0;0),\ (\Ostate_1,\Ostate_0;1),
 (\Ostate_1,\Ostate_1;1),\ (\Ostate_2,\Ostate_1;1)\}.
\end{equation}
\end{proposition}
\begin{proof}
For the fan, $H_{\mathrm{fan}}[I]$ is the path $u-v-w$. Its admissible nonempty supports are
\[
 F=\{u,w\},\quad \{u,v,w\}.
\]
Both give $|N(a)\cap F|\ge2$ and $|N(b)\cap F|=1$.

For the serpent, $H_{\mathrm{serp}}[I]=K_3$. The admissible nonempty supports are precisely the three two-element subsets of $I$ and $I$ itself. Their boundary states are
\[
 \begin{array}{c|cccc}
 F&\{v,w\}&\{u,w\}&\{u,v\}&I\\\hline
 (\sigma_a,\sigma_b)
 &(\Ostate_1,\Ostate_0)&(\Ostate_1,\Ostate_1)
 &(\Ostate_2,\Ostate_1)&(\Ostate_2,\Ostate_1).
 \end{array}
\]
Adding the empty support in both cases proves the assertions.
\end{proof}

In particular, if $b$ is an active boundary vertex and $a$ is inactive, the condition $\sigma_b\ne\Ostate_1$ removes every nonempty fan support but retains the serpent support $\{v,w\}$ with $\sigma_b=\Ostate_0$. This gives the relative forcing distinction of Theorem~\ref{thm:strip}.

\subsection{Exact finite values}
For the displayed vertex order, the complete-profile recursion of Appendix~\ref{app:profile-algorithm} and exhaustive fort transversals give the same values:
\[
\begin{array}{c|rrrrrr}
 \text{graph}&|V|&|E|&P(G)&Z(G)&\#\text{ forts}&\text{max. table entries}\\\hline
 F_8&8&13&2&2&70&38\\
 G_6&6&9&3&3&21&42\\
 G_{\mathrm s}&12&21&3&4&752&296\\
 G_{\mathrm f}&12&21&3&3&690&198
\end{array}
\]
The displayed forcing sequences, fort obstructions, and path covers certify
$Z(G)$ and $P(G)$. The fort counts come from exhaustive support enumeration.
The last column uses the schedule in
Appendix~\ref{app:profile-algorithm}. Vertices and earlier neighbors are
introduced in increasing order, and the earliest eligible active vertex
is forgotten first. The table size is measured after each vertex and its
edges to earlier vertices are introduced, before forgetting.

\appendix
\section{Computation by complete support profiles}\label{app:profile-algorithm}

We specify the input encoding and schedule used for the values in
Section~\ref{sec:outerplanar}. The update rules and their invariant are proved
in Theorem~\ref{thm:finite-algorithm}.

\subsection{Encoding and schedule}
The input is a finite simple graph with a fixed vertex order
$v_1,\ldots,v_n$. An edge is stored as an unordered pair, with no duplicate
pairs. On the ordered active boundary $B$, a signature is encoded as
$(\sigma,\eps)$, where
\[
 \Ostate_0\mapsto0,\quad \Ostate_1\mapsto1,\quad
 \Ostate_2\mapsto2,\quad \Istate\mapsto3,\qquad \eps\in\{0,1\}.
\]
The profile is an immutable set of these tuples. A table key consists of
the boundary seed set $P$ and the complete profile $\Pi$; its value is
\begin{equation}\label{eq:algo-table}
 T_B(P,\Pi)=
 \min\{|S\setminus B|:S\cap B=P,\ \Pi_H(S)=\Pi\},
\end{equation}
where $H$ is the graph built so far. Missing keys have cost $+\infty$.
Start with
\[
 B=P=\varnothing,\qquad \Pi_0=\{((),0)\},\qquad
 T_{\varnothing}(\varnothing,\Pi_0)=0.
\]
For each $v_i$, perform the following operations from
Subsection~\ref{sec:state-operations}.
\begin{enumerate}
\item Append $v_i$ to the active boundary and use the two introduce-vertex
updates, according as $v_i$ is chosen as a seed or not.
\item Introduce every edge $v_jv_i$ with $j<i$, in increasing order of $j$.
\item While an active vertex has no neighbor among
$\{v_{i+1},\ldots,v_n\}$, forget the earliest such vertex in the active order.
\end{enumerate}
After each operation, merge equal table keys by taking the least cost.
A minimizing predecessor and the seed choices can be stored to recover a
seed set. Each vertex is charged once, when it is forgotten. After the
last iteration the boundary is empty and the output is
\begin{equation}\label{eq:algo-output}
 Z(G)=T_{\varnothing}(\varnothing,\{((),0)\}).
\end{equation}
A forcing chronology for the recovered seed set is obtained by the legal
closure algorithm of Lemma~\ref{lem:closure}.

\begin{proposition}\label{prop:algorithm-correctness}
This schedule computes $Z(G)$ for every finite simple graph $G$.
\end{proposition}
\begin{proof}
The schedule consists of operations in Theorem~\ref{thm:finite-algorithm}.
Every edge is introduced when its second endpoint appears, and a vertex is
forgotten only after all its incident edges have appeared. Hence that
theorem's invariant applies at every step. At the empty boundary, the
profile containing only the empty support has minimum cost $Z(G)$, which
is~\eqref{eq:algo-output}.
\end{proof}

For the examples, the vertex order is $0,1,\ldots,n-1$. The implementation
samples the table after step (b), before step (c). The largest sampled
boundary sizes for $F_8,G_6,G_{\mathrm s},G_{\mathrm f}$ are
$3,4,6,5$, respectively. Their table sizes are the statistics reported in
Section~\ref{sec:outerplanar}. The bound for an arbitrary boundary is
\eqref{eq:state-bound}. A tree decomposition of a maximal outerplanar graph
can instead keep bags of size at most three; its joins use
Theorem~\ref{thm:cost-gluing}.

\subsection{Induced path covers}
The induced path covering number has a subset recurrence. Write $\mathfrak P(G)$ for the nonempty vertex subsets $W\subseteq V$ such that $G[W]$ is an induced path, including singletons. Fix any total order on $V$ and put $m(M)=\min M$ for $M\ne\varnothing$. Define
\begin{equation}\label{eq:algo-pathcover}
 p(\varnothing)=0,\qquad
 p(M)=1+\min_{\substack{W\in\mathfrak P(G)\\m(M)\in W\subseteq M}}p(M\setminus W).
\end{equation}
Every path cover of $G[M]$ contains a unique path through $m(M)$.
Removing that path gives a cover of $M\setminus W$, and adjoining it to
a minimizing cover of $M\setminus W$ gives the reverse construction.
Induction on $|M|$ proves the recurrence and $p(V)=P(G)$.
To recover the paths, retain a minimizing $W$ at each step. For a
nonsingleton $G[W]$, start at an endpoint and successively follow its
unused neighbor.

\section{Two matrix calculations}\label{sec:verification}

\subsection{The cube coefficient matrix}
The transport product in Example~\ref{ex:cube} is $7/15$.
The vanishing of the coefficient kernel can also be read from a maximal minor.
After eliminating the diagonals, order the cube edges lexicographically and
put the two equations
\[
 \sum_{v\sim u}(v-u)t_{uv}=0,\qquad
 \sum_{v\sim u}(v-u)^2t_{uv}=0
\]
in that order for $u=0,\ldots,7$.
For the resulting $16\times12$ matrix $B$, take the zero-indexed row set
\[
 Q=(0,1,2,3,4,5,6,8,9,10,11,12).
\]
Integer elimination gives
\[
 \det B_{Q,*}=622080\ne0.
\]
Thus the edge-coefficient kernel is zero. The diagonal equations then
give $H^0(G;\sC_W)=0$.

\subsection{Coefficients over the dual numbers}
Let $R=\QQ[\eps]/(\eps^2)$. Write an $R$-matrix as
$A=A_0+\eps A_1$ with rational matrices $A_0,A_1$.
Expansion and $\eps^2=0$ give the operations in
Table~\ref{tab:verification-entrypoints}.

\begin{table}[ht]
\caption{Matrix operations over $\QQ[\eps]/(\eps^2)$.}
\label{tab:verification-entrypoints}
\centering\small
\begin{tabular}{@{}ll@{}}
\toprule
Operation & Coefficient matrices \\
\midrule
Addition & $(A_0,A_1)+(B_0,B_1)=(A_0+B_0,A_1+B_1)$ \\
Multiplication & $(A_0,A_1)(B_0,B_1)=(A_0B_0,A_0B_1+A_1B_0)$ \\
Inverse, $B_0$ invertible & $(B_0,B_1)^{-1}=(B_0^{-1},-B_0^{-1}B_1B_0^{-1})$ \\
\bottomrule
\end{tabular}
\end{table}

For example,
\[
\begin{aligned}
 (B_0+\eps B_1)
 (B_0^{-1}-\eps B_0^{-1}B_1B_0^{-1})
 &=I+\eps(B_1B_0^{-1}-B_1B_0^{-1})\\
 &=I.
\end{aligned}
\]
If $d=d_0+\eps d_1$ and $h=h_0+\eps h_1$, the chain-homotopy identity
$dh+hd=I-\iota\pi$ is equivalent to
\[
 d_0h_0+h_0d_0=I-\iota_0\pi_0,
\]
\[
 d_0h_1+d_1h_0+h_0d_1+h_1d_0
 =-\iota_0\pi_1-\iota_1\pi_0.
\]
These are the constant and first-order parts of the contractions in
Theorems~\ref{thm:contraction} and~\ref{thm:port-gluing}.

\section*{Acknowledgments}
We began studying zero forcing in 2021, during our second year as
undergraduates, through a combinatorial argument for maximal biconnected
outerplanar graphs. We later returned to that argument, recast it in terms
of sheaf cohomology, and pursued its extension to general finite graphs.
This research grew out of our own mathematical curiosity and our interest
in exploring these questions.

We used AI tools to assist with mathematical derivations, computations,
and preliminary drafts. We are now checking the arguments and rewriting
both manuscripts ourselves. The versions submitted for publication will
be the manuscripts rewritten by the authors.

\end{document}